\documentclass[12pt]{amsart}
\usepackage[margin=1.0in]{geometry} 
\usepackage{esint,amsthm,amsmath,amssymb,amsfonts,graphicx,color,comment,enumerate,psfrag,caption,bbm,bm,hyperref,enumitem}
\usepackage{mathtools}

\usepackage{tikz,pgf}
\usepackage{pgfplots}
\usetikzlibrary{calc}
\usetikzlibrary{patterns}

\usepackage{cleveref}
\allowdisplaybreaks

\DeclareMathOperator\supp{supp}

\newtheorem{lemma}{Lemma}[section]

\numberwithin{equation}{section}
\newtheorem{theorem}{Theorem}[section]
\newtheorem{proposition}[theorem]{Proposition}

\begin{document}
\title[Bilinear Bochner-Riesz Means on M\'etivier groups]{Bilinear Bochner-Riesz Means on M\'etivier groups}

\author[S. Bagchi, Md N. Molla, J. Singh]
{Sayan Bagchi \and Md Nurul Molla \and Joydwip Singh} 

\address[S. Bagchi]{Department of Mathematics and Statistics, Indian Institute of Science Education and Research Kolkata, Mohanpur--741246, West Bengal, India.}
\email{sayan.bagchi@iiserkol.ac.in}

\address[Md N. Molla]{Department of Mathematics and Statistics, Indian Institute of Science Education and Research Kolkata, Mohanpur--741246, West Bengal, India; Department of General Sciences, BITS Pilani Dubai Campus, International Academic City, Dubai, 345055, UAE}
\email{nurul.pdf@iiserkol.ac.in \& nurul@dubai.bits-pilani.ac.in}

\address[J. Singh]{Department of Mathematics and Statistics, Indian Institute of Science Education and Research Kolkata, Mohanpur--741246, West Bengal, India.}
\email{js20rs078@iiserkol.ac.in}



\begin{abstract}
In this paper, we study the $L^{p_1}(G) \times L^{p_2}(G)$ to $L^{p}(G)$ boundedness of the bilinear Bochner-Riesz means associated with the sub-Laplacian on M\'etivier group $G$ under the H\"older's relation $1/p = 1/p_1 + 1/p_2$, $1\leq p_1, p_2 \leq \infty$. Our objective is to obtain boundedness results, analogous to the Euclidean setting, where the Euclidean dimension in the smoothness threshold is possibly replaced by the topological dimension of the underlying M\'etivier group $G$.
\end{abstract}

\maketitle

\section{Introduction}
\subsection{Bochner-Riesz on Euclidean spaces}
A central theme in harmonic analysis is understanding the convergence of Fourier series and integrals in Lebesgue spaces. The Bochner-Riesz mean plays a crucial role in this context, as it offers an approach to validating the Fourier inversion formula in the $L^p$ setting. For $R>0$, the Bochner-Riesz operator, denoted by $S^{\alpha}_R$ in $\mathbb{R}^n$ and of order $\alpha \geq 0$, is the Fourier multiplier operator defined by
\begin{align*}
    S^{\alpha}_R(f)(x) = \int_{\mathbb{R}^n} \left(1-\frac{|\xi|^2}{R^2}\right)_{+}^{\alpha} \widehat{f}(\xi)\ e^{2 \pi i x \cdot \xi} \ d\xi,
\end{align*}
where $(r)_{+} = \max\{r, 0\}$ for $r \in \mathbb{R}$, $f \in \mathcal{S}(\mathbb{R}^n)$, the space of all Schwartz class functions in $\mathbb{R}^n$. The famous Bochner-Riesz conjecture concerns finding the optimal range of the parameter $\alpha \geq 0$, for which the operator $S^{\alpha}_R$ are bounded in $L^p$-spaces. For $1\leq p \leq \infty$ and $p\neq 2$, it has been conjectured that the Bochner-Riesz means $S^{\alpha}_R$ is bounded on $L^p(\mathbb{R}^n)$ if and only if $\alpha> \alpha(p) := \max\left\{n\left|\frac{1}{p}-\frac{1}{2}\right|-\frac{1}{2}, 0 \right\}$. In 1972, Carleson and Sj\"olin \cite{Carleson_Sjolin_Multiplier_Problem_on_Disc_1972} proved that the conjecture is indeed true when $n=2$. Despite extensive research on the Bochner-Riesz problem, only partial results are known to be true, and in general it remains open for $n\geq 3$. For historical background and recent progress on the Bochner-Riesz conjecture, see \cite{Tao_Recent_Progress_Restriction_conjecture_2004}, \cite{Fefferman_Strongly_Singular_Convolution_Operator_1970}, \cite{Bourgain_Guth_Oscillatory_Integral_2011}, \cite{Lee_Improved_Bounds_Bochner_Riesz_Maximal_2004}, \cite{Tao_Vargas_Vega_Bilinear_Restriction_Kakeya_1998}, \cite{Tao_Vargas_Bilinear_approach_2000} and references therein.

One can also consider a bilinear generalization of the Bochner-Riesz operator, called the bilinear Bochner-Riesz operator. As in the linear setting, it is related to the convergence of the product of two $n$-dimensional Fourier series; see \cite{Bernicot_Grafakos_Song_Yan_Bilinear_Bochner_Riesz_2015} for more details. For $R>0$, the bilinear Bochner-Riesz operator $B^{\alpha}_R$ in $\mathbb{R}^n$, of order $\alpha \geq 0$ is defined by
\begin{align*}
    B^{\alpha}_R(f,g)(x) = \int_{\mathbb{R}^n} \int_{\mathbb{R}^n} \left(1-\frac{|\xi|^2+|\eta|^2}{R^2} \right)_{+}^{\alpha} \widehat{f}(\xi)\ \widehat{g}(\eta)\ e^{2 \pi i x \cdot (\xi + \eta)} \ d\xi \  d\eta,
\end{align*}
where $f, g \in \mathcal{S}(\mathbb{R}^n)$ and $\widehat{f}, \widehat{g}$ are their Fourier transforms. As the bilinear Bochner-Riesz operator is the obvious bilinear generalization of the linear Bochner-Riesz operator, it is therefore natural, just as in the linear case, to ask for the optimal range of the parameter $\alpha$, such that the corresponding bilinear Bochner-Riesz operator $B^{\alpha}_R$ is bounded from $L^{p_1}(\mathbb{R}^n) \times L^{p_2}(\mathbb{R}^n)$ to $L^{p}(\mathbb{R}^n)$ under the condition $1\leq p_1, p_2 \leq \infty$ and $1/p=1/p_1 + 1/p_2$.  This condition is often referred to as $(p_1, p_2,p)$ satisfies the H\"older's relation. Recently, several authors have investigated the convergence of $B^{\alpha}_R$ under this condition; see \cite{Bernicot_Germain_Bilinear_multilier_narrow_support_2013}, \cite{Bernicot_Grafakos_Song_Yan_Bilinear_Bochner_Riesz_2015}, \cite{Diestel_Grafakos_Ball_multiplier_problem_2007}, \cite{Jeong_Lee_Vargas_Bilinear_Bochner_Riesz_2018} and \cite{Liu_Wang_Bilinear_Bochner_Riesz_Non_Banach_2020}. For $n=1$, the problem has been nearly completely solved for the Banach triangle case, that is, when all $p_1, p_2, p \in [1, \infty]$ and $1/p=1/p_1 + 1/p_2$; see \cite[Theorem 4.1]{Bernicot_Grafakos_Song_Yan_Bilinear_Bochner_Riesz_2015} and \cite{Grafakos_Li_Disc_Multiplier_2006}. For the non-Banach range ($p <1$), some progress has been made, notably in \cite[Theorem 2.2]{Jotsaroop_Shrivastava_Maximal_Bochner_Riesz_2022}. When $n\geq 2$ and $\alpha>0$, Bernicot et al. addressed this problem in \cite{Bernicot_Grafakos_Song_Yan_Bilinear_Bochner_Riesz_2015}, establishing both positive and negative results for the bilinear Bochner-Riesz operator under the H\"older's relation. Following the work of \cite{Bernicot_Grafakos_Song_Yan_Bilinear_Bochner_Riesz_2015}, it was subsequently improved in two different regimes. In \cite{Jeong_Lee_Vargas_Bilinear_Bochner_Riesz_2018}, Jeong, Lee and Vargas studied the bilinear Bochner-Riesz problem. By introducing a new decomposition, they related the estimate of bilinear Bochner-Riesz operator to the product of square function estimate of the linear Bochner-Riesz operator, and from that they were able to improve the results of \cite{Bernicot_Grafakos_Song_Yan_Bilinear_Bochner_Riesz_2015} in certain ranges for the Banach triangle case, that is, when $2 \leq p_1, p_2 \leq \infty$ and $p \geq 1$. On the other hand, when $0<p < 1$,  Liu and Wang \cite{Liu_Wang_Bilinear_Bochner_Riesz_Non_Banach_2020} further improved the results of \cite{Bernicot_Grafakos_Song_Yan_Bilinear_Bochner_Riesz_2015} by obtaining a lower smoothness threshold $\alpha$. Specifically, they improved the range of $\alpha$ at the point $(p_1, p_2, p)=(1,2,2/3)$ and by symmetry at $(2,1,2/3)$. In fact, \cite{Bernicot_Grafakos_Song_Yan_Bilinear_Bochner_Riesz_2015} and \cite{Liu_Wang_Bilinear_Bochner_Riesz_Non_Banach_2020} obtained the following result.

\begin{theorem}\cite[Proposition 4.10, 4.11]{Bernicot_Grafakos_Song_Yan_Bilinear_Bochner_Riesz_2015}, \cite[Theorem 1.1]{Liu_Wang_Bilinear_Bochner_Riesz_Non_Banach_2020}
\label{Theorem: Euclidean bilinear Bochner-Riesz for Grafakos}
Let $n \geq 2$ and $1\leq p_1, p_2 \leq \infty$ with $1/p = 1/p_1 +1/p_2$. Then $B^{\alpha}_R$ is bounded from $L^{p_1}(\mathbb{R}^n) \times L^{p_2}(\mathbb{R}^n)$ to $L^{p}(\mathbb{R}^n)$ if $p_1, p_2, p$ and $\alpha$ satisfy one of the following conditions:
\begin{enumerate}
    \item (Region I) $2 \leq p_1, p_2 \leq \infty$, $1\leq p \leq 2$ and $\alpha> (n-1)(1-\frac{1}{p})$.
    \item (Region II) $2 \leq p_1, p_2, p \leq \infty$ and $\alpha> \frac{n-1}{2} + n(\frac{1}{2}-\frac{1}{p})$.
    \item (Region III) $2 \leq p_2 < \infty$, $1\leq p_1, p \leq 2$ and $\alpha> n(\frac{1}{2}-\frac{1}{p_2})-(1-\frac{1}{p})$.
    \item (Region III) $2 \leq p_1 < \infty$, $1\leq p_2, p \leq 2$ and $\alpha> n(\frac{1}{2}-\frac{1}{p_1})-(1-\frac{1}{p})$.
    \item (Region IV) $1\leq p_1 \leq 2 \leq p_2 \leq \infty$, $0<p<1$ and $\alpha> n(\frac{1}{p_1}-\frac{1}{2})$.
    \item (Region IV) $1\leq p_2 \leq 2 \leq p_1 \leq \infty$, $0<p<1$ and $\alpha> n(\frac{1}{p_2}-\frac{1}{2})$.
    \item (Region V) $1\leq p_1 \leq p_2 \leq 2$ and $\alpha> n(\frac{1}{p}-1)-(\frac{1}{p_2}-\frac{1}{2})$.
    \item (Region V) $1\leq p_2 \leq p_1 \leq 2$ and $\alpha> n(\frac{1}{p}-1)-(\frac{1}{p_1}-\frac{1}{2})$. 
\end{enumerate}
\end{theorem}

\begin{figure}[!ht]
\begin{centering}
\definecolor{qqqqff}{rgb}{0,0,1}
\begin{tikzpicture}[line cap=round,line join=round,x=0.8cm,y=0.8cm]
\draw (6,6)-- (0,6);
\draw (6,0)-- (6,6);
\draw (0,3)-- (6,3);
\draw [->] (0,0) -- (8,0);
\draw [->] (0,0) -- (0,7.5);
\draw (3,6)-- (3,3);
\draw (0,6)-- (6,0);
\draw (3,3)-- (3,0);
\draw (3,3)-- (6,6);
\draw (0,3)-- (3,0);
\begin{scriptsize}
\fill [color=qqqqff] (0,0) circle (1.5pt);
\draw[color=qqqqff] (1,-.5) node {$\alpha>n-\frac{1}{2}$};
\fill [color=qqqqff] (0,6) circle (1.5pt);
\draw[color=qqqqff] (0.65,6.5) node {$\alpha > \frac{n}{2}$};
\fill [color=qqqqff] (6,6) circle (1.5pt);
\draw[color=qqqqff] (6.19,6.5) node {$\alpha>n-\frac{1}{2}$};
\fill [color=qqqqff] (6,0) circle (1.5pt);
\draw[color=qqqqff] (6.7,0.4) node {$\alpha>\frac{n}{2}$};
\fill [color=qqqqff] (3,0) circle (1.5pt);
\draw[color=qqqqff] (4.2,-0.5) node {$\alpha>\frac{n-1}{2}$};
\fill [color=qqqqff] (3,6) circle (1.5pt);
\draw[color=qqqqff] (3.3,6.5) node {$\alpha>\frac{n}{2}$};
\fill [color=qqqqff] (0,3) circle (1.5pt);
\draw[color=qqqqff] (-1,3.6) node {$\alpha>\frac{n-1}{2}$};
\fill [color=qqqqff] (3,3) circle (1.5pt);
\draw[color=qqqqff] (4,3.33) node {$\alpha>0$};
\fill [color=qqqqff] (6,3) circle (1.5pt);
\draw[color=qqqqff] (6.7,2.8) node {$\alpha>\frac{n}{2}$};
\fill [color=qqqqff] (6,-0.5) circle (0pt);
\draw[color=qqqqff] (6.0,-0.38) node {$1$};
\fill [color=qqqqff] (3,-0.5) circle (0pt);
\draw[color=qqqqff] (2.9,-0.48) node {$\frac{1}{2}$};
\fill [color=qqqqff] (-0.25,-0.42) circle (0pt);
\draw[color=qqqqff] (-0.25,-0.28) node {$O$};
\fill [color=qqqqff] (-0.5,3) circle (0pt);
\draw[color=qqqqff] (-0.38,3.0) node {$\frac{1}{2}$};
\fill [color=qqqqff] (-0.5,6) circle (0pt);
\draw[color=qqqqff] (-0.38,6.0) node {$1$};
\fill [color=qqqqff] (8.5,-0.5) circle (0pt);
\draw[color=qqqqff] (8.5,-0.16) node {$\frac{1}{p_1}$};
\fill [color=qqqqff] (-0.5,8.5) circle (0pt);
\draw[color=qqqqff] (-0.58,7.5) node {$\frac{1}{p_2}$};
\draw[color=qqqqff] (2.3,2.1) node {$I$};
\draw[color=qqqqff] (1.0,1.0) node {$II$};
\draw[color=qqqqff] (4.1,1.0) node {$III$};
\draw[color=qqqqff] (1.0,4.1) node {$III$};
\draw[color=qqqqff] (4.7,2.1) node {$IV$};
\draw[color=qqqqff] (2.3,4.7) node {$IV$};
\draw[color=qqqqff] (3.7,4.7) node {$V$};
\draw[color=qqqqff] (5,4) node {$V$};
\end{scriptsize}
\end{tikzpicture}
        \caption{Here $O=(0,0)$, and $\alpha>\alpha(p_1, p_2)$ represents that $B^{\alpha}_R$ is bounded on $L^{p_1}(\mathbb{R}^n) \times L^{p_2}(\mathbb{R}^n) \to L^p(\mathbb{R}^n)$ for $\alpha>\alpha(p_1, p_2)$ (see Theorem \ref{Theorem: Euclidean bilinear Bochner-Riesz for Grafakos}).}
\end{centering}        
\end{figure}

\subsection{Bochner-Riesz beyond Euclidean spaces}
Considerable attention has been paid to the boundedness of Bochner-Riesz means and more generally for multipliers in non-Euclidean frameworks as well. For the boundedness of Bochner-Riesz means related to the Hermite operator, see \cite{Thangavelu_Summability_hermite_Expansion_1989} and \cite{Karadzhov_Riesz_Summability_Hermite_1994}. For the sub-Laplacians on the Heisenberg groups, one can refer to \cite{Thangavelu_Riesz_Means_Heisenberh_Group_1990}, for the sharper result with mixed norm, see \cite{Muller_Riesz_Means_Kohn_Laplacian_1989}, and for multiplier related result, see \cite{Muller_Stein_Spectral_Multiplier_Heisenberg_Related_groups_1994}, \cite{Hebisch_Spectral_Multiplier_Heisenberg_1993}, \cite{Mauceri_weyl_transform_1980}, \cite{Lin_Mult_Heisenberg_group_1995} and \cite{Bagchi_Fourier_Mult_Heisenberg_2021}. 

Beyond the Heisenberg group, extensive research has also been conducted. Let $L$ be the sub-Laplacian on any stratified Lie group $G$ with homogeneous dimension $Q$. In 1991, Christ \cite{Christ_Spectral_multiplier_Nilpotent_groups_1991} and independently Mauceri and Meda \cite{Mauceri_Meda_Multipliers_Stratified_groups_1990}, established the $L^p$-boundedness of spectral multiplier for $L$ under Mihlin-H\"ormander type condition with order of differentiability $s>Q/2$. In particular, these results imply that the Bochner-Riesz means $(1-tL)_{+}^{\alpha}$ is bounded on $L^p(G)$, $1<p<\infty$, provided $\alpha>(Q-1)/2$. However, for stratified Lie groups with step bigger than one, in general, the homogeneous dimension $Q$ is always strictly bigger than the topological dimension $d$ of $G$. At that time, it was not known whether these result was sharp or not. The first surprise came when, for sub-Laplacian on the Heisenberg (type) groups, M\"uller and Stein \cite{Muller_Stein_Spectral_Multiplier_Heisenberg_Related_groups_1994}, and independently Hebisch \cite{Hebisch_Spectral_Multiplier_Heisenberg_1993}, showed that the above Mihlin-H\"ormander type multiplier result is not sharp. They showed that the previously known threshold  $Q/2$ can be replaced with $d/2$, where $d$ is the topological dimension of Heisenberg (type) groups. In particular, this refinement also improved the result of Bochner-Riesz multiplier by lowering the required smoothness threshold from $(Q-1)/2$ to $(d-1)/2$, and this improvement turned out to be sharp (see \cite{Martini_Muller_spectral_mult_two_step_2016}). Following these discoveries, there has been extensive research on determining the sharp threshold  in Mihlin-H\"ormander type result for various sub-Laplacians in many different settings. Such improvements have been established  for certain classes of two-step stratified Lie groups,  for  instance, Heisenberg-Reiter type groups \cite{Martini_Spectral_Multiplier_Heisenberg_Reiter_2015}, M\'etivier groups, and more generally, for Lie groups of polynomial growth \cite{Martini_Lie_groups_Polynomial_Growth_2012}, as well as two-step stratified groups with lower dimensions \cite{Martini_Muller_New_Class_Two_Step_Stratified_Groups_2014}. These sharp spectral multiplier results also yield sharp Bochner-Riesz multiplier result with the critical index $(d-1)/2$. However, it remains an open question whether the smoothness threshold in Mihlin-H\"ormander type condition $s>d/2$ is sufficient or not for boundedness of spectral multiplier on all two step stratified Lie groups (see \cite{Martini_Muller_spectral_mult_two_step_2016}). For related results in different settings, one can consult \cite{Martini_Sikora_Grushin_Weighted_Plancherel_2012}, \cite{Martini_Muller_Sharp_Multiplier_Grushin_2014}, \cite{Ahrens_Cowling_Martini_Muller_Quaternionic_Sphere_2020}, \cite{Casarino_Ciatti_Martini_Grushin_Sphere_2019}, \cite{Cowling_Klima_Sikora_Kohn_Laplacian_On_Sphere_2011} and references therein. A somewhat different problem concerning the $p$-specific boundedness of Bochner-Riesz means, that is, boundedness of Bochner-Riesz operator for $0<\alpha\leq (d-1)/2$, in the context of Heisenberg type groups, or more generally on M\'etivier groups has recently been studied by Niedorf in \cite{Niedorf_p_specific_Heisenber_group_2024}, \cite{Niedorf_Metivier_group_2023}.

There are also studies about the boundedness of bilinear Bochner-Riesz means beyond Euclidean spaces, such as for sub-Laplacians on the Heisenberg group \cite{Liu_Wang_Reasz_Means_Heisenberg_Group}, Heisenberg-type groups \cite{Wang_Wang_Riesz_means_Htype_groups_2024}. It is worth noting that in all the results by \cite{Liu_Wang_Reasz_Means_Heisenberg_Group} and \cite{Wang_Wang_Riesz_means_Htype_groups_2024}, the smoothness threshold $\alpha(p_1, p_2)$ is expressed in terms of the homogeneous dimension $Q$ of the underlying space. However, as observed earlier, in the linear setting for stratified Lie groups with step greater than one, the boundedness of spectral multipliers or Bochner-Riesz multipliers with smoothness threshold expressed in terms of the homogeneous dimension are generally not sharp. In certain cases, the smoothness threshold can be further refined and expressed in terms of the topological dimension $d$. This suggests that, analogous to the linear theory, one may expect the boundedness of the bilinear Bochner–Riesz operator to also hold with the smoothness threshold $\alpha(p_1, p_2)$ expressed in terms of $d$ rather than $Q$. 

Motivated by this perspective, our goal in this paper is to establish the boundedness of bilinear Bochner-Riesz operator associated with the sub-Laplacians on the M\'etivier groups $G$,  a class that strictly contains the Heisenberg type groups (see \cite{Muller_Seeger_Nilpotent_Lie_group_2004}). Furthermore, we aim to express the smoothness threshold $\alpha(p_1, p_2)$ in terms of the topological dimension $d$ of $G$. Our result applies to both Banach and non-Banach triangle cases, where $(p_1, p_2, p)$ satisfies H\"older's relation, that is $1/p=1/p_1+1/p_2$ with $1\leq p_1, p_2 \leq \infty$.


\subsection{Sub-Laplacian on M\'etivier groups}
Let $G$ be a connected, simply connected, two-step nilpotent Lie group with Lie algebra $\mathfrak{g}$, such that $\mathfrak{g} = \mathfrak{g}_1 \oplus \mathfrak{g}_2$ with $[\mathfrak{g}_1, \mathfrak{g}_1]=\mathfrak{g}_2$ and $[\mathfrak{g}, \mathfrak{g}_2]=\{0\}$. We refer to $\mathfrak{g}_1, \mathfrak{g}_2$ as first layer and second layer respectively. Let $d_1= \text{dim}\ \mathfrak{g}_1$, $d_2= \text{dim}\ \mathfrak{g}_2$ and $d=d_1+d_2$. Suppose $X_1, \ldots, X_{d_1}$ is a basis of $\mathfrak{g}_1$ and $T_1, \ldots, T_{d_2}$ is a basis of $\mathfrak{g}_2$. Also, there is an inner product $\langle \cdot, \cdot \rangle$ on $\mathfrak{g}$, so that the basis $X_1, \ldots, X_{d_1}, T_1, \ldots, T_{d_2}$ becomes an orthonormal basis. This inner product $\langle \cdot, \cdot \rangle$ induces a norm on $\mathfrak{g}_2^{*}$, the dual of $\mathfrak{g}_2$, which we denoted by $|\cdot|$. Then for any $\lambda \in \mathfrak{g}_2^{*}$, there is a skew-symmetric endomorphism $J_{\lambda}$ on $\mathfrak{g}_1$ such that 
\begin{align*}
    \lambda([x, x']) = \langle J_{\lambda} x, x' \rangle, \quad \text{for all}\ \ x, x' \in \mathfrak{g}_1 .
\end{align*}
Consequently, $G$ is said to be a M\'etivier group if and only if $J_{\lambda}$ is invertible for all $\lambda \in \mathfrak{g}_2^{*} \setminus \{0\}$. In addition, if $J_{\lambda}$ satisfies $J_{\lambda}^2 = -|\lambda|^2 \ \text{id}_{\mathfrak{g}_1}$ for all $\lambda \in \mathfrak{g}_2^{*}$, the group $G$ is called a Heisenberg-type group. Therefore, the class of M\'etivier groups is larger than the class of Heisenberg-type groups; in fact the containment is strict (see \cite{Muller_Seeger_Nilpotent_Lie_group_2004}).
Since $G$ is a simply connected nilpotent Lie group, the exponential map $\exp : \mathfrak{g} \to G$ is a global diffeomorphism, and therefore $G$ can be identified with its Lie algebra $\mathfrak{g}$, which in turn can be identified with $\mathbb{R}^{d_1} \times \mathbb{R}^{d_2}$, via the chosen basis of $\mathfrak{g}$. On $G$, the Haar measure coincides with the Lebesgue measure on $\mathfrak{g}$ and the group operation is given by
\begin{align*}
    (x,u) (x', u') &= \left(x+x', u+u'+\tfrac{1}{2}[x,x'] \right), \quad x, x' \in \mathfrak{g}_1, \ u, u' \in \mathfrak{g}_2 .
\end{align*}
In this paper, we will always assume that $G$ is a M\'etivier group, unless otherwise specified. The sub-Laplacian $\mathcal{L}$, generated by the first-layer vector fields $X_1, \ldots, X_{d_1}$, is defined by
\begin{align*}
    \mathcal{L} = -(X_1^2 + \cdots + X_{d_1}^2) .
\end{align*}
Then $\mathcal{L}$ is positive and essentially self-adjoint on $L^2(G)$. Consequently, the spectral theorem allows us to define the functional calculus for $\mathcal{L}$; that is, for every bounded Borel measurable function $F: \mathbb{R} \to \mathbb{C}$, the spectral multiplier operator $F(\mathcal{L})$ is bounded on $L^2(G)$.

\subsection{Bilinear Bochner-Riesz means associated to \texorpdfstring{$\mathcal{L}$}{}}
\label{subsection: Bilinear Bochner-Riesz means}
The spectral decomposition of $\mathcal{L}$ has been well studied in the literature; see, for example, \cite{Niedorf_Metivier_group_2023}. For $f\in L^1(G)$ and $\lambda \in \mathfrak{g}_2^{*}$, we define the Fourier transform of $f$ along the central variable by 
\begin{align}
\label{Definition: Fourier transform along central variable}
    \mathcal{F}_2 f(x, \lambda) := f^{\lambda}(x) &= \int_{\mathfrak{g}_2} f(x,u)\, e^{- i \langle \lambda, u \rangle} \ du , \quad x \in \mathfrak{g}_1.
\end{align}
We define the $\lambda$-twisted convolution of two functions $\phi, \psi \in \mathcal{S}(\mathfrak{g}_1)$ by
\begin{align}
\label{definition: twisted convolution}
    \phi \times_{\lambda} \psi(x) &= \int_{\mathfrak{g}_1} \phi(x') \psi(x-x') e^{\frac{i}{2}\lambda([x, x'])} \ dx' , \quad x \in \mathfrak{g}_1 .
\end{align}

Let $\Lambda \in \mathbb{N}$, $\mathbf{b}=(b_1, \ldots, b_{\Lambda}) \in (0, \infty)^\Lambda$, $\mathbf{r}=(r_1, \ldots, r_\Lambda) \in \mathbb{N}^\Lambda$, $\mathbf{k}=(k_1, \ldots, k_\Lambda) \in \mathbb{N}_0^\Lambda$. We define the $(\mathbf{b}, \mathbf{r})$-rescaled Laguerre functions $\varphi_{\mathbf{k}}^{\mathbf{b}, \mathbf{r}}$ by setting
\begin{align*}
    \varphi_{\mathbf{k}}^{\mathbf{b}, \mathbf{r}} &= \varphi_{k_1}^{(b_1, r_1)} \otimes \cdots \otimes \varphi_{k_\Lambda}^{(b_\Lambda, r_\Lambda)},
\end{align*}
where $\varphi_{k}^{(\mu, m)}(z) = \mu^m L^{m-1}_k(\tfrac{1}{2}\mu |z|^2) e^{-\frac{1}{2}\mu |z|^2}$ for $z \in \mathbb{R}^{2m}$, $\mu>0$ is the $\mu$-rescaled Laguerre function and $L^{m-1}_k$ denotes the $k$-th Laguerre polynomial of type $m-1$.
    
Let $f \in \mathcal{S}(G)$. Then, for any Schwartz class functions $F: \mathbb{R} \to \mathbb{C}$, from \cite[Proposition 3.10]{Niedorf_Restriction_Stratified_group_2025}, the operator $F(\mathcal{L})$ is given by
\begin{align}
\label{Functional calculus for sub Laplacian}
    F(\mathcal{L})f(x,u) &= \frac{1}{(2\pi)^{d_2}} \int_{\mathfrak{g}_{2,r}^{*}} \sum_{\mathbf{k} \in \mathbb{N}^\Lambda} F(\eta_{\mathbf{k}}^{\lambda}) \left[f^{\lambda} \times_{\lambda} \varphi_{\mathbf{k}}^{\mathbf{b}^{\lambda}, \mathbf{r}}(R_{\lambda}^{-1}\cdot) \right](x) \ e^{i \langle \lambda, u \rangle} \ d\lambda ,
\end{align}
where $\eta_{\mathbf{k}}^{\lambda} := \eta_{\mathbf{k}}^{\mathbf{b}^{\lambda}, \mathbf{r}} = \displaystyle{\sum_{n=1}^{\Lambda} (2 k_n + r_n) b_n^{\lambda}}$, the functions $\lambda \to R_{\lambda}$ are Borel measurable on $\mathfrak{g}_{2,r}^{*}$ and $\mathfrak{g}_{2,r}^{*}$ is the Zariski open subset of $\mathfrak{g}_{2}^{*}$.

In particular for $\alpha \geq 0$, if we take $F(\eta)=(1-\eta)_{+}^{\alpha}$, then it is easy to verify that the expression of $F(\mathcal{L})$ given above is well defined. For $R>0$, we define the Bochner-Riesz operator associated with the sub-Laplacian $\mathcal{L}$ on the M\'etivier groups by
\begin{align*}
    S_R^{\alpha}(\mathcal{L})f(x,u) &= \frac{1}{(2\pi)^{d_2}} \int_{\mathfrak{g}_{2,r}^{*}} \sum_{\mathbf{k} \in \mathbb{N}^\Lambda} \left(1-\frac{\eta_{\mathbf{k}}^{\lambda}}{R} \right)_{+}^{\alpha} \left[f^{\lambda} \times_{\lambda} \varphi_{\mathbf{k}}^{\mathbf{b}^{\lambda}, \mathbf{r}}(R_{\lambda}^{-1}\cdot) \right](x) \, e^{i \langle \lambda, u \rangle} \, d\lambda .
\end{align*}

Correspondingly, for $f, g \in \mathcal{S}(G)$, the bilinear Bochner-Riesz operator associated to the sub-Laplacian $\mathcal{L}$, denoted by $\mathcal{B}_R^{\alpha}$ and is defined as
\begin{align}
\label{Bilinear Bochner-Riesz menas}
    \mathcal{B}_R^{\alpha}(f,g)(x,u) &= \frac{1}{(2\pi)^{2 d_2}} \int_{\mathfrak{g}_{2,r}^{*}} \int_{\mathfrak{g}_{2,r}^{*}} e^{i \langle \lambda_1 + \lambda_2 , u \rangle} \sum_{\mathbf{k}_1, \mathbf{k}_2 \in \mathbb{N}^\Lambda} \left(1-\frac{\eta_{\mathbf{k}_1}^{\lambda_1}+ \eta_{\mathbf{k}_2}^{\lambda_2}}{R} \right)_{+}^{\alpha} \\
    &\nonumber \hspace{1cm} \left[f^{\lambda_1} \times_{\lambda_1} \varphi_{\mathbf{k}_1}^{\mathbf{b}^{\lambda_1}, \mathbf{r}_1}(R_{\lambda_1}^{-1}\cdot) \right](x) \left[g^{\lambda_2} \times_{\lambda_2} \varphi_{\mathbf{k}_2}^{\mathbf{b}^{\lambda_2}, \mathbf{r}_2}(R_{\lambda_2}^{-1}\cdot) \right](x) \,  d\lambda_1 \, d\lambda_2 .
\end{align}

\subsection{Statement of the main result}We are concerned with the following estimate: for any $R>0$, whenever $\alpha>\alpha(p_1, p_2)$ for some $\alpha(p_1, p_2) \geq 0$, then we have
\begin{align}
\label{Boundedness of Bochner-Ries with R}
    \|\mathcal{B}_R^{\alpha}(f,g)\|_{L^p(G)} &\leq C \|f\|_{L^{p_1}(G)} \|g\|_{L^{p_2}(G)} ,
\end{align}
for all $f, g \in \mathcal{S}(G)$, where $1\leq p_1, p_2 \leq \infty$ and $1/p=1/p_1 + 1/p_2$, with the constant $C>0$ independent of $R$. In this article, we aim to determine the smoothness threshold $\alpha(p_1, p_2)$
analogous to the smoothness threshold of the Euclidean bilinear Bochner-Riesz means from \cite{Bernicot_Grafakos_Song_Yan_Bilinear_Bochner_Riesz_2015}, \cite{Liu_Wang_Bilinear_Bochner_Riesz_Non_Banach_2020}, where the Euclidean dimension is replaced by the topological dimension of the group. 

On $G$, we have the family of non-isotropic dilation $\{\delta_t\}_{t>0}$ defined by $\delta_t(x, u)= (tx, t^2u)$ (see \eqref{Definition: Dilation in Metivier groups}). It is then straightforward to check that
\begin{align*}
    \mathcal{B}_{t^{-2}R}^{\alpha}(f,g)(x,u) = \delta_{t^{-1}} \circ \mathcal{B}_R^{\alpha} (\delta_t f,\delta_t g)(x,u) .
\end{align*}

In view of the above relation, to study the $L^{p_1}(G) \times L^{p_2}(G) \to L^p(G)$ boundedness of $\mathcal{B}_R^{\alpha}$, it is enough to consider the case $R=1$. When $R=1$, we simply write $\mathcal{B}_1^{\alpha}$ as $\mathcal{B}^{\alpha}$.

The following is our first main result in this direction.
\begin{theorem}
\label{Bilinear Bochner-Riesz Main theorem}
Let $1\leq p_1, p_2 \leq \infty$ with $1/p = 1/p_1 +1/p_2$. Then $\mathcal{B}^{\alpha}$ is bounded from $L^{p_1}(G) \times L^{p_2}(G)$ to $L^{p}(G)$, if $p_1, p_2, p$ and $\alpha> \alpha(p_1, p_2)$ satisfy one of the following conditions:
\begin{enumerate}
    \item (Region I) $2 \leq p_1, p_2 \leq \infty$, $1\leq p \leq 2$ and $\alpha(p_1, p_2)= (d-1)(1-\frac{1}{p})$.
    \item (Region II) $2 \leq p_1, p_2, p \leq \infty$ and $\alpha(p_1, p_2)= \frac{d-1}{2} + d(\frac{1}{2}-\frac{1}{p})$.
    \item (Region III) $2 \leq p_2 \leq \infty$, $1\leq p_1, p \leq 2$ and $\alpha(p_1, p_2)= Q(\frac{1}{p_1}-\frac{1}{2})+(d-1)(1-\frac{1}{p})$.
    \item (Region III) $2 \leq p_1 \leq \infty$, $1\leq p_2, p \leq 2$ and $\alpha(p_1, p_2)= Q(\frac{1}{p_2}-\frac{1}{2})+(d-1)(1-\frac{1}{p})$.
    \item (Region IV) $1\leq p_1 \leq 2 \leq p_2 \leq \infty$, $0<p\leq 1$ and $\alpha(p_1, p_2)= (d+1)(\frac{1}{p}-1)+Q(\frac{1}{2}-\frac{1}{p_2})$.
    \item (Region IV) $1\leq p_2 \leq 2 \leq p_1 \leq \infty$, $0<p \leq 1$ and $\alpha(p_1, p_2)= (d+1)(\frac{1}{p}-1)+Q(\frac{1}{2}-\frac{1}{p_1})$.
    \item (Region V) $1\leq p_1, p_2 \leq 2$ and $\alpha(p_1, p_2)= (d+1)(\frac{1}{p}-1)$.
\end{enumerate}
\end{theorem}

\vspace{-0.5cm}
\begin{figure}[!ht]
\begin{centering}
\definecolor{qqqqff}{rgb}{0,0,1}
\begin{tikzpicture}[line cap=round,line join=round,x=0.7cm,y=0.7cm]
\draw (6,6)-- (0,6);
\draw (6,0)-- (6,6);
\draw (0,3)-- (6,3);
\draw [->] (0,0) -- (8,0);
\draw [->] (0,0) -- (0,7.5);
\draw (3,6)-- (3,3);
\draw (0,6)-- (6,0);
\draw (3,3)-- (3,0);
\draw (0,3)-- (3,0);
\begin{scriptsize}
\fill [color=qqqqff] (0,0) circle (1.5pt);
\draw[color=qqqqff] (1,-.5) node {$\alpha>d-\frac{1}{2}$};
\fill [color=qqqqff] (0,6) circle (1.5pt);
\draw[color=qqqqff] (0.8,6.5) node {$\alpha > \frac{Q}{2}$};
\fill [color=qqqqff] (6,6) circle (1.5pt);
\draw[color=qqqqff] (6.19,6.5) node {$\alpha>d+1$};
\fill [color=qqqqff] (6,0) circle (1.5pt);
\draw[color=qqqqff] (6.9,0.4) node {$\alpha>\frac{Q}{2}$};
\fill [color=qqqqff] (3,0) circle (1.5pt);
\draw[color=qqqqff] (4.2,-0.5) node {$\alpha>\frac{d-1}{2}$};
\fill [color=qqqqff] (3,6) circle (1.5pt);
\draw[color=qqqqff] (3.3,6.5) node {$\alpha>\frac{d+1}{2}$};
\fill [color=qqqqff] (0,3) circle (1.5pt);
\draw[color=qqqqff] (-1,3.7) node {$\alpha>\frac{d-1}{2}$};
\fill [color=qqqqff] (3,3) circle (1.5pt);
\draw[color=qqqqff] (3.7,3.33) node {$\alpha>0$};
\fill [color=qqqqff] (6,3) circle (1.5pt);
\draw[color=qqqqff] (7,2.8) node {$\alpha>\frac{d+1}{2}$};
\fill [color=qqqqff] (6,-0.5) circle (0pt);
\draw[color=qqqqff] (6.0,-0.38) node {$1$};
\fill [color=qqqqff] (3,-0.5) circle (0pt);
\draw[color=qqqqff] (2.9,-0.48) node {$\frac{1}{2}$};
\fill [color=qqqqff] (-0.25,-0.42) circle (0pt);
\draw[color=qqqqff] (-0.25,-0.28) node {$O$};
\fill [color=qqqqff] (-0.5,3) circle (0pt);
\draw[color=qqqqff] (-0.38,3.0) node {$\frac{1}{2}$};
\fill [color=qqqqff] (-0.5,6) circle (0pt);
\draw[color=qqqqff] (-0.38,6.0) node {$1$};
\fill [color=qqqqff] (8.5,-0.5) circle (0pt);
\draw[color=qqqqff] (8.5,-0.16) node {$\frac{1}{p_1}$};
\fill [color=qqqqff] (-0.5,8.5) circle (0pt);
\draw[color=qqqqff] (-0.58,7.5) node {$\frac{1}{p_2}$};
\draw[color=qqqqff] (2.3,2.1) node {$I$};
\draw[color=qqqqff] (1.0,1.0) node {$II$};
\draw[color=qqqqff] (4.1,1.0) node {$III$};
\draw[color=qqqqff] (1.0,4.1) node {$III$};
\draw[color=qqqqff] (4.7,2.1) node {$IV$};
\draw[color=qqqqff] (2.3,4.7) node {$IV$};
\draw[color=qqqqff] (4.7,4.7) node {$V$};
\end{scriptsize}
\end{tikzpicture}
\begin{tikzpicture}[line cap=round,line join=round,x=0.7cm,y=0.7cm]
\draw (6,6)-- (0,6);
\draw (6,0)-- (6,6);
\draw (0,3)-- (6,3);
\draw [->] (0,0) -- (8,0);
\draw [->] (0,0) -- (0,7.5);
\draw (3,6)-- (3,3);
\draw (0,6)-- (6,0);
\draw (3,3)-- (3,0);
\draw (0,3)-- (3,0);
\begin{scriptsize}
\fill [color=qqqqff] (0,0) circle (1.5pt);
\draw[color=qqqqff] (1,-.5) node {$\alpha>d-\frac{1}{2}$};
\fill [color=qqqqff] (0,6) circle (1.5pt);
\draw[color=qqqqff] (0.75,6.5) node {$\alpha > \frac{d}{2}$};
\fill [color=qqqqff] (6,6) circle (1.5pt);
\draw[color=qqqqff] (6.19,6.5) node {$\alpha>d$};
\fill [color=qqqqff] (6,0) circle (1.5pt);
\draw[color=qqqqff] (6.8,0.4) node {$\alpha>\frac{d}{2}$};
\fill [color=qqqqff] (3,0) circle (1.5pt);
\draw[color=qqqqff] (4.2,-0.5) node {$\alpha>\frac{d-1}{2}$};
\fill [color=qqqqff] (3,6) circle (1.5pt);
\draw[color=qqqqff] (3.3,6.5) node {$\alpha>\frac{d}{2}$};
\fill [color=qqqqff] (0,3) circle (1.5pt);
\draw[color=qqqqff] (-1,3.7) node {$\alpha>\frac{d-1}{2}$};
\fill [color=qqqqff] (3,3) circle (1.5pt);
\draw[color=qqqqff] (3.7,3.33) node {$\alpha>0$};
\fill [color=qqqqff] (6,3) circle (1.5pt);
\draw[color=qqqqff] (6.8,2.8) node {$\alpha>\frac{d}{2}$};
\fill [color=qqqqff] (6,-0.5) circle (0pt);
\draw[color=qqqqff] (6.0,-0.38) node {$1$};
\fill [color=qqqqff] (3,-0.5) circle (0pt);
\draw[color=qqqqff] (2.9,-0.48) node {$\frac{1}{2}$};
\fill [color=qqqqff] (-0.25,-0.42) circle (0pt);
\draw[color=qqqqff] (-0.25,-0.28) node {$O$};
\fill [color=qqqqff] (-0.5,3) circle (0pt);
\draw[color=qqqqff] (-0.38,3.0) node {$\frac{1}{2}$};
\fill [color=qqqqff] (-0.5,6) circle (0pt);
\draw[color=qqqqff] (-0.38,6.0) node {$1$};
\fill [color=qqqqff] (8.5,-0.5) circle (0pt);
\draw[color=qqqqff] (8.5,-0.16) node {$\frac{1}{p_1}$};
\fill [color=qqqqff] (-0.5,8.5) circle (0pt);
\draw[color=qqqqff] (-0.58,7.5) node {$\frac{1}{p_2}$};
\draw[color=qqqqff] (2.3,2.1) node {$I$};
\draw[color=qqqqff] (1.0,1.0) node {$II$};
\draw[color=qqqqff] (4.1,1.0) node {$III$};
\draw[color=qqqqff] (1.0,4.1) node {$III$};
\draw[color=qqqqff] (4.7,2.1) node {$IV$};
\draw[color=qqqqff] (2.3,4.7) node {$IV$};
\draw[color=qqqqff] (4.7,4.7) node {$V$};
\end{scriptsize}
\end{tikzpicture}
        \caption{Here $O=(0,0)$, and $\alpha>\alpha(p_1, p_2)$ represents that $\mathcal{B}^{\alpha}$ is bounded on $L^{p_1}(G) \times L^{p_2}(G) \to L^p(G)$ for $\alpha>\alpha(p_1, p_2)$. The left picture is described by Theorem \ref{Bilinear Bochner-Riesz Main theorem}, while right picture is described by Theorem \ref{Bilinear Bochner-Riesz theorem with restricted f and g}.}
\end{centering}        
\end{figure}

To understand the significance of Theorem \ref{Bilinear Bochner-Riesz Main theorem}, let us compare it with its Euclidean counterpart, Theorem \ref{Theorem: Euclidean bilinear Bochner-Riesz for Grafakos}. For $p>1$, on the region $I$ and $II$, our result is an exact analogue of the Theorem \ref{Theorem: Euclidean bilinear Bochner-Riesz for Grafakos}, in which the Euclidean dimension of $\mathbb{R}^n$ in the expression of the smoothness threshold $\alpha(p_1, p_2)$ is replaced by the topological dimension $d$ of the underlying M\'etivier groups $G$. On the other hand, for the region $p\leq 1$, note that at $(1,1,1/2)$, the smoothness threshold in Theorem \ref{Theorem: Euclidean bilinear Bochner-Riesz for Grafakos} is $n-1/2$, while in our setting the corresponding threshold is $d+1$. This difference arises because in the Euclidean case, the kernel of the bilinear Bochner-Riesz operator can be explicitly expressed in terms of the Bessel functions (see \cite[Proposition 4.2 (i)]{Bernicot_Grafakos_Song_Yan_Bilinear_Bochner_Riesz_2015}). In our setting, an explicit kernel representation for the Bochner-Riesz operator $\mathcal{B}^{\alpha}$ is not known (see \cite[Remark, p. 118]{Muller_Riesz_Means_Kohn_Laplacian_1989}). Likewise, for $(1,2,2/3)$, the Theorem \ref{Theorem: Euclidean bilinear Bochner-Riesz for Grafakos} requires $\alpha>n/2$, whereas our result holds for $\alpha>(d+1)/2$. However, at $(1, \infty,1)$ we only get $\alpha>Q/2$. Hence, we conclude that our theorem on M\'etivier groups gives boundedness of $\mathcal{B}^{\alpha}$ for $\alpha>\alpha(p_1, p_2)$, where $\alpha(p_1, p_2)$ is expressed in terms of the topological dimension $d$ for the region $I$, $II$ and $V$ and in terms of a combination of $d$ and $Q$ for the region $III$ and $IV$, see also Figure 2.

In addition, an improvement of Theorem \ref{Bilinear Bochner-Riesz Main theorem} for $p\leq 1$ is possible whenever the Fourier transform of the input functions $f$ or $g$ or both is supported away from a fixed small neighborhood of the origin. In fact, we have the following theorem.
\begin{theorem}
\label{Bilinear Bochner-Riesz theorem with restricted f and g}
Let $1\leq p_1, p_2 \leq \infty$ with $1/p = 1/p_1 +1/p_2$. Then the bound 
    \begin{align*}
        \|\mathcal{B}^{\alpha}(f,g)\|_{L^p} &\leq C \|f\|_{L^{p_1}} \|g\|_{L^{p_2}}
    \end{align*}
    holds true whenever $p_1, p_2, p$ and $\alpha> \alpha(p_1, p_2)$ satisfy one of the following conditions:
\begin{enumerate}
    \item (Region III) $2 \leq p_2 \leq \infty$, $1\leq p_1, p \leq 2$ and $\alpha(p_1, p_2)= d(\frac{1}{2}-\frac{1}{p_2})-(1-\frac{1}{p})$, if $\supp \mathcal{F}_2 g(z, \cdot) \subseteq \{|\lambda_2| \geq \kappa_2 \}$ for some $\kappa_2>0$ and every $z \in \mathfrak{g}_1$.
    \item (Region III) $2 \leq p_1 \leq \infty$, $1\leq p_2, p \leq 2$ and $\alpha(p_1, p_2)= d(\frac{1}{2}-\frac{1}{p_1})-(1-\frac{1}{p})$, if $\supp \mathcal{F}_2 f(y, \cdot) \subseteq \{|\lambda_1| \geq \kappa_1 \}$ for some $\kappa_1>0$ and every $y \in \mathfrak{g}_1$.
    \item (Region IV) $1\leq p_1 \leq 2 \leq p_2 \leq \infty$, $0<p\leq 1$ and $\alpha(p_1, p_2)= d(\frac{1}{p_1}-\frac{1}{2})$, if $\supp \mathcal{F}_2 g(z, \cdot) \subseteq \{|\lambda_2| \geq \kappa_2 \}$ for some $\kappa_2>0$ and every $z \in \mathfrak{g}_1$.
    \item (Region IV) $1\leq p_2 \leq 2 \leq p_1 \leq \infty$, $0<p\leq 1$ and $\alpha(p_1, p_2)= d(\frac{1}{p_2}-\frac{1}{2})$, if $\supp \mathcal{F}_2 f(y, \cdot) \subseteq \{|\lambda_1| \geq \kappa_1 \}$ for some $\kappa_1>0$ and every $y \in \mathfrak{g}_1$.
    \item (Region V) $1\leq p_1, p_2 \leq 2$ and $\alpha(p_1, p_2)= d(\frac{1}{p}-1)$, if $\supp \mathcal{F}_2 f(y, \cdot) \subseteq \{|\lambda_1| \geq \kappa_1 \}$ and $\supp \mathcal{F}_2 g(z, \cdot) \subseteq \{|\lambda_2| \geq \kappa_2 \}$ for some $\kappa_1, \kappa_2>0$ and every $y, z \in \mathfrak{g}_1$.
\end{enumerate}
\end{theorem}

Notice that at $(1, \infty, 1)$, under the additional assumption on the support of the input functions, we are able to replace the threshold from $Q/2$ to $d/2$. Similarly, at the points $(1,2,2/3)$ and $(1,1,1/2)$, we have further reduced the threshold from $(d+1)/2$ to $d/2$, and from $d+1$ to $d$, respectively. Hence, the above result provides an exact analogue of the Euclidean counterpart (see Theorem \ref{Theorem: Euclidean bilinear Bochner-Riesz for Grafakos}), except at the point $(1,1,1/2)$, under the assumptions on the support of the input functions.

As observed in Theorem \ref{Bilinear Bochner-Riesz Main theorem},  at the point $(1, \infty,1)$, the bilinear Bochner-Riesz mean is bounded from $L^1(G) \times L^{\infty}(G) \to L^1(G)$ whenever $\alpha>Q/2$. This assertion can be further improved if we consider the mixed norm estimates. For $0<p,q< \infty$, let us define the mixed norm of a measurable function $h$ on $G$ given by
\begin{align*}
    \|h\|_{L_x^{p} L_u^{q}(G)} &:= \Big(\int_{\mathbb{R}^{d_2}} \Big( \int_{\mathbb{R}^{d_1}} |h(x,u)|^p \ dx \Big)^{q/p} \ du \Big)^{1/q} ,
\end{align*}
with obvious modification if one of $p, q$ is $\infty$.

Concerning the mixed norm estimate, the following theorems are our first main contributions in this direction.
\begin{theorem}
\label{Theorem: Mixed norm estimate for first layer}
    If $\alpha>(d+1)/2$, then
    \begin{align*}
        \|\mathcal{B}^{\alpha}(f,g)\|_{L_x^{2/3} L_u^{1}(G)} \leq C \|f\|_{L^1(G)} \|g\|_{L_x^2 L_u^{\infty}(G)} .
    \end{align*}
\end{theorem}

Before stating the other mixed norm estimate, let us first make an assumption about the second-layer weighted Plancherel estimates, which will be crucial in our proof. Let us set $T := (-( T_1^2+ \cdots + T_{d_2}^2))^{1/2}$.

\textbf{Assumption A:} If $F: \mathbb{R} \to \mathbb{C}$ is a bounded Borel function supported in a compact subset $A \subseteq \mathbb{R}$ and $\Theta: (0, \infty) \to \mathbb{C}$ is a smooth function with compact support, then the convolution kernel $\mathcal{K}_{F(\mathcal{L})\Theta(2^{M}T)}$ of $F(\mathcal{L})\Theta(2^{M}T)$ satisfies
\begin{align*}
        \int_G \left| |u|^{N} \mathcal{K}_{F(\mathcal{L})\Theta(2^{M}T)}(x,u) \right|^2 \ d(x,u) &\leq C_{A, \Theta, N} 2^{M(2 N-d_2)} \|F\|_{L_{N}^2}^2 ,
\end{align*}
for all $N \geq 0$ and $M \in \mathbb{Z}$.

The Assumption A is known to hold for the Heisenberg type groups, see \cite[Lemma 10]{Martini_Spectral_Multiplier_Heisenberg_Reiter_2015}. One can also see \cite{Hebisch_Spectral_Multiplier_Heisenberg_1993} for related discussion. Unfortunately, for all M\'etivier groups, whether Assumption A holds remains an open question. In \cite{Martini_Muller_New_Class_Two_Step_Stratified_Groups_2014}, Martini and M\"uller proved the second layer weighted Plancherel estimates under the assumption of some appropriate bounds of the derivatives of the functions $\lambda \to b_n^{\lambda}$ and $\lambda \to P_{n,{\lambda}}$ for $n=1, \ldots, \Lambda$, appeared in the spectral decomposition of $-J_{\lambda}^2$ (see Proposition \ref{Prop: Spectral decomposition of Jmu}). The singularities of these functions are lie in the Zariski closed subset $\mathfrak{g}_2^{*} \setminus \mathfrak{g}_{2,r}^{*}$. In general for any two-step stratified groups, the singularity set can be quite complicated, but in \cite{Martini_Muller_New_Class_Two_Step_Stratified_Groups_2014}, they were able to handle the situation for some particular cases, for example when $\dim \mathfrak{g}_2 \leq 2$ or $d \leq 7$.

In connection with the other mixed norm estimate, we have the following results.
\begin{theorem}
\label{Theorem: Mixed norm estimate for second layer}
    Under the Assumption A, if $\alpha>(d+1)/2$, then 
    \begin{align*}
        \|\mathcal{B}^{\alpha}(f,g)\|_{L_u^{2/3} L_x^{1}(G)} \leq C \|f\|_{L^1(G)} \|g\|_{L_u^2 L_x^{\infty}(G)} .
    \end{align*}
\end{theorem}

To prove our theorems, we utilize the Fourier series decomposition of the bilinear Bochner-Riesz multiplier, a technique employed in \cite[Proposition 3.8]{Bernicot_Grafakos_Song_Yan_Bilinear_Bochner_Riesz_2015} and \cite[Theorem 3.2]{Liu_Wang_Bilinear_Bochner_Riesz_Non_Banach_2020}. Although one can lift the Euclidean technique to our setting, this approach only yields a smoothness threshold in terms of the homogeneous dimension $Q$ of the M\'etivier groups. The main challenge is to refine this and replace $Q$ with the topological dimension $d$ of $G$. Similar to linear setup, for sub-Laplacians on M\'etivier groups, one might consider employing suitable weights to reduce the dimension from $Q$ to $d$. We show that at some particular points the boundedness result can indeed be established with smoothness threshold expressed in terms of $d$, using weighted Plancherel estimates. In the linear setting, it was Fefferman and Stein's idea \cite{Fefferman_Spherical_multiplier_1973} to use the restriction estimates to obtain sharp results for the Bochner-Riesz multiplier. Similarly here also one might consider restriction type estimates in this setup. However, a weighted version of restriction-type estimates would be required. Unfortunately, such results cannot generally be expected to hold, as discussed in \cite[Section 8]{Niedorf_p_specific_Heisenber_group_2024}. To overcome this difficulty, we use ideas from \cite{Niedorf_Metivier_group_2023}, where the author studied $p$-specific Bochner-Riesz multipliers in the linear setting. However, adapting such techniques to the bilinear settings has its own technical challenges.

The rest of this paper is organized as follows. In Section \ref{Preliminaries}, we gather several well-known results related to the sub-Riemannian geometry of $G$, the spectral decomposition of $-J_{\lambda}^2$, and the integration of weights and homogeneous norms. Section \ref{Section: pointwise kernel estimate} focuses on the pointwise kernel estimates for Bochner-Riesz means, the weighted Plancherel estimates and a bilinear version of the weighted Plancherel estimates for the sub-Laplacian $\mathcal{L}$ with a first-layer weight. To prove Theorem \ref{Bilinear Bochner-Riesz Main theorem}, we decompose the bilinear Bochner-Riesz operator and the corresponding kernel in Section \ref{Section: Dyadic decomposition of Bochner-Riesz} and reduce the proof to several specific cases. Sections \ref{Section: Claim A} through \ref{Section: Proof of claim at 1 infinity} are devoted to establishing those particular cases. Section \ref{Section: Proof of theorem for restricted f and g} contains the proof of Theorem \ref{Bilinear Bochner-Riesz theorem with restricted f and g}, while in Section \ref{Section: Mixed norm estimates}, we present the proofs of Theorem \ref{Theorem: Mixed norm estimate for first layer} and Theorem \ref{Theorem: Mixed norm estimate for second layer}.

Throughout the article, we use standard notation. Let $\mathbb{N} = \{1,2, \ldots\}$ and $\mathbb{N}_0 = \mathbb{N} \cup \{0\}$. We use letter $C$ to indicate a positive constant that is independent of the main parameters, but may vary from line to line. When writing estimates, we use the notation $f \lesssim g$ to indicate $f \leq Cg$ for some $C > 0$, and whenever $f \lesssim g\lesssim f$ , we shall write $f \sim g$. We sometimes write $f \lesssim_{\epsilon} g$ to denote $f \leq C g$ where the constant $C$ may depend on the implicit constant $\epsilon$. For a Lebesgue measurable subset $E$ of $\mathbb{R}^d$, we denote by $\chi_E$ the characteristic function of the set $E$. Let $\Bar{B}$ denote the closure of a ball $B$. For any function $G$ on $\mathbb{R}$, define $\delta_R G(\eta) = G(R \eta) $ for $R>0$. Let $\mathcal{S}(G)$ denote the space of all Schwartz class functions on $G$, where we have identified $G \cong \mathbb{R}^d$. For $f,g \in \mathcal{S}(G)$, the group convolution of $f$ and $g$ is given by
\begin{align*}
    f*g(x,u) &= \int_G f(x',u') g((x',u')^{-1}(x,u)) \ d(x',u'), \quad (x,u) \in G.
\end{align*}

\section{Preliminaries}
\label{Preliminaries}
In this section, we collect some preliminary results which are well known in the literature; see, for example, \cite{Niedorf_Metivier_group_2023}, \cite{Martini_Muller_New_Class_Two_Step_Stratified_Groups_2014}, \cite{Martini_Muller_Wave_Metivier_2024}, \cite{Martini_Muller_spectral_mult_two_step_2016}, \cite{Martini_Muller_Golo_Spectral_Multiplier_Lower_Regularity_2023}, and \cite{Muller_Ricci_Nilpotent_1996} for further details. Let $\varrho$ denote the Carnot-Carath\'eodory distance on $G$, associated with the left-invariant vector fields $X_1, \ldots, X_{d_1}$, which satisfy the H\"ormander's bracket-generating condition. Therefore in view of the Chow-Rashevskii theorem (\cite[Proposition III.4.1]{Varopoulos_Saloff_Coulhon__Analysis_groups_1992}), the distance $\varrho$ defines a metric on $G$, which induces the Euclidean topology on $G$. Furthermore, by the left-invariant property of $X_1, \cdots, X_{d_1} $, the Carnot-Carath\'eodory distance  $\varrho$ is itself left-invariant, that is for any $(g,h) \in G$, we have
\begin{align*}
    \varrho((g,h)(x,u), (g,h)(x',u')) = \varrho((x,u), (x',u')), \quad \text{for all}\quad (x,u), (x',u') \in G .
\end{align*}
If we further set 
\begin{align*}
    |(x,u)| &:= \varrho((x,u),0) ,
\end{align*}
where $0=(0,0)$ is the identity element of the group $G$, then with respect to the family of automorphic dilations $\delta_t$ defined by
\begin{align}
\label{Definition: Dilation in Metivier groups}
    \delta_t(x,u) &= (t x, t^2 u), \quad t> 0 .
\end{align} 
$|(x,u)|$ satisfies $|\delta_t (x,u)| = t |(x,u)|$. Therefore, $|(x,u)|$ becomes a homogeneous norm in the sense of Folland and Stein \cite[p. 8]{Folland_Stein_Hardy_space_homogeneous_group_1982}. On the other hand, if one define
\begin{align*}
    \|(x,u)\| &:= (|x|^4 + |u|^2)^{1/4}, \quad (x,u) \in G ,
\end{align*}
then this also defines a homogeneous norm on $G$. Now since any two homogeneous norms on homogeneous groups are always equivalent \cite[Proposition 1.5]{Folland_Stein_Hardy_space_homogeneous_group_1982}, therefore, due to the left-invariance of $\varrho$, we have
\begin{align}
\label{Equivalent distance formula}
    \varrho((x,u), (x',u')) &\sim \|(x,u)^{-1}(x',u')\| .
\end{align}

Let us also mention that, for $t>0$, the heat kernel $\mathcal{K}_{\exp(-t\mathcal{L})}$ associated with the sub-Laplacians (see \cite{Varopoulos_Analysis_Lie_Group_1988}) satisfies the following bound.
\begin{align}
\label{Heat kernel bound for subLaplacians}
    |\mathcal{K}_{\exp(-t\mathcal{L})}(x,u)| &\leq C\, t^{-Q/2} \exp{\Big\{-c \tfrac{\|(x,u)\|^2}{t} \Big\}} .
\end{align}

We denote $B^{\varrho}((x,u), R)$ to be the ball ($\varrho$-ball) centered at $(x,u)$ and radius $R>0$ with respect to the Carnot-Carathe\'odory distance $\varrho$. Then volume of the ball satisfies
\begin{align}
\label{Measure of the ball}
    |B^{\varrho}((x,u), R)| &\sim R^Q |B^{\varrho}(0, 1)| ,
\end{align}
where $|\cdot|$ denote the Lebesgue measure and $Q=d_1+2d_2$ is the homogeneous dimension of the underlying space $G$. We call $d=d_1+d_2$ to be the topological dimension of $G$. In the sequel, we denote  $B^{\varrho}((x,u), R)$ simply by $B((x,u), R)$, which means the ball is taken with respect to the Carnot-Carathe\'odory distance $\varrho$. Note that since $G$ is a M\'etivier group, we always have $d_2 = \dim \mathfrak{g}_2 < \dim \mathfrak{g}_1 = d_1$. This follows easily from the fact that the map $\lambda \to \lambda([\cdot, x'])$ from $\mathfrak{g}_2^{*} \to (\mathfrak{g}_1/\mathbb{R}x')^{*}$ is injective for $x' \neq 0$.

Recall that $G$ is a M\'etivier group if and only if the skew-symmetric endomorphism $J_{\lambda}$ on $\mathfrak{g}_1$ is invertible for all $\lambda \in \mathfrak{g}_2^{*} \setminus \{0\}$. Consequently, $-J_{\lambda}^2=J_{\lambda}^* J_{\lambda}$ is self-adjoint and non-negative. The following proposition states that the family $J_{\lambda}$ admits a simultaneous spectral decomposition for all $\lambda$ belonging to a certain Zariski-open subset of $\mathfrak{g}_2^*$. 
\begin{proposition}\cite[Proposition 3.1]{Niedorf_Metivier_group_2023}
\label{Prop: Spectral decomposition of Jmu}, \cite[Lemma 5]{Martini_Muller_New_Class_Two_Step_Stratified_Groups_2014}
    There exists $\Lambda \in \mathbb{N}$, $\mathbf{r}=(r_1, \ldots, r_\Lambda) \in \mathbb{N}^\Lambda$, a non-empty and homogeneous Zariski-open subset $\mathfrak{g}_{2,r}^{*}$ of $\mathfrak{g}_2^{*}$, a function $\lambda \to \mathbf{b}^{\lambda}=(b_1^{\lambda}, \ldots, b_\Lambda^{\lambda}) \in [0, \infty)^{\Lambda}$ defined on $\mathfrak{g}_2^{*}$, functions $\lambda \mapsto P_{n,{\lambda}} $ on $\mathfrak{g}_{2,r}^{*}$ where $P_{n,{\lambda}} : \mathfrak{g}_1 \to \mathfrak{g}_1$ for $n \in \{1, \ldots, \Lambda\}$ and a function $\lambda \mapsto R_{\lambda} \in O(d_1)$ defined on $\mathfrak{g}_{2,r}^{*}$ such that
    \begin{align*}
        -J_{\lambda}^2 &= \sum_{n=1}^{\Lambda} (b_n^{\lambda})^2 P_{n,{\lambda}} \quad \text{for all} \quad \lambda \in \mathfrak{g}_{2,r}^{*},
    \end{align*}
    with $P_{n,{\lambda}} R_{\lambda} = R_{\lambda} P_n$, $J_{\lambda}(\text{ran}\  P_{n,{\lambda}}) \subseteq \text{ran}\  P_{n,{\lambda}}$, where $\text{ran}\  P_{n,{\lambda}}$ is the range of $P_{n,{\lambda}}$ for all $\lambda \in \mathfrak{g}_{2,r}^{*}$ and $P_n$ denotes the projection from $\mathbb{R}^{d_1}= \mathbb{R}^{2 r_1}\oplus \cdots \oplus \mathbb{R}^{2 r_\Lambda}$ onto the $n$-th layer for all $n \in \{1, \ldots, \Lambda\}$. Moreover
    \begin{enumerate}
        \item $\lambda \to b_n^{\lambda}$ are homogeneous of degree $1$, real analytic on $\mathfrak{g}_{2,r}^{*}$ and continuous on $\mathfrak{g}_2^{*}$, further it satisfies $b_n^{\lambda}>0$ for all $\lambda \in \mathfrak{g}_{2,r}^{*}$, $n \in \{1, \ldots, \Lambda\}$, and $b_n^{\lambda} \neq b_{n'}^{\lambda}$ if $n \neq n'$ for all $\lambda \in \mathfrak{g}_{2,r}^{*}$ and $n, n' \in \{1, \ldots, \Lambda\}$,
        \item $\lambda \to P_{n,{\lambda}}$ are homogeneous of degree $0$, (componentwise) real analytic functions on $\mathfrak{g}_{2,r}^{*}$, and the functions $P_{n,{\lambda}}$ are orthogonal projections on $\mathfrak{g}_1$ of rank $2 r_n$ for all $\lambda \in \mathfrak{g}_{2,r}^{*}$, such that the ranges are pairwise orthogonal. Moreover
        \begin{align}
        \label{HilbertSchmidt norm on dual}
            \sum_{n=1}^{\Lambda} r_n b_n^{\lambda} \sim \Big( \sum_{n=1}^{\Lambda} 2 r_n (b_n^{\lambda})^2 \Big)^{1/2} = (tr(J_{\lambda}^* J_{\lambda}))^{1/2} ,
        \end{align}
        and as a function of $\lambda$, this expression gives a norm induced by an inner product on $\mathfrak{g}_2^{*}$ .
        \item the functions $\lambda \to R_{\lambda}$ are Borel measurable on $\mathfrak{g}_{2,r}^{*}$, homogeneous of degree $0$ and there exists a family $(U_{\ell})_{\ell \in \mathbb{N}}$ of disjoint Euclidean open subsets $U_{\ell} \subseteq \mathfrak{g}_{2,r}^{*}$ such that the union is $\mathfrak{g}_{2,r}^{*}$, up to a set of measure zero and $\lambda \to R_{\lambda}$ is (componentwise) real analytic functions on each $U_{\ell}$.
   \end{enumerate}
\end{proposition}

The following lemma plays an important role in our subsequent proofs.
\begin{lemma}
\label{Lemma: Decomposition of ball}
    Let $R>0$. If $\varrho((a,b), 0) \leq K R$ for some $K>0$, then there exists a constant $C>0$ such that
    \begin{align*}
        B((a, b),R) \subseteq B^{|\cdot|}(a, C R) \times B^{|\cdot|}(b, C R^2) \subseteq \mathbb{R}^{d_1} \times \mathbb{R}^{d_2} ,
    \end{align*}
    where $B^{|\cdot|}(a, R)$ denotes the ball of radius $R$ and centered at $a$ with respect to Euclidean distance.

    In particular, there exists a constant $C>0$ such that
    \begin{align*}
        B(0, R) \subseteq B^{|\cdot|}(0, C R) \times B^{|\cdot|}(0, C R^2) \subseteq \mathbb{R}^{d_1} \times \mathbb{R}^{d_2} .
    \end{align*}
    \end{lemma}
\begin{proof}
    For $(x,u) \in B((a, b),R)$ we have $\varrho((x,u), (a,b)) \leq R$. So that by \eqref{Equivalent distance formula}, we also have $\|(a,b)^{-1}(x,u)\| \leq C R$, for some $C>0$. Therefore,
    \begin{align*}
        (|x-a|^4 + |u-b-\tfrac{1}{2}[a, x]|^2)^{1/4} &\leq C R .
    \end{align*}
    From this we can easily see
    \begin{align*}
        |x-a| \leq C R \quad \text{and} \quad |u-b| \leq C R^2 + \tfrac{1}{2}|[a,x]| .
    \end{align*}
    Note that
    \begin{align*}
        |[a,x]| \leq C_0 |a||x| \quad \text{with} \quad C_0 =\sup_{a,x \neq 0} \tfrac{|[a,x]|}{|a||x|} .
    \end{align*}
    On the other hand, the assumption $\varrho((a,b), 0) \leq K R$ implies $|a| \leq K R$. Therefore, we also have $|x| \leq C R$. Consequently, we get $|[a,x]| \leq C R^2$. Hence, we have $|u-b| \leq C R^2$. This completes the proof of the lemma.
\end{proof}

The following two results are about the integration of weights and homogeneous norms.
\begin{lemma}
\label{Lemma: Integral of weight over ball}
Suppose $0\leq \gamma < d_1 $. Then for any $R>0$, we have
\begin{align*}
    \int_{B((a,b), R)} \frac{d(x,u)}{|x|^{\gamma}} &\leq C R^{Q-\gamma} .
\end{align*}
\end{lemma}

\begin{proof}
Using Lemma \eqref{Lemma: Decomposition of ball} for any $0\leq \gamma < d_1 $, we get
\begin{align*}
    \int_{B((a,b), R)} \frac{d(x, u)}{|x|^{\gamma}} = \int_{B(0, R)} \frac{d(x, u)}{|x-a|^{\gamma}} &\leq C \int_{B^{|\cdot|}(0, C R)} \int_{B^{|\cdot|}(0, C R^2)} \frac{dx \ du}{|x-a|^{\gamma}} \\
    &\leq C \int_{B^{|\cdot|}(a, C R)} \frac{dx}{|x|^{\gamma}} \int_{B^{|\cdot|}(0, C R^2)} du \leq C R^{Q - \gamma} .
\end{align*}  
\end{proof}

\begin{lemma}
\label{lemma: Estimate of distance on outside ball}
    Let $R> 0$. Then for any $N> Q$, we have
    \begin{align*}
        \int_{\|(x,u)\| > R} \frac{d(x,u)}{\big( 1 + \|(x,u)\| \big)^N} & \leq C R^{-N + Q} .
    \end{align*}
\end{lemma}
\begin{proof}
Decomposing the integral into annular region for $N>Q$, we can see
\begin{align*}
  \int_{\|(x,u)\| > R} \frac{d(x,u)}{\big( 1 + \|(x,u)\| \big)^N} &= \sum_{k = 0} ^\infty \int_{2^k R < \|(x,u)\| \leq 2^{k+ 1} R}  \frac{d(x,u)}{\big( 1 + \|(x,u)\| \big)^N} \\
  & \leq C \sum_{k = 0} ^\infty \frac{1}{(2^k R)^N} (2^k R)^{Q} \leq C R^{-N + Q} .
\end{align*} 
\end{proof}

\section{Kernel Estimates}
\label{Section: pointwise kernel estimate}
Recall that from \eqref{Bilinear Bochner-Riesz menas}, for $f, g \in \mathcal{S}(G)$, the bilinear Bochner-Riesz mean $\mathcal{B}^{\alpha}$ is defined by
\begin{align*}
    \mathcal{B}^{\alpha}(f,g)(x,u) &= \frac{1}{(2\pi)^{2 d_2}} \int_{\mathfrak{g}_{2,r}^{*}} \int_{\mathfrak{g}_{2,r}^{*}} e^{i \langle \lambda_1 + \lambda_2 , u \rangle} \sum_{\mathbf{k}_1, \mathbf{k}_2 \in \mathbb{N}^\Lambda} \left(1- \eta_{\mathbf{k}_1}^{\lambda_1}- \eta_{\mathbf{k}_2}^{\lambda_2} \right)_{+}^{\alpha} \\
    &\nonumber \hspace{1cm} \left[f^{\lambda_1} \times_{\lambda_1} \varphi_{\mathbf{k}_1}^{\mathbf{b}^{\lambda_1}, \mathbf{r}_1}(R_{\lambda_1}^{-1}\cdot) \right](x) \left[g^{\lambda_2} \times_{\lambda_2} \varphi_{\mathbf{k}_2}^{\mathbf{b}^{\lambda_2}, \mathbf{r}_2}(R_{\lambda_2}^{-1}\cdot) \right](x) \  d\lambda_1 \  d\lambda_2.
\end{align*}
Consequently, we can express the operator $\mathcal{B}^{\alpha}$ in terms of its kernel as
\begin{align}
\label{Equation: Kernel representation of Bochner-Riesz}
    \mathcal{B}^{\alpha}(f,g)(x,u) &= \int_G \int_G \mathcal{K}^{\alpha}((y,t)^{-1}(x,u), (z,s)^{-1}(x,u)) f(y,t) g(z,s) \, d(y, t) \, d(z, s) ,
\end{align}
where $\mathcal{K}^{\alpha}$ denotes the associated kernel of the  bilinear Bochner-Riesz kernel, given by
\begin{align}
\label{Kernel expression}
    \mathcal{K}^{\alpha}((y,t), (z,s)) &= \frac{1}{(2\pi)^{2 d_2}} \int_{\mathfrak{g}_{2,r}^{*}} \int_{\mathfrak{g}_{2,r}^{*}} e^{i \langle \lambda_1, t \rangle} e^{i \langle \lambda_2, s \rangle} \sum_{\mathbf{k}_1, \mathbf{k}_2 \in \mathbb{N}^\Lambda} \left(1- \eta_{\mathbf{k}_1}^{\lambda_1}- \eta_{\mathbf{k}_2}^{\lambda_2} \right)_{+}^{\alpha} \\
    &\nonumber \hspace{4cm} \times \varphi_{\mathbf{k}_1}^{\mathbf{b}^{\lambda_1}, \mathbf{r}_1}(R_{\lambda_1}^{-1} y) \, \varphi_{\mathbf{k}_2}^{\mathbf{b}^{\lambda_2}, \mathbf{r}_2}(R_{\lambda_2}^{-1} z) \ d\lambda_1 \  d\lambda_2 .
\end{align}

Let us set $m(\eta_1, \eta_2)= (1-\eta_1-\eta_2)_{+}^{\alpha}$. Also, let $\mathcal{L}_1 := \mathcal{L} \otimes I$ and $\mathcal{L}_2 := I \otimes \mathcal{L}$. It follows that the operators $\mathcal{L}_1$ and $\mathcal{L}_2$ commute strongly (see \cite[Lemma 7.24]{Konrad_Unbounded_Selfadjoint_operator_2012}). Then, bivariate spectral theorem (see \cite[Theorem 5.21]{Konrad_Unbounded_Selfadjoint_operator_2012}) allows us to consider the operator given by
\begin{align}
\label{Bivariate bochner-Riesz multiplier}
    & m(\mathcal{L}_1,\mathcal{L}_2)(f \otimes g)((x,u),(x',u'))
    = \frac{1}{(2\pi)^{2 d_2}} \int_{\mathfrak{g}_{2,r}^{*}} \int_{\mathfrak{g}_{2,r}^{*}} e^{i \langle \lambda_1 , u \rangle} e^{i \langle \lambda_2 , u' \rangle} \sum_{\mathbf{k}_1, \mathbf{k}_2 \in \mathbb{N}^N} m(\eta_{\mathbf{k}_1}^{\lambda_1}, \eta_{\mathbf{k}_2}^{\lambda_2}) \\
    &\nonumber \hspace{3cm} \times \left[f^{\lambda_1} \times_{\lambda_1} \varphi_{\mathbf{k}_1}^{\mathbf{b}^{\lambda_1}, \mathbf{r}_1}(R_{\lambda_1}^{-1}\cdot) \right](x) \left[g^{\lambda_2} \times_{\lambda_2} \varphi_{\mathbf{k}_2}^{\mathbf{b}^{\lambda_2}, \mathbf{r}_2}(R_{\lambda_2}^{-1}\cdot) \right](x') \  d\lambda_1 \  d\lambda_2 .
\end{align}

If we take $f, g \in \mathcal{S}(G)$, then it is straightforward to check that the above expression for $m(\mathcal{L}_1,\mathcal{L}_2)(f \otimes g)((x,u),(x',u'))$ is well-defined everywhere in $G \times G$. In fact, an application of Lebesgue dominated convergence theorem, shows that $m(\mathcal{L}_1,\mathcal{L}_2)(f \otimes g)((x,u),(x',u'))$ is continuous on $G \times G$. This implies that the restriction of $m(\mathcal{L}_1,\mathcal{L}_2)(f \otimes g)$ to the diagonal $\{((x,u), (x,u)) : (x,u) \in G\}$ is well-defined and $m(\mathcal{L}_1,\mathcal{L}_2)(f \otimes g)((x,u),(x,u))$ coincides with the bilinear Bochner-Riesz operator $\mathcal{B}^{\alpha}(f,g)(x,u)$.

Choose a non-negative function $\Psi \in C_c  ^{\infty}(\tfrac{1}{2}, 2)$ such that $\sum_{j \in \mathbb{Z}} \Psi(2^j t) = 1$ for $t> 0$. Then for $0\leq \eta_1, \eta_2 \leq 1$, we decompose the bilinear Bochner-Riesz multiplier as
\begin{align*}
    (1-\eta_1-\eta_2 )_{+}^{\alpha} = \sum_{j\in \mathbb{Z}} (1-\eta_1-\eta_2 )_+^{\alpha} \Psi \big(2^j (1 - \eta_1 - \eta_2)\big) = \sum_{j\in \mathbb{Z}} \Psi_j^{\alpha}(\eta_1, \eta_2) ,
\end{align*}
where 
\begin{align*}
    \Psi_j^{\alpha} (\eta_1, \eta_2) :=  (1-\eta_1-\eta_2 )_+^{\alpha} \Psi \big(2^j (1 - \eta_1 - \eta_2)\big) .
\end{align*}
Note that $\Psi_j^{\alpha} =0$ for $j <0$. Thus, for 
$f, g \in \mathcal{S}(G)$, based on  the above decomposition,   $\mathcal{B}^{\alpha}$ can be written as
\begin{align}
\label{Equation: Dyadic decomposition of Bochner-Riesz}
    \mathcal{B}^{\alpha} = \sum_{j=0}^{\infty} \mathcal{B}_j^{\alpha} ,
\end{align}
where
\begin{align}
\label{Expression after dyadic decomposition}
    \mathcal{B}_j^{\alpha}(f,g)(x,u) &= \frac{1}{(2\pi)^{2 d_2}} \int_{\mathfrak{g}_{2,r}^{*}} \int_{\mathfrak{g}_{2,r}^{*}} e^{i \langle \lambda_1 + \lambda_2 , u \rangle} \sum_{\mathbf{k}_1, \mathbf{k}_2 \in \mathbb{N}^\Lambda} \Psi_j^{\alpha}(\eta_{\mathbf{k}_1}^{\lambda_1}, \eta_{\mathbf{k}_2}^{\lambda_2}) \\
    &\nonumber \hspace{0.5cm} \times \left[f^{\lambda_1} \times_{\lambda_1} \varphi_{\mathbf{k}_1}^{\mathbf{b}^{\lambda_1}, \mathbf{r}_1}(R_{\lambda_1}^{-1}\cdot) \right](x) \left[g^{\lambda_2} \times_{\lambda_2} \varphi_{\mathbf{k}_2}^{\mathbf{b}^{\lambda_2}, \mathbf{r}_2}(R_{\lambda_2}^{-1}\cdot) \right](x) \  d\lambda_1 \  d\lambda_2 .
\end{align}

We have the following pointwise kernel estimate of $\mathcal{B}_j^{\alpha}$, which will be useful later in our proofs.
\begin{lemma}
\label{Lemma: Pointwise kernel estimate for Bj}
Let $\mathcal{K}_j^{\alpha}$ denote the kernel corresponding to the operator $\mathcal{B}_j^{\alpha}$. Then for all $\beta_1, \beta_2 \geq 0$ and $\epsilon>0$, we have 
    \begin{align*}
     \left|\mathcal{K}_j^{\alpha}((y,t),(z,s))\right| (1+\|(y,t)\|)^{\beta_1} (1+\|(z,s)\|)^{\beta_2} &\leq C\, 2^{j(\beta_1 +\beta_2 +\epsilon-\alpha-1/2)} ,
\end{align*}
for some constant $C>0$, independent of $j$.
\end{lemma}

\begin{proof}

Proof of Lemma 3.1. is a direct consequence of \cite[Theorem 2.7]{Martini_Lie_groups_Polynomial_Growth_2012}. Since $\mathcal{K}_j^{\alpha}$ is the convolution kernel of $\Psi_j^{\alpha}(\mathcal{L}_1, \mathcal{L}_2) = \widetilde{\Psi}_j^{\alpha}(\mathcal{L}_1 + \mathcal{L}_2)$ on $G \times G$, where $\widetilde{\Psi}_j^{\alpha}(s) = (1-s)_{+}^{\alpha} \Psi(2^j (1-s))$ and $\mathcal{L}_1 + \mathcal{L}_2$ is a homogeneous sub-Laplacian on a product of M\'etivier groups $G \times G$, from \cite[Theorem 2.7]{Martini_Lie_groups_Polynomial_Growth_2012} one deduce that, for all $\beta_1, \beta_2 \geq 0$ and $\epsilon>0$,
\begin{align*}
    \left|\mathcal{K}_j^{\alpha}((y,t),(z,s))\right| (1+\|(y,t)\|+\|(z,s)\|)^{\beta_1+\beta_2} &\leq C\, \|\widetilde{\Psi}_j^{\alpha}\|_{L^2_{\beta_1+\beta_2+\epsilon}} \leq C\, 2^{j(\beta_1 +\beta_2 +\epsilon-\alpha-1/2)} ,
\end{align*}
for some constant $C>0$ independent of $j$.
\end{proof}

Now we discuss weighted Plancherel estimate for the sub-Laplacian $\mathcal{L}$, which plays a significant role in our subsequent proofs. Recall that $X_1, \ldots, X_{d_1}, T_1, \ldots, T_{d_2}$ form an orthonormal basis for $\mathfrak{g}$, and the associated left-invariant sub-Laplacians is given by 
\begin{align*}
    \mathcal{L} = -(X_1^2 + \cdots + X_{d_1}^2) .
\end{align*}
The operators $\mathcal{L}, -i T_1, \ldots, -i T_{d_2}$ constitute a system of formally self-adjoint, left-invariant, and pairwise commuting differential operators; hence, they admit a joint functional calculus. Therefore, if we define $T := (-( T_1^2+ \cdots + T_{d_2}^2))^{1/2}$, then the operators $\mathcal{L}$ and $T$ also admit a joint functional calculus.

Let $\Theta : \mathbb{R} \to [0,1]$ be compactly supported smooth function such that it is supported in $[1/2, 2]$ and satisfies
\begin{align}
\label{Definition: Cutoff function theta}
    \sum_{M \in \mathbb{Z}} \Theta_{M}(\tau) =1 ,
\end{align}
where $\Theta_{M}(\tau) = \Theta(2^{M} \tau)$. Also, let $F : \mathbb{R} \to \mathbb{C}$ be a bounded Borel function supported in $ [0, 2]$. Then, for $M \in \mathbb{Z}$, we define $F_M : \mathbb{R} \times \mathbb{R} \to \mathbb{C}$ by 
\begin{align}
\label{Introducing theta in multiplier}
    F_M(\eta, \tau) = F(\eta) \Theta( 2^{M} \tau) .
\end{align}
Consequently, we can decompose $F(\mathcal{L})$ as
\begin{align}
\label{Use of Remark in Niedorf}
  F(\mathcal{L}) = \sum_{M= -\ell_0 }^{\infty} F_{M}(\mathcal{L}, T) ,
\end{align}
where $\ell_0 \in \mathbb{N}$ depends solely on $J_{\lambda}$ and the inner product on $\mathfrak{g}$. The fact that no terms with $M < -\ell_0$ contribute can be proved from an argument from \cite[Remark 5.2]{Niedorf_Restriction_Stratified_group_2025}. Indeed, recall that the functions $\lambda \mapsto b_n^{\lambda}$ are homogeneous of degree $1$, hence, $b_n^{\lambda} = |\lambda| b_n^{\Bar{\lambda}}$ where $\Bar{\lambda} = |\lambda|^{-1} \lambda$. Now as we have $\eta_{\mathbf{k}}^{\lambda} \in \supp{F}$ and $2^{M} |\lambda| \in \supp{\Theta}$, it follows that
\begin{align}
    \label{Inequality: Vanishing of multiplier for l small}
        1 \gtrsim \eta_{\mathbf{k}}^{\lambda} \geq \sum_{n=1}^{\Lambda} r_n b_n^{\lambda} = |\lambda| \sum_{n=1}^{\Lambda} r_n b_n^{\Bar{\lambda}} \sim 2^{-M} \sum_{n=1}^{\Lambda} r_n b_n^{\Bar{\lambda}} .
\end{align}
From \eqref{HilbertSchmidt norm on dual}, we can see that the summand $\sum_{n=1}^{\Lambda} r_n b_n^{\lambda}$ is non-zero for every $\lambda \neq 0$. Also, from Proposition \ref{Prop: Spectral decomposition of Jmu}, the maps $\lambda \mapsto b_n^{\lambda}$ are continuous on $\mathfrak{g}_2^{*}$. As $\Bar{\lambda} \in \{ \lambda \in \mathfrak{g}_2^{*} : |\lambda|=1\}$, from (\ref{Inequality: Vanishing of multiplier for l small}), we obtain $2^{-M} \lesssim 1$. Therefore, there exists $\ell_0 \in \mathbb{N}$ such that \eqref{Use of Remark in Niedorf} holds.

With the same notation introduced above, we now state the following weighted Plancherel estimate.
\begin{proposition}\cite[Proposition 6.1]{Niedorf_Metivier_group_2023}, \cite[Lemma 5.1]{Niedorf_Restriction_Stratified_group_2025}
\label{Proposition: First layer weighted Plancherel}
    Let $F : \mathbb{R} \to \mathbb{C}$ is a bounded Borel function supported in $ [0, 2]$ and $F_M$ be  defined as in \eqref{Introducing theta in multiplier}. Then the convolution kernel $\mathcal{K}_{F_{M}(\mathcal{L}, T)}$ of the  operator $F_{M}(\mathcal{L}, T)$ satisfies the estimate
    \begin{align}
    \label{First layer kernel estimate inside Prop}
        \int_G \left| |x|^{N} \mathcal{K}_{F_{M}(\mathcal{L}, T)}(x,u) \right|^2 \, d(x,u) &\leq C 2^{M(2 N-d_2)} \|F\|_{L^2(\mathbb{R})}^2 ,
    \end{align}
    for all $N \geq 0$.

    Moreover, we also have
    \begin{align}
    \label{Restriction estimate inside prop}
        \| F_{M}(\mathcal{L}, T)f\|_{L^2 } &\leq C 2^{-M d_2/2} \|F\|_{L^2(\mathbb{R})} \, \|f\|_{L^1} .
    \end{align}
\end{proposition}

With the help of \eqref{Use of Remark in Niedorf}, the following result can be easily deduced from Proposition \ref{Proposition: First layer weighted Plancherel}.
\begin{proposition}
\label{Proposition: First layer weighted Plancherel for limited alpha}
    If $F: \mathbb{R} \to \mathbb{C}$ is a bounded Borel function supported in $[0,2]$, then for all $0\leq \gamma<d_2/2$,
    \begin{align}
    \label{First layer kernel estimate in second prop}
        \int_G \left| |x|^{\gamma} \mathcal{K}_{F(\mathcal{L})}(x,u) \right|^2 \ d(x,u) &\leq C \|F\|_{L^2(\mathbb{R})}^2 .
    \end{align}
    In addition, we also have
    \begin{align}
    \label{First layer restriction estimate in second prop}
        \| F(\mathcal{L}) f\|_{L^2 } &\leq C \|F\|_{L^2(\mathbb{R})} \, \|f\|_{L^1} .
    \end{align}
\end{proposition}

\section{Proof of Theorem \ref{Bilinear Bochner-Riesz Main theorem}}
\label{Section: Dyadic decomposition of Bochner-Riesz}
We begin by recalling the following decomposition of $\mathcal{B}^{\alpha}$:
\begin{align*}
    \mathcal{B}^{\alpha} = \sum_{j=0}^{\infty} \mathcal{B}_j^{\alpha} ,
\end{align*}
with $\mathcal{B}_j^{\alpha}$ given by the expression  in \eqref{Expression after dyadic decomposition}. Therefore, in order to prove Theorem \ref{Bilinear Bochner-Riesz Main theorem}, it is enough to prove that, there exists some $\delta>0$ (depending on $\alpha$) such that for $f, g \in \mathcal{S}(G)$, the following inequality holds
\begin{align}
\label{Inequality: Estimate of Bj}
    \|\mathcal{B}^{\alpha}_{j}(f,g)\|_{L^p(G)} &\leq C 2^{-j \delta} \|f\|_{L^{p_1}(G)} \|g\|_{L^{p_2}(G)},
\end{align}
where $(p_1, p_2, p)$ satisfies $1/p=1/p_1+1/p_2$ and $1\leq p_1, p_2 \leq \infty$.

Let $\mathcal{K}_j^{\alpha}$ denote the kernel corresponding to the operator $\mathcal{B}_j^{\alpha}$ (see \eqref{Kernel expression}). Then for some fixed $\varepsilon>0$, we split the kernel $\mathcal{K}_j^{\alpha}$ as 
\begin{align}
\label{Decomposition kernel into four parts}
     \mathcal{K}_j^{\alpha} &= \mathcal{K}_{j,1}^{\alpha} + \mathcal{K}_{j,2}^{\alpha} + \mathcal{K}_{j,3}^{\alpha} + \mathcal{K}_{j,4}^{\alpha} , 
\end{align}
where
\begin{align}
\label{Expression: Kernel expression for Kj1}
      \mathcal{K}_{j,1}^{\alpha}((y,t), (z, s)) = \mathcal{K}_{j}^{\alpha}((y,t), (z,s))\  \chi_{B (0, 2^{j(1+\varepsilon)})} (y,t) \  \chi_{B (0, 2^{j(1+\varepsilon)})} (z,s),
\end{align}
\begin{align*}
      \mathcal{K}_{j,2}^{\alpha}((y,t), (z, s)) = \mathcal{K}_{j}^{\alpha}((y,t), (z,s))\  \chi_{B (0, 2^{j(1+\varepsilon)})} (y,t) \  \chi_{B (0, 2^{j(1+\varepsilon)})^c} (z,s),
\end{align*}
\begin{align*}
      \mathcal{K}_{j,3}^{\alpha}((y,t), (z, s)) = \mathcal{K}_{j}^{\alpha}((y,t), (z,s))\  \chi_{B (0, 2^{j(1+\varepsilon)})^c} (y,t) \  \chi_{B (0, 2^{j(1+\varepsilon)})} (z,s),
\end{align*} 
\begin{align*}
      \mathcal{K}_{j,4}^{\alpha}((y,t), (z, s)) = \mathcal{K}_{j}^{\alpha}((y,t), (z,s))\  \chi_{B (0, 2^{j(1+\varepsilon)})^c} (y,t) \  \chi_{B (0, 2^{j(1+\varepsilon)})^c} (z,s) .
\end{align*} 

For each $l=1, 2, 3, 4$, we denote the bilinear operator corresponding to the kernel $\mathcal{K}_{j,l}^{\alpha}$ by $\mathcal{B}^{\alpha}_{j,l}$. Now in order to prove \eqref{Inequality: Estimate of Bj}, again it is enough to prove that, there exists some $\delta>0$ such that for $f, g \in \mathcal{S}(G)$,
\begin{align}
\label{Inequality: Estimate of Bjl}
    \|\mathcal{B}^{\alpha}_{j,l}(f,g)\|_{L^p(G)} &\leq C 2^{-j \delta} \|f\|_{L^{p_1}(G)} \|g\|_{L^{p_2}(G)},
\end{align}
with $1/p=1/p_1+1/p_2$ and $1\leq p_1, p_2 \leq \infty$.

In the sequel, we only demonstrate how to establish \eqref{Inequality: Estimate of Bjl} when $l=1,3,4$. The estimate of $\mathcal{B}^{\alpha}_{j,2}$ is similar to that of $\mathcal{B}^{\alpha}_{j,3}$. Let us first start with the estimate of $\mathcal{B}^{\alpha}_{j,4}$.

\subsection{Estimate of \texorpdfstring{$\mathcal{B}^{\alpha}_{j,4}$}{}}
\label{Subsection: Estimate of Bj4}
Note that Lemma \ref{Lemma: Pointwise kernel estimate for Bj} also holds if we replace $\mathcal{K}_j^{\alpha}$ by $\mathcal{K}^{\alpha}_{j,l}$, for each $l=1,2,3,4$. Hence for any $N>0$ and $\epsilon_1>0$, applying Lemma \ref{Lemma: Pointwise kernel estimate for Bj} we can estimate $\mathcal{B}^{\alpha}_{j,4}$ as
\begin{align*}
    & |\mathcal{B}^{\alpha}_{j,4}(f,g)(x,u)| \\
    & \leq C 2^{j2N} \Big\{ \int_{G} \frac{|f(y,t)| \chi_{B (0, 2^{j(1+\varepsilon)})^c} ((y,t)^{-1}(x,u))} {{\big(1 + \|(y,t)^{-1}(x,u)\| \big)^N}} d(y,t) \Big\} \\
    & \hspace{6cm} \times \Big\{ \int_{G}  \frac{|g(z,s)| \chi_{B (0, 2^{j(1+\varepsilon)})^c}((z,s)^{-1}(x,u))} {{\big(1 + \|(z,s)^{-1}(x,u)\|\big)^N}} d(z,s) \Big\} \\
    & \leq C 2^{j2N} (|f| * k_1)(x,u) (|g| * k_1)(x,u),
\end{align*}
where $k_1(y,t)= \frac{\chi_{B (0, 2^{j(1+\varepsilon)})^c} (y,t)}{{(1 + \|(y,t)\|)^N}}$.

Using Lemma \ref{lemma: Estimate of distance on outside ball}, the $L^1$-norm of $k_1$ can be estimated as
\begin{align}
\label{Estimate of L1 norm of k1 factor}
    \|k_1\|_{L^1} &\leq C 2^{j(1+\varepsilon)(-N+Q)} .
\end{align}

Subsequently, using the above estimate, together with H\"older's inequality and Young's inequality, we obtain
\begin{align*}
    \|\mathcal{B}^{\alpha}_{j,4}(f,g)\|_{L^p} &\leq C 2^{j2N} \|k_1\|_{L^1}^2 \|f\|_{L^{p_1}} \|g\|_{L^{p_2}} \\
    &\leq C 2^{j2N} 2^{2 j(1+\varepsilon)(-N+Q)} \|f\|_{L^{p_1}} \|g\|_{L^{p_2}} \\
    &\leq C 2^{-j \delta} \|f\|_{L^{p_1}} \|g\|_{L^{p_2}} ,
\end{align*}
where $\delta=2\varepsilon N-2Q(1+\varepsilon)$. By choosing $N$ sufficiently large we can make $N \varepsilon>Q(1+\varepsilon)$ so that $\delta>0$.

\subsection{Estimate of \texorpdfstring{$\mathcal{B}^{\alpha}_{j,3}$}{}}
Using Lemma \ref{Lemma: Pointwise kernel estimate for Bj}, similarly to the estimate of $\mathcal{B}^{\alpha}_{j,4}$, it follows that for any $N>0$, 
\begin{align*}
    |\mathcal{B}^{\alpha}_{j,3}(f,g)(x,u)| & \leq C 2^{jN} (|f| * k_1)(x,u) (|g| * k_2)(x,u) ,
\end{align*}
where $k_1$ is as defined in the estimate of $\mathcal{B}^{\alpha}_{j,4}$ and $k_2(z,s)= \chi_{B (0, 2^{j(1+\varepsilon)})}(z,s)$.

We then proceed by applying H\"older's inequality, along with Young's inequality and \eqref{Estimate of L1 norm of k1 factor}, and deduce that 
\begin{align*}
    \|\mathcal{B}^{\alpha}_{j,3}(f,g)\|_{L^p} &\leq C 2^{jN} \|k_1\|_{L^1} \|f\|_{L^{p_1}} \|k_2\|_{L^1} \|g\|_{L^{p_2}} \\
    &\leq C 2^{jN} 2^{j(1+\varepsilon)(-N+Q)} \|f\|_{L^{p_1}} 2^{j(1+\varepsilon)Q} \|g\|_{L^{p_2}} \\
    &\leq C 2^{-j \delta} \|f\|_{L^{p_1}} \|g\|_{L^{p_2}} ,
\end{align*}
where $\delta=\varepsilon N-2Q(1+\varepsilon)$. Again by choosing $N$ sufficiently large we can make $N \varepsilon>2Q(1+\varepsilon)$ such that $\delta>0$.

\subsection{Estimate of \texorpdfstring{$\mathcal{B}^{\alpha}_{j,1}$}{}}
\label{subsection:Estimate of Bj1}
Let $\varepsilon>0$ be the same as the one chosen before the equation \eqref{Decomposition kernel into four parts}. We can choose a sequence $\{(a_n, b_n)\}_{n \in \mathbb{N}}$ such that $\varrho((a_{n_1}, b_{n_1}), (a_{n_2}, b_{n_2})) > \frac{2^{j(1+\varepsilon)}}{10}$ for $n_1 \neq n_2$ and $\sup_{(a, b) \in G} \inf_{n} \varrho((a,b), (a_{n},b_{n})) \leq \frac{2^{j(1+\varepsilon)}}{10}$. With the help of this sequence, we define the following disjoint sets given by
\begin{align}
\label{First decomposition into disjoint balls}
    S_{n}^j = \Bar{B}\left((a_n, b_n), \tfrac{2^{j(1+\varepsilon)}}{10}\right) \setminus \bigcup_{m < n} \Bar{B}\left((a_m, b_m), \tfrac{2^{j(1+\varepsilon)}}{10} \right) .
\end{align}
From \eqref{Expression: Kernel expression for Kj1}, for $(x,u) \in G$ we write
\begin{align*}
    \mathcal{K}^{\alpha}_{j,1}((y,t)^{-1}(x,u), (z,s)^{-1}(x,u)) =: \widetilde{ \mathcal{K}^{\alpha}_{j,1}}((x,u),(y,t),(z,s)) ,
\end{align*}
then we see that
\begin{align*}
    \supp{\widetilde{ \mathcal{K}^{\alpha}_{j,1}}} \subseteq \mathcal{D}_{j} := \{((x,u),(y,t),(z,s)) : \,  \varrho((x,u),(y,t)) \leq 2^{j(1+\varepsilon)}, \varrho((x,u),(z,s)) \leq 2^{j(1+\varepsilon)} \},
\end{align*}
which readily implies
\begin{align*}
    \mathcal{D}_{j} \subseteq \bigcup_{\substack{n,n_1,n_2: \varrho((a_n,b_n), (a_{n_1},b_{n_1}))\leq 2\cdot 2^{j(1+\varepsilon)}, \\ \varrho((a_n,b_n), (a_{n_2},b_{n_2}))\leq 2\cdot 2^{j(1+\varepsilon)}}} S_{n}^j \times (S_{n_1}^j \times S_{n_2}^j) .
\end{align*}
As a result, we can decompose $\mathcal{B}^{\alpha}_{j,1}$ as 
\begin{align}
\label{Equality: First decom of f and g}
    \mathcal{B}^{\alpha}_{j,1}(f,g)(x,u) &= \sum_{n=0}^{\infty} \sum_{\substack{n_1: \varrho((a_n,b_n) , (a_{n_1},b_{n_1}))\leq 2 \cdot 2^{j(1+\varepsilon)} \\ n_2: \varrho((a_n, b_n) , (a_{n_2},b_{n_2}))\leq 2 \cdot 2^{j(1+\varepsilon)}}} \chi_{S_{n}^j}(x,u) \mathcal{B}^{\alpha}_{j,1} (f_{n_1}^j ,g_{n_2}^j)(x,u) ,
\end{align}
where $f_{n_1}^j = f \chi_{S_{n_1}^j}$ and $g_{n_2}^j = g \chi_{S_{n_2}^j}$.

Now, we make the following claim for the operator $\mathcal{B}^{\alpha}_{j,1}$. For 
\begin{align*}
    \varrho((a_n,b_n) , (a_{n_1},b_{n_1}))\leq 2 \cdot 2^{j(1+\varepsilon)} \quad \text{and} \quad \varrho((a_n,b_n) , (a_{n_2},b_{n_2}))\leq 2 \cdot 2^{j(1+\varepsilon)} ,
\end{align*}
whenever $\alpha>\alpha(p_1, p_2)$, there exists $\delta>0$ such that
\begin{align}
\label{Claim for boundedness of Bj1}
     \| \chi_{S_{n}^j} \mathcal{B}^{\alpha}_{j,1} (f_{n_1}^j ,g_{n_2}^j)\|_{L^{p}(G)} \leq C 2^{-j \delta} \|f_{n_1}^j\|_{L^{p_1}(G)} \|g_{n_2}^j\|_{L^{p_2}(G)} ,
\end{align}
where $(p_1, p_2, p)$ satisfies $1/p=1/p_1+1/p_2$ and $1\leq p_1, p_2 \leq \infty$.

In this subsection, we complete the proof of the estimate of \eqref{Inequality: Estimate of Bjl} for $\mathcal{B}^{\alpha}_{j,1}$, under the assumption that claim \eqref{Claim for boundedness of Bj1} holds. Observe that for $n \neq m$, the balls $B((a_{n}, b_{n}), \frac{2^{j(1+\varepsilon)}}{20})$ and $ B((a_{m}, b_{m}), \frac{2^{j(1+\varepsilon)}}{20})$ are disjoint. Therefore, we have the following bounded overlapping property,
\begin{align}
\label{Bounded overlapping property for j ball}
   \sup_{n} \#\{m : \varrho((a_{n}, b_{n}), (a_{m}, b_{m})) \leq  2 \cdot 2^{j(1+\varepsilon)}\} \leq C .
\end{align}
Since the sets $S_n^j$ are disjoint and applying bounded overlapping property, it follows from (\ref{Equality: First decom of f and g}) that
\begin{align*}
    &\|\mathcal{B}^{\alpha}_{j,1} (f,g)\|_{L^{p}(G)}^p = \Big\|\sum_{n=0}^{\infty} \sum_{\substack{n_1: \varrho((a_n,b_n) , (a_{n_1},b_{n_1}))\leq 2 \cdot 2^{j(1+\varepsilon)} \\ n_2: \varrho((a_n, b_n) , (a_{n_2},b_{n_2}))\leq 2 \cdot 2^{j(1+\varepsilon)}}} \chi_{S_{n}^j} \mathcal{B}^{\alpha}_{j,1} (f_{n_1}^j ,g_{n_2}^j) \Big\|_{L^{p}(G)}^p \\
    &\nonumber \leq C \sum_{n=0}^{\infty} \sum_{\substack{n_1: \varrho((a_n,b_n) , (a_{n_1},b_{n_1}))\leq 2 \cdot 2^{j(1+\varepsilon)} \\ n_2: \varrho((a_n, b_n) , (a_{n_2},b_{n_2}))\leq 2 \cdot 2^{j(1+\varepsilon)}}} \| \chi_{S_{n}^j} \mathcal{B}^{\alpha}_{j,1} (f_{n_1}^j ,g_{n_2}^j) \|_{L^{p}(G)}^{p} .
\end{align*}
Consequently, invoking the claim \eqref{Claim for boundedness of Bj1}, the above expression can be dominated by
\begin{align*}
    C 2^{-j p \delta } \sum_{n=0}^{\infty} \Big\{ \sum_{n_1: \varrho((a_n,b_n) , (a_{n_1},b_{n_1}))\leq 2 \cdot 2^{j(1+\varepsilon)}} \|f_{n_1}^j\|_{L^{p_1}(G)}^{p} \Big\} \Big\{ \sum_{n_2: \varrho((a_n, b_n) , (a_{n_2},b_{n_2}))\leq 2 \cdot 2^{j(1+\varepsilon)}} \|g_{n_2}^j\|_{L^{p_2}(G)}^{p} \Big\} .
\end{align*}
Since $1=p/p_1+p/p_2$, applying H\"older's inequality with respect to the sums over $n_1,n_2$ and $n$ respectively and again using bounded overlapping property, the above expression can be controlled by
\begin{align*}
    & C 2^{-j p \delta } \Big\{\sum_{n=0}^{\infty} \sum_{n_1: \varrho((a_n,b_n) , (a_{n_1},b_{n_1}))\leq 2 \cdot 2^{j(1+\varepsilon)}} \|f_{n_1}^j\|_{L^{p_1}(G)}^{p_1} \Big\}^{\frac{p}{p_1}} \\
    &\nonumber \hspace{6cm} \Big\{\sum_{n=0}^{\infty} \sum_{n_2: \varrho((a_n, b_n) , (a_{n_2},b_{n_2}))\leq 2 \cdot 2^{j(1+\varepsilon)}} \|g_{n_2}^j\|_{L^{p_2}(G)}^{p_2} \Big\}^{\frac{p}{p_2}} \\
    &\nonumber \leq C 2^{-j p \delta} \|f\|_{L^{p_1}(G)}^p \|g\|_{L^{p_2}(G)}^p .
\end{align*}
This completes the proof of the inequality \eqref{Inequality: Estimate of Bjl} for $\mathcal{B}^{\alpha}_{j,1}$, upon assuming the claim.

It remains to prove the claim stated in \eqref{Claim for boundedness of Bj1}. First we note that via an argument based on bilinear interpolation using real method \cite{Grafakos_Liu_Lu_Zhao_Marcinkiewicz_interpolation_2012},  explained in detail in \cite[Section 4.3]{Bernicot_Grafakos_Song_Yan_Bilinear_Bochner_Riesz_2015}, it is enough to verify the claim for $(p_1, p_2,p)=(2,2,1)$, $(1,1,2)$, $(1,2,2/3)$, $(2,1,2/3)$, $(1,\infty,1)$, $(\infty,1,1)$,  $(2,\infty,2)$, $(\infty,2,2)$ and $(\infty, \infty, \infty)$. Furthermore, by inter changing the role of input functions $f$ and $g$, we may exclude the cases when $(p_1, p_2,p)=$ $(2,1,2/3)$, $(\infty,1,1)$,  and $(\infty,2,2)$.


For each $n \in \mathbb{N}$, let us denote $S_{n,0}^j := (a_n, b_n)^{-1} S_n^j$, $S_{n_1,0}^j := (a_n, b_n)^{-1} S_{n_1}^j$ and $S_{n_2,0}^j := (a_n, b_n)^{-1} S_{n_2}^j$. Then we can easily see that
\begin{align*}
    \| \chi_{S_{n}^j} \mathcal{B}^{\alpha}_{j,1} (f_{n_1}^j ,g_{n_2}^j) \|_{L^{p}(G)} &= \| \chi_{S_{n,0}^j} \mathcal{B}^{\alpha}_{j,1} (f_{n_1,0}^j ,g_{n_2,0}^j) \|_{L^{p}(G)} ,
\end{align*}
where $f_{n_1,0}^j = f \chi_{S_{n_1,0}^j}$ and $g_{n_2,0}^j = g \chi_{S_{n_2,0}^j}$.

In view of left-invariance of $\varrho$, the claim \eqref{Claim for boundedness of Bj1} is further reduced to showing that: for 
\begin{align}
\label{Condition for all claims}
    \varrho((a_{n_1},b_{n_1}) , 0)\leq 2 \cdot 2^{j(1+\varepsilon)} \quad \text{and} \quad \varrho((a_{n_2},b_{n_2}) , 0)\leq 2 \cdot 2^{j(1+\varepsilon)} ,
\end{align}
whenever $\alpha>\alpha(p_1, p_2)$, there exists a $\delta>0$ such that
\begin{align}
\label{Reduced claim for Bj1}
     \| \chi_{S_{n,0}^j} \mathcal{B}^{\alpha}_{j,1} (f_{n_1,0}^j ,g_{n_2,0}^j) \|_{L^{p}(G)} \leq C 2^{-j \delta} \|f_{n_1,0}^j\|_{L^{p_1}(G)} \|g_{n_2,0}^j\|_{L^{p_2}(G)} ,
\end{align}
for $(p_1,p_2,p) \in \{(1,1,1/2),(1,2,2/3),(2,2,1),(1,\infty,1),(2,\infty,2),(\infty, \infty, \infty) \}$.

Note that $S_{n,0}^j \subseteq B(0, 2^{j(1+\varepsilon)}/10)$. From \eqref{Condition for all claims} we can easily see that
\begin{align*}
    S_{n_1,0}^j \subseteq B(0, 3 \cdot 2^{j(1+\varepsilon)}) \quad \text{and} \quad S_{n_2,0}^j \subseteq B(0, 3 \cdot 2^{j(1+\varepsilon)}) .
\end{align*}
Hence all the sets $S_{n,0}^j$, $S_{n_1,0}^j$ and $S_{n_2,0}^j$ are contained in a ball $B_j=B(0, 3 \cdot 2^{j(1+\varepsilon)})$ centered at origin and of radius $2^{j(1+\varepsilon)}$. Therefore in order to prove \eqref{Reduced claim for Bj1}, it is enough to prove the following:
\begin{align*}
    \|\chi_{B_j} \mathcal{B}^{\alpha}_{j,1}(\mathfrak{F}_j, \mathfrak{G}_j)\|_{L^p(G)} &\leq C 2^{-j \delta} \|\mathfrak{F}_j\|_{L^{p_1}(G)} \|\mathfrak{G}_j\|_{L^{p_2}(G)} \quad \text{whenever} \ \supp{\mathfrak{F}_j}, \  \supp{\mathfrak{G}_j} \subseteq B_j .
\end{align*}
Let us write
\begin{align*}
    \chi_{B_j}(x,u) \mathcal{B}^{\alpha}_{j,1}(\mathfrak{F}_j, \mathfrak{G}_j)(x,u) &= \chi_{B_j}(x,u) \mathcal{B}^{\alpha}_{j}(\mathfrak{F}_j, \mathfrak{G}_j)(x,u) + \chi_{B_j}(x,u) (\mathcal{B}^{\alpha}_{j,1}-\mathcal{B}^{\alpha}_{j})(\mathfrak{F}_j, \mathfrak{G}_j)(x,u) .
\end{align*}
Now note that
\begin{align*}
    & \chi_{B_j}(x,u) (\mathcal{B}^{\alpha}_{j,1}-\mathcal{B}^{\alpha}_{j})(\mathfrak{F}_j, \mathfrak{G}_j)(x,u) \\
    &= \chi_{B_j}(x,u) \int_G \int_G (\mathcal{K}^{\alpha}_{j,1}-\mathcal{K}^{\alpha}_{j})((y,t)^{-1}(x,u), (z,s)^{-1}(x,u)) \mathfrak{F}_j(y,t) \mathfrak{G}_j(z,s) \, d(y, t) \, d(z, s) \\
    &= \chi_{B_j}(x,u) \int_G \int_G \widetilde{K}_{j}((y,t)^{-1}(x,u), (z,s)^{-1}(x,u)) \mathfrak{F}_j(y,t) \mathfrak{G}_j(z,s) \, d(y, t) \, d(z, s),
\end{align*}
where
\begin{align*}
    \supp{\widetilde{K}_{j}}((y,t),(z,s)) &\subseteq B(0, 6 \cdot 2^{j(1+\varepsilon)}) \times (B(0, 6 \cdot 2^{j(1+\varepsilon)}) \setminus B(0, 2^{j(1+\varepsilon)})) \\
    &\hspace{2cm} \cup (B(0, 6 \cdot 2^{j(1+\varepsilon)}) \setminus B(0, 2^{j(1+\varepsilon)})) \times B(0, 6 \cdot 2^{j(1+\varepsilon)}) .
\end{align*}
Therefore similarly as in the estimate of $\mathcal{B}^{\alpha}_{j,2}$ and $\mathcal{B}^{\alpha}_{j,3}$ we get
\begin{align*}
    \|\chi_{B_j} (\mathcal{B}^{\alpha}_{j,1}-\mathcal{B}^{\alpha}_{j})(\mathfrak{F}_j, \mathfrak{G}_j)\|_{L^{p}(G)} &\leq C 2^{-j \delta} \|\mathfrak{F}_j\|_{L^{p_1}(G)} \|\mathfrak{G}_j\|_{L^{p_2}(G)} .
\end{align*}
Hence it suffices prove the following estimate:
\begin{align}
\label{Reduction to the same support ball}
    \|\chi_{B_j} \mathcal{B}^{\alpha}_{j}(\mathfrak{F}_j, \mathfrak{G}_j)\|_{L^p(G)} &\leq C 2^{-j \delta} \|\mathfrak{F}_j\|_{L^{p_1}(G)} \|\mathfrak{G}_j\|_{L^{p_2}(G)} \quad \text{whenever} \ \supp{\mathfrak{F}_j}, \  \supp{\mathfrak{G}_j} \subseteq B_j ,
\end{align}

\medskip
Now over the next several sections, our goal is to establish the claim stated in (\ref{Reduction to the same support ball}) for the points $(p_1,p_2,p)= (2,2,1)$ and $(1, \infty, 1)$. And for the remaining points $(p_1,p_2,p)=(1,1,1/2)$, $(1,2,2/3)$, $(2,2,1)$ and $(\infty, \infty, \infty)$ instead of proving the claim \eqref{Reduction to the same support ball}, we prove the estimate \eqref{Inequality: Estimate of Bj}.

\section{Proof of the claim (\ref{Inequality: Estimate of Bj}) at \texorpdfstring{$(p_1,p_2,p)=(1,1,1/2)$}{}{}}
\label{Section: Claim A}

This section is devoted to proving the claim \eqref{Inequality: Estimate of Bj} for the point $(p_1,p_2,p)=(1,1,1/2)$. Note that $\alpha(1,1)= d+1$. In the Euclidean setting, the kernel expression of the bilinear Bochner-Riesz means $B^{\alpha}_R$ is explicitly known and can be explicitly expressed in terms of Bessel functions. This fact has been exploited in the work of Bernicot et al. (see \cite[Proposition 4.2 (i)]{Bernicot_Grafakos_Song_Yan_Bilinear_Bochner_Riesz_2015}) to get the boundedness of $B^{\alpha}_R$ for $\alpha>n-1/2$, where $n$ is the Euclidean dimension. In contrast to  the case of M\'etivier groups, an explicit kernel representation of the bilinear Bochner-Riesz operator associated with the sub-Laplacian is not known. As a result, establishing the estimate of $\mathcal{B}^{\alpha}_{j,1}$ for $\alpha>d+1$ at the point $(1,1,1/2)$ becomes more delicate. In order to get the required estimate, here we draw upon some ideas from \cite{Niedorf_Metivier_group_2023}, where the author studied the $p$-specific Bochner-Riesz multiplier. However, in bilinear set-up the proofs are more technical and require additional adaptations.

From \eqref{Expression after dyadic decomposition}, we have
\begin{align*}
    & \mathcal{B}_{j}^{\alpha}(f,g)(x,u) = \frac{1}{(2\pi)^{2 d_2}} \int_{\mathfrak{g}_{2,r}^{*}} \int_{\mathfrak{g}_{2,r}^{*}} e^{i \langle \lambda_1 + \lambda_2 , u \rangle} \sum_{\mathbf{k}_1, \mathbf{k}_2 \in \mathbb{N}^N} \Psi_j^{\alpha}(\eta_{\mathbf{k}_1}^{\lambda_1}, \eta_{\mathbf{k}_2}^{\lambda_2}) \\
    & \hspace{2cm} \times \left[f^{\lambda_1} \times_{\lambda_1} \varphi_{\mathbf{k}_1}^{\mathbf{b}^{\lambda_1}, \mathbf{r}}(R_{\lambda_1}^{-1}\cdot) \right](x) \left[g^{\lambda_2} \times_{\lambda_2} \varphi_{\mathbf{k}_2}^{\mathbf{b}^{\lambda_2}, \mathbf{r}}(R_{\lambda_2}^{-1}\cdot) \right](x) \  d\lambda_1 \  d\lambda_2 .
\end{align*}
We fix $\eta_1= \eta_{\mathbf{k}_1}^{\lambda_1}$, and view $\Psi_j^{\alpha}(\eta_{1}, \cdot)$ as a function of the second variable, supported on $[0,1]$ and vanishing at $1$. Thus, for any fix $\eta_1$, the function $\eta_2 \to \Psi_j^{\alpha}(\eta_{1}, \eta_2)$ is a smooth function on $\mathbb{R}$ supported in $(-\infty, 1)$. Let $\Tilde{\chi} \in C_c^{\infty}(\mathbb{R})$ such that it equals to $1$ on $[-1,1]$ and $0$ outside  $[-2,2]$. Then the function $\eta_2 \to \Tilde{\chi}(\eta_2) \Psi_j^{\alpha}(\eta_{1}, \eta_2)$ become a smooth function on $\mathbb{R}$, supported on $(-2,2)$ and coincides with the function $\eta_2 \to \Psi_j^{\alpha}(\eta_{1}, \eta_2)$ on $[0, \infty)$. Subsequently, we extend the function $\eta_2 \to \Tilde{\chi}(\eta_2) \Psi_j^{\alpha}(\eta_{1}, \eta_2)$ periodically to $\mathbb{R}$ as a $4$-periodic function. Hence, we can expand it into a Fourier series as
\begin{align}
\label{Fourier series decomposition with chi tilde}
    \Tilde{\chi}(\eta_2) \Psi_j^{\alpha}(\eta_1, \eta_{2}) &= \sum_{l \in \mathbb{Z}} \phi_{j,l}^{\alpha}(\eta_1) e^{i \pi l \eta_{2}/2} ,
\end{align}
where $\phi_{j,l}^{\alpha}$ for $l \in \mathbb{Z}$, is given by $\displaystyle{\phi_{j,l}^{\alpha}(\eta_1)= \frac{1}{4} \int_{-2}^2 \Tilde{\chi}(\eta_2) \Psi_j^{\alpha}(\eta_1, \eta_{2}) e^{-i \pi l \eta_{2}/2}\, d\eta_{2}}$. It then follows that $\phi_{j,l}^{\alpha}$ satisfies the following estimate:
\begin{align}
\label{Convergence of sum of l using phi}
\sup_{\eta_1 \in [0,1]} |\phi_{j,l}^{\alpha}(\eta_1)| (1 + |l|)^{1 + \beta} \leq C 2^{-j \alpha} 2^{j \beta} \quad \text{for all} \quad \beta \geq 0 .
\end{align}
Since $\Tilde{\chi}(\eta_2) \Psi_j^{\alpha}(\eta_{1}, \eta_2)$ coincides with $\Psi_j^{\alpha}(\eta_{1}, \eta_2)$ on $[0, \infty)$ as a function of $\eta_2$, from \eqref{Fourier series decomposition with chi tilde} for $\eta_1, \eta_2 \in [0, \infty)$ we can write
\begin{align}
\label{Fourier series decomposition}
    \Psi_j^{\alpha}(\eta_1, \eta_{2}) &= \sum_{l \in \mathbb{Z}} \phi_{j,l}^{\alpha}(\eta_1) e^{i \pi l \eta_{2}/2} \Tilde{\chi}(\eta_2) .
\end{align}
The above expansion \eqref{Fourier series decomposition} of $\Psi_j^{\alpha}$, therefore leads us to the following representation of  $\mathcal{B}^{\alpha}_{j,1}$.
\begin{align}
\label{Equation: Decomposition in terms of Fourier series}
    \mathcal{B}_{j}^{\alpha}(f ,g)(x,u) &= C \sum_{l \in \mathbb{Z}} \Big\{\int_{\mathfrak{g}_{2,r}^{*}} e^{i \langle \lambda_1, u \rangle} \sum_{\mathbf{k}_1 \in \mathbb{N}^N} \phi_{j,l}^{\alpha}(\eta_{\mathbf{k}_1}^{\lambda_1}) \Big[f^{\lambda_1} \times_{\lambda_1} \varphi_{\mathbf{k}_1}^{\mathbf{b}^{\lambda_1}, \mathbf{r}}(R_{\lambda_1}^{-1}\cdot) \Big](x) \ d\lambda_1 \Big\} \\
    &\nonumber \hspace{1cm} \Big\{\int_{\mathfrak{g}_{2,r}^{*}} e^{i \langle \lambda_2, u \rangle} \sum_{\mathbf{k}_2 \in \mathbb{N}^N} e^{i \pi l \eta_{\mathbf{k}_2}^{\lambda_2}/2} \Tilde{\chi}(\eta_{\mathbf{k}_2}^{\lambda_2}) \Big[g^{\lambda_2} \times_{\lambda_2} \varphi_{\mathbf{k}_2}^{\mathbf{b}^{\lambda_2}, \mathbf{r}}(R_{\lambda_2}^{-1}\cdot) \Big](x) \ d\lambda_2 \Big\} \\
    &\nonumber= C \sum_{l \in \mathbb{Z}} \left\{ \phi_{j,l}^{\alpha}(\mathcal{L})f(x,u) \right\} \left\{ \psi_l(\mathcal{L})g(x,u) \right\} ,
\end{align}
where $\psi_l(\eta_2) := e^{i \pi l \eta_2/2} \Tilde{\chi}(\eta_2)$.

Note that for $0<p<1$, $\|\cdot\|_{L^p}^p$ satisfies the following estimate,
\begin{align}
\label{Triangle inequality for p less than 1 case}
    \|f+g\|_{L^p}^p \leq \|f\|_{L^p}^p + \|g\|_{L^p}^p .
\end{align}
Using this fact along with H\"older's inequality yields
    \begin{align}
    \label{Application of triangle for p less tha 1 case}
        \|\mathcal{B}_j^{\alpha}(f,g)\|_{L^{1/2}} &\leq C \left(\sum_{l \in \mathbb{Z}} \|\phi_{j,l}^{\alpha}(\mathcal{L})f\|_{L^1}^{1/2} \|\psi_l(\mathcal{L})g\|_{L^1}^{1/2} \right)^{2} .
    \end{align}
    Let $\mathcal{K}_{\phi_{j,l}^{\alpha}(\mathcal{L})}$ and $\mathcal{K}_{\psi_l(\mathcal{L})}$ denote the convolution kernel of $\phi_{j,l}^{\alpha}(\mathcal{L})$ and $\psi_l(\mathcal{L})$ respectively. Then from \eqref{Application of triangle for p less tha 1 case} applying Young's inequality we obtain
    \begin{align*}
        \|\mathcal{B}_j^{\alpha}(f,g)\|_{L^{1/2}} &\leq C \left(\sum_{l \in \mathbb{Z}} \|\mathcal{K}_{\phi_{j,l}^{\alpha}(\mathcal{L})}\|_{L^1}^{1/2} \|\mathcal{K}_{\psi_l(\mathcal{L})}\|_{L^1}^{1/2} \right)^{2} \|f\|_{L^1} \|g\|_{L^1} .
    \end{align*}
    Since $\mathcal{L}$ is a sub-Laplacian on M\'etivier groups $G$, from \cite[Theorem 3.11, Proposition 3.9]{Martini_Lie_groups_Polynomial_Growth_2012} for any $s>\dim G/2=d/2$ we get
    \begin{align*}
        \|\mathcal{K}_{\phi_{j,l}^{\alpha}(\mathcal{L})}\|_{L^1} \lesssim \|\phi_{j,l}^{\alpha}\|_{L_s^2(\mathbb{R})} \quad \text{and} \quad \|\mathcal{K}_{\psi_l(\mathcal{L})}\|_{L^1} \lesssim \|\psi_l\|_{L_s^2(\mathbb{R})} .
    \end{align*}
    Note that similar to (6.2), we also have
    \begin{align*}
        \|\phi_{j,l}^{\alpha}\|_{L_s^2(\mathbb{R})} \lesssim (1+|l|)^{-1-\beta} 2^{j(\beta+s-\alpha)} \quad \text{and} \quad \|\psi_l\|_{L_s^2(\mathbb{R})} \lesssim (1+|l|)^s ,
    \end{align*}
    for any $\beta>0$.

    Hence combining all the above estimates we obtain
    \begin{align*}
        \|\mathcal{B}_j^{\alpha}(f,g)\|_{L^{1/2}} &\leq C 2^{j(\beta+s-\alpha)} \left(\sum_{l \in \mathbb{Z}} (1+|l|)^{(s-1-\beta)/2} \right)^{2} \|f\|_{L^1} \|g\|_{L^1} .
    \end{align*}
    Note that if $\beta>1+s$, then the above series in $l$ will converge. Since $\alpha>d+1$ and $\beta>1+s$, so that $\beta+s>1+2s$ and as $s>d/2$, we can choose $\delta= \alpha-\beta-s>0$ such that
    \begin{align*}
        \|\mathcal{B}_j^{\alpha}(f,g)\|_{L^{1/2}} &\lesssim 2^{-j \delta} \|f\|_{L^1} \|g\|_{L^1} .
    \end{align*}
    This completes the proof of the claim \eqref{Inequality: Estimate of Bj} for the point $(p_1,p_2,p)=(1,1,1/2)$. \qed

\section{Proof of the claim (\ref{Inequality: Estimate of Bj}) at \texorpdfstring{$(p_1,p_2,p)=(1,2,2/3)$}{}{}} 
\label{Proof at (1,2,2/3) for claim A}

In this case $\alpha(1,2)=(d+1)/2$. The idea of the proof is similar to the argument used  in the estimate for $(p_1, p_2,p)=(1,1,1/2)$. Similarly as in \eqref{Application of triangle for p less tha 1 case}, here also we get
    \begin{align*}
        \|\mathcal{B}_j^{\alpha}(f,g)\|_{L^{2/3}} &\leq C \left(\sum_{l \in \mathbb{Z}} \|\phi_{j,l}^{\alpha}(\mathcal{L})f\|_{L^1}^{2/3} \|\psi_l(\mathcal{L})g\|_{L^2}^{2/3} \right)^{3/2} .
    \end{align*}
    Now using Young's convolution inequality we obtain
    \begin{align}
    \label{Application of Young for 1 2 case}
        \|\mathcal{B}_j^{\alpha}(f,g)\|_{L^{2/3}} &\leq C \left(\sum_{l \in \mathbb{Z}} \|\mathcal{K}_{\phi_{j,l}^{\alpha}(\mathcal{L})}\|_{L^1}^{2/3} \|\psi_l\|_{L^{\infty}}^{2/3} \right)^{3/2} \|f\|_{L^1} \|g\|_{L^2} .
    \end{align}
    Again similar to the case $(1,1,1/2)$, since $\mathcal{L}$ is a sub-Laplacian on M\'etivier groups $G$, from \cite[Theorem 3.11, Proposition 3.9]{Martini_Lie_groups_Polynomial_Growth_2012} for any $s>\dim G/2=d/2$ we get
    \begin{align}
    \label{Similar estimate of previous for phijl}
        \|\mathcal{K}_{\phi_{j,l}^{\alpha}(\mathcal{L})}\|_{L^1} \lesssim \|\phi_{j,l}^{\alpha}\|_{L_s^2(\mathbb{R})} \lesssim (1+|l|)^{-1-\beta} 2^{j(\beta+s-\alpha)} ,
    \end{align}
    for any $\beta>0$.

    Hence putting the estimate \eqref{Similar estimate of previous for phijl} into \eqref{Application of Young for 1 2 case} we have
    \begin{align*}
        \|\mathcal{B}_j^{\alpha}(f,g)\|_{L^{2/3}} &\leq C 2^{j(\beta+s-\alpha)} \left(\sum_{l \in \mathbb{Z}} (1+|l|)^{-(1+\beta)2/3} \right)^{3/2} \|f\|_{L^1} \|g\|_{L^2} .
    \end{align*}
    Now the above series in $l$ will converge provided $\beta>1/2$. Therefore we obtain
    \begin{align*}
        \|\mathcal{B}_j^{\alpha}(f,g)\|_{L^{2/3}} &\leq C 2^{-j \delta} \|f\|_{L^1} \|g\|_{L^2} ,
    \end{align*}
    where we choose $\alpha>(d+1)/2$, $s>d/2$ and $\beta>1/2$ such that $\delta=\alpha-\beta-s>0$.

This completes the proof of the claim \eqref{Inequality: Estimate of Bj} for the point $(p_1,p_2,p)=(1,2,2/3)$. \qed

\section{Proof of the claim (\ref{Inequality: Estimate of Bj}) at \texorpdfstring{$(p_1,p_2,p)=(2,2,1)$}{}{}}

Note that here we have $\alpha(2, 2)=0$. Using the decomposition from \eqref{Equation: Decomposition in terms of Fourier series} and applying Cauchy-Schwartz inequality, the fact from (\ref{Convergence of sum of l using phi}) we obtain
\begin{align*}
   \|\mathcal{B}^{\alpha}_{j}(f, g)\|_{L^1} &\leq C \sum_{l \in \mathbb{Z}} \| \phi_{j,l}^{\alpha}(\mathcal{L})f\|_{L^2} \| \psi_l(\mathcal{L})g \|_{L^2} \\
   &\leq C \sum_{l \in \mathbb{Z}} \frac{1}{(1+|l|)^{(1+\varepsilon)}} \{(1 + |l|)^{1 + \varepsilon} \|\phi_{j,l}^{\alpha}\|_{L^{\infty}}\} \|f\|_{L^{2}} \|g\|_{L^{2}} \\
   &\leq C 2^{-j\alpha} 2^{j \varepsilon} \|f\|_{L^{2}} \|g\|_{L^{2}} \\
   &\leq C 2^{-j \delta} \|f\|_{L^{2}} \|g\|_{L^{2}} .
\end{align*}
where $\delta=\alpha-\varepsilon>0$, since $\alpha > 0$ we can choose $\varepsilon>0$ so small such that $\alpha - \varepsilon>0$.

This completes the proof of the claim \eqref{Inequality: Estimate of Bj} for the point $(p_1,p_2,p)=(2,2,1)$. \qed

\section{Proof of the claim \ref{Inequality: Estimate of Bj} at \texorpdfstring{$(p_1,p_2,p)=(\infty,\infty,\infty)$}{}{}}
\label{section; Proof of claim C at infinity}

It is worth noting that in the Euclidean context, similar to the case $(1,1,1/2)$, the boundedness of bilinear Bochner-Riesz means $B^{\alpha}_R$ at $(\infty, \infty, \infty)$ is a consequence of the explicit expression of corresponding bilinear kernel of $B^{\alpha}_R$ and of H\"older's inequality. However, the present case requires a different approach. Note that $\alpha(\infty, \infty)=d-1/2$.


Let $\mathcal{K}_j^{\alpha}$ denote the kernel corresponding to the operator $\mathcal{B}_j^{\alpha}$, that is if we set $F=(f,g)$ then we have
    \begin{align*}
        \mathcal{B}_j^{\alpha}(f,g)(x,u) &= \mathcal{B}_j^{\alpha}F(x,u) \\
        &= (F * \mathcal{K}_j^{\alpha}) ((x_1,u_1), (x_2,u_2))|_{(x_1,u_1)=(x_2,u_2)=(x,u)} .
    \end{align*}
    Therefore $\mathcal{K}_j^{\alpha}$ is the convolution kernel of $\Psi_j^{\alpha}(\mathcal{L}_1, \mathcal{L}_2) = \widetilde{\Psi}_j^{\alpha}(\mathcal{L}_1 + \mathcal{L}_2)$ on $G \times G$, where $\widetilde{\Psi}_j^{\alpha}(s) = (1-s)_{+}^{\alpha} \Psi(2^j (1-s))$ with $\Psi \in C_c^{\infty}(1/2,2)$.

    Since $\mathcal{L}_1 + \mathcal{L}_2$ is a homogeneous sub-Laplacian on a product of M\'etivier groups $G \times G$, from \cite[Theorem 3.2]{Hebisch_Product_Generalized_Heisenberg_1996} and a standard argument (see \cite[Theorem 3.11, Proposition 3.9]{Martini_Lie_groups_Polynomial_Growth_2012}) for any $s>\dim(G \times G)/2=d$ we have
    \begin{align}
    \label{Application of hebisch result}
        \|\mathcal{K}_j^{\alpha}\|_{L^1(G \times G)} \lesssim \|\widetilde{\Psi}_j^{\alpha}\|_{L_s^2(\mathbb{R})} \lesssim 2^{j(s-\alpha-1/2)} .
    \end{align}
    
   On the other hand for $1\leq p_1, p_2, p\leq \infty$, similarly as in \cite[Lemma 2.1]{Bernicot_Grafakos_Song_Yan_Bilinear_Bochner_Riesz_2015}, an application of Minkowski's integral inequality and H\"older's inequality implies that
    \begin{align}
    \label{Application of Minkowski and Holder}
        \|\mathcal{B}_j^{\alpha}(f,g)\|_{L^p} &\leq \|\mathcal{K}_j^{\alpha}\|_{L^1(G \times G)} \|f\|_{L^{p_1}} \|g\|_{L^{p_2}} .
    \end{align}
    Hence from \ref{Application of hebisch result} and \eqref{Application of Minkowski and Holder} we obtain
    \begin{align*}
        \|\mathcal{B}_j^{\alpha}(f,g)\|_{L^p} &\lesssim 2^{-j\delta} \|f\|_{L^{p_1}} \|g\|_{L^{p_2}} ,
    \end{align*}
    where since $\alpha>d-1/2$ we can choose $s>d$ such that $\delta=\alpha+1/2-s>0$.

This completes the proof of the claim \eqref{Inequality: Estimate of Bj} for the point $(p_1,p_2,p)=(\infty,\infty,\infty)$. \qed

\section{Proof of the claim (\ref{Reduction to the same support ball}) at \texorpdfstring{$(p_1,p_2,p)=(2,\infty,2)$}{}{}}
\label{Section: proof of claim B}



In this section we establish the claim \eqref{Reduction to the same support ball} for the point $(p_1,p_2,p)=(2,\infty,2)$. Note that here $\alpha(2, \infty)= d/2$. To  derive the required estimate, the main ingredient we use is the weighted Plancherel estimate with respect to the first-layer weight (see Proposition \ref{Proposition: First layer weighted Plancherel for limited alpha}). Let $\gamma >0$. Then an application of Cauchy-Schwartz inequality implies that 
\begin{align}
\label{Estimate of first L2 norm in 2 infinity}
    & \|\chi_{B_j} \mathcal{B}^{\alpha}_{j}(\mathfrak{F}_j, \mathfrak{G}_j)\|_{L^2} \leq \Big[ \sup_x \Big(\int_{G} \frac{|\mathfrak{G}_j(z,s)|^2}{|z-x|^{2\gamma}} \, d(z,s)\Big)^{\frac{1}{2}} \Big] \times \\
    &\nonumber \hspace{1cm} \Big(\int_{G \times G} |x-z|^{2\gamma} \Big| \int_{G} \mathcal{K}^{\alpha}_{j} ((y,t)^{-1}(x,u), (z,s)^{-1}(x,u)) \mathfrak{F}_j(y,t)\, d(y,t) \Big|^2 d(z,s) \, d(x,u) \Big)^{\frac{1}{2}} .
\end{align}

For the first factor in the right hand side of the above inequality, using H\"older's inequality and Lemma \eqref{Lemma: Integral of weight over ball} for any $0\leq \gamma< d_1/2$, we get 
\begin{align}
\label{Inequality: Integration of g with weight}
    \Big(\int_{G} \frac{|\mathfrak{G}_j(z,s)|^2}{|z-x|^{2\gamma}} \, d(z,s)\Big)^{\frac{1}{2}} & \leq C \|\mathfrak{G}_j\|_{L^{\infty}} 2^{j(1+\varepsilon)(Q/2-\gamma)} .
\end{align}

In order to estimate the second factor, let us interpret the integral inside the modulus as the kernel of a spectral multiplier of sub-Laplacian in the following way,
\begin{align}
\label{Equation: simplification of bilinear kernel into linear kernel}
   & \int_{G} \mathcal{K}^{\alpha}_{j} ((y,t)^{-1}(x,u), (z,s)^{-1}(x,u))\, \mathfrak{F}_j(y,t)\, d(y,t) \\
   &\nonumber = \frac{1}{(2\pi)^{d_2}} \int_{\mathfrak{g}_{2,r}^{*}} e^{i \langle \lambda_2, u-s \rangle} \sum_{\mathbf{k}_2 \in \mathbb{N}^\Lambda} F_{(x,u)}^j (\eta_{\mathbf{k}_2}^{\lambda_2})\, \varphi_{\mathbf{k}_2}^{\mathbf{b}^{\lambda_2}, \mathbf{r}_2}(R_{\lambda_2}^{-1} (x-z)) \exp{\left(\tfrac{i}{2} \lambda_2([x, z]) \right)} \, d\lambda_2\\
   &\nonumber =: \mathcal{K}_{F_{(x,u)}^j(\mathcal{L})} (x-z, u-s-\tfrac{1}{2}[z,x]),
\end{align}
where 
\begin{align*}
  F_{(x,u)}^j(\eta_{\mathbf{k}_2}^{\lambda_2}) &= \frac{1}{(2\pi)^{d_2}} \int_{\mathfrak{g}_{2,r}^{*}} e^{i \langle \lambda_1, u \rangle} \sum_{\mathbf{k}_1 \in \mathbb{N}^\Lambda} \Psi_j^{\alpha}(\eta_{\mathbf{k}_1}^{\lambda_1}, \eta_{\mathbf{k}_2}^{\lambda_2}) \left[\mathfrak{F}_j^{\lambda_1} \times_{\lambda_1} \varphi_{\mathbf{k}_1}^{\mathbf{b}^{\lambda_1}, \mathbf{r}_1}(R_{\lambda_1}^{-1}\cdot) \right](x)\, d\lambda_1 .
\end{align*}
Thus, with the help of \eqref{Equation: simplification of bilinear kernel into linear kernel} and applying Proposition \ref{Proposition: First layer weighted Plancherel for limited alpha} for $0\leq \gamma<d_2/2$, we obtain
\begin{align}
\label{Estimate: Estimate of first term for 2,infinity,2}
    & \Big[\int_{G} \int_{G} |x-z|^{2\gamma}  \Big| \int_{G} \mathcal{K}^{\alpha}_{j} ((y,t)^{-1}(x,u), (z,s)^{-1}(x,u))\, \mathfrak{F}_j(y,t)\, d(y,t) \Big|^2 d(z,s) \, d(x,u) \Big]^{\frac{1}{2}} \\
    &\nonumber = \Big[\int_{G} \left(\int_{G} |x-z|^{2\gamma} |\mathcal{K}_{F_{(x,u)}^j(\mathcal{L})} (x-z, u-s-\tfrac{1}{2}[z,x]) |^2 d(z,s) \right) d(x,u) \Big]^{\frac{1}{2}} \\
    &\nonumber \leq C \Big[\int_{G} \int_{0}^{1} |F_{(x,u)}^j(\eta_2)|^2 \ d\eta_2\ d(x,u) \Big]^{\frac{1}{2}} .
\end{align}
As a result, the final expression in the quantity above can be estimated as follows.
\begin{align}
\label{Estimate of G(x,u) for 2 infinity}
    & \int_{0}^{1} \int_{G} |F_{(x,u)}^j(\eta_2)|^2 \ d(x,u)\ d\eta_2 \\
    &\nonumber=C \int_{0}^{1} \int_{\mathfrak{g}_{2,r}^{*}} \sum_{\mathbf{k}_1 \in \mathbb{N}^\Lambda} |\Psi_j^{\alpha}(\eta_{\mathbf{k}_1}^{\lambda_1}, \eta_2)|^2 \|\mathfrak{F}_j^{\lambda_1} \times_{\lambda_1} \varphi_{\mathbf{k}_1}^{\mathbf{b}^{\lambda_1}, \mathbf{r}_1}(R_{\lambda_1}^{-1}\cdot)\|_{L^2}^2 \, d\lambda_1 \, d\eta_2 \\
    &\nonumber = C \int_{\mathfrak{g}_{2,r}^{*}} \sum_{\mathbf{k}_1 \in \mathbb{N}^\Lambda} \Big(\int_{0}^{1} |\Psi_j^{\alpha}(\eta_{\mathbf{k}_1}^{\lambda_1}, \eta_2)|^2 \, d\eta_2 \Big) \|\mathfrak{F}_j^{\lambda_1} \times_{\lambda_1} \varphi_{\mathbf{k}_1}^{\mathbf{b}^{\lambda_1}, \mathbf{r}_1}(R_{\lambda_1}^{-1}\cdot)\|_{L^2}^2 \, d\lambda_1 \\
    &\nonumber \leq C 2^{-2 j \alpha} 2^{-j} \|\mathfrak{F}_j\|_{L^2}^2 ,
\end{align}
where we have used the fact that, $\displaystyle{\sup_{\eta_{\mathbf{k}_1}^{\lambda_1}} \int_{0}^{1} |\Psi_j^{\alpha}(\eta_{\mathbf{k}_1}^{\lambda_1}, \eta_2)|^2 \, d\eta_2 \leq C\, 2^{-2 j \alpha}\, 2^{-j}}$.

Finally, combining \eqref{Inequality: Integration of g with weight}, \eqref{Estimate: Estimate of first term for 2,infinity,2}, and \eqref{Estimate of G(x,u) for 2 infinity} and plugging them into the estimate \eqref{Estimate of first L2 norm in 2 infinity}, yields
\begin{align*}
    \|\chi_{B_j} B^{\alpha}_{j} (\mathfrak{F}_j, \mathfrak{G}_j)\|_{L^2} &\leq C 2^{-j \alpha} 2^{-j/2} \|\mathfrak{F}_j\|_{L^2} \|\mathfrak{G}_j\|_{L^{\infty}} 2^{j(1+\varepsilon)(Q/2-\gamma)} \\
    &\leq C 2^{-j \delta} \|\mathfrak{F}_j\|_{L^2} \|\mathfrak{G}_j\|_{L^{\infty}} , 
\end{align*}
where as $\alpha>(d-1)/2$, we can choose $\varepsilon>0$ so small and $\gamma$ very close to $d_2/2$ such that $\delta=\alpha-(Q/2-\gamma)(1+\varepsilon)-1/2>0$. It is important to note that since $G$ is M\'etivier group, we always have $d_1>d_2$, so that $0\leq \gamma<d_2/2<d_1/2$.

This completes the proof of the claim \eqref{Reduction to the same support ball} for the point $(p_1,p_2,p)=(2,\infty,2)$. \qed

\section{Proof of the claim (\ref{Reduction to the same support ball}) at \texorpdfstring{$(p_1,p_2,p)=(1,\infty,1)$}{}{}}
\label{Section: Proof of claim at 1 infinity}


In this Section, we focus on proving the claim \eqref{Reduction to the same support ball} for the points $(p_1, p_2,p)=(1,\infty,1)$. Note that for $(p_1, p_2,p)=(1, \infty,1)$, we have $\alpha(1, \infty)=Q/2$. Using Cauchy-Schwartz inequality, \eqref{First layer restriction estimate in second prop} of Proposition \ref{Proposition: First layer weighted Plancherel for limited alpha}, the fact from (\ref{Convergence of sum of l using phi}), and H\"older's inequality from the expression (\ref{Equation: Decomposition in terms of Fourier series}), we obtain
\begin{align*}
   \|\chi_{B_j} B^{\alpha}_{j} (\mathfrak{F}_j, \mathfrak{G}_j)\|_{L^1} &\leq C \sum_{l \in \mathbb{Z}} \| \phi_{j,l}^{\alpha}(\mathcal{L})\mathfrak{F}_j\|_{L^2} \| \psi_l(\mathcal{L})\mathfrak{G}_j \|_{L^2} \\
   &\leq C \sum_{l \in \mathbb{Z}} \frac{1}{(1+|l|)^{(1+\varepsilon)}} \{(1 + |l|)^{1 + \varepsilon} \|\phi_{j,l}^{\alpha}\|_{L^{\infty}}\} \|\mathfrak{F}_j\|_{L^{1}} \|\mathfrak{G}_j\|_{L^{2}} \\
   &\leq C 2^{-j\alpha} 2^{j \varepsilon} \|\mathfrak{F}_j\|_{L^{2}} 2^{j Q/2} \|\mathfrak{G}_j\|_{L^{\infty}} \\
   &\leq C 2^{-j \delta} \|\mathfrak{F}_j\|_{L^{2}} \|\mathfrak{G}_j\|_{L^{\infty}} .
\end{align*}
where $\delta=\alpha-Q/2-\varepsilon>0$, since $\alpha > Q/2$ we can choose $\varepsilon>0$ so small such that $\alpha -Q/2 - \varepsilon>0$.

This completes the proof of the claim \eqref{Reduction to the same support ball} for the point $(p_1,p_2,p)=(1,\infty,1)$. \qed

\section{Proof of Theorem \ref{Bilinear Bochner-Riesz theorem with restricted f and g}}
\label{Section: Proof of theorem for restricted f and g}
In this section, we prove Theorem \ref{Bilinear Bochner-Riesz theorem with restricted f and g}. Under certain assumptions on the support of the Fourier transforms of $f$ and $g$, this theorem serves as a precise analogue of the corresponding Euclidean results (see Theorem \ref{Theorem: Euclidean bilinear Bochner-Riesz for Grafakos}) except at the point $(1,1,1/2)$. In our setting, the Euclidean dimension $n$ in the smoothness threshold is replaced by the topological dimension $d$ of $G$.

Analogous to  Theorem \ref{Bilinear Bochner-Riesz Main theorem}, the proof of Theorem \ref{Bilinear Bochner-Riesz theorem with restricted f and g} essentially reduces to the estimate of $\mathcal{B}^{\alpha}_{j,1}$ for the points $(p_1,p_2,p) \in \{(1,1,1/2)$,$(1,2,2/3)$,$(2,2,1)$,$(1,\infty,1)$,$(2,\infty,2)$, $(\infty, \infty, \infty) \}$ (see \ref{subsection:Estimate of Bj1}). Here we only prove Theorem \ref{Bilinear Bochner-Riesz theorem with restricted f and g} at the point $(p_1,p_2,p)=(1,\infty,1)$. The argument for the remaining cases follows from similar ideas.

\subsection*{Proof of Theorem \ref{Bilinear Bochner-Riesz theorem with restricted f and g} for \texorpdfstring{$\mathbf{(p_1,p_2,p)=(1,\infty,1)}$}{}}
Note that in this case $\alpha(1, \infty)=d/2$ and $\supp \mathcal{F}_2 g(z, \cdot) \subseteq \{\lambda_2 : |\lambda_2| \geq \kappa_2 \}$ for some $\kappa_2>0$ and every $z \in \mathfrak{g}_2$.

Let $\Omega : \mathbb{R} \to \mathbb{R}$ be a smooth function such that, $1-\Omega$ is bump function which equals to $1$ in $(-\kappa_2/2, \kappa_2/2)$ and is supported on $(-\kappa_2, \kappa_2)$. Then from \eqref{Expression after dyadic decomposition}, for each $j \geq 0$, using the support of $\mathcal{F}_2 g(z, \cdot)$, we can express
\begin{align*}
    \mathcal{B}_{j}^{\alpha}(f,g)(x,u) &= \frac{1}{(2\pi)^{2 d_2}} \int_{\mathfrak{g}_{2,r}^{*}} \int_{\mathfrak{g}_{2,r}^{*}} e^{i \langle \lambda_1 + \lambda_2 , u \rangle} \sum_{\mathbf{k}_1, \mathbf{k}_2 \in \mathbb{N}^N} \Psi_j^{\alpha}(\eta_{\mathbf{k}_1}^{\lambda_1}, \eta_{\mathbf{k}_2}^{\lambda_2}) \Omega(|\lambda_2|) \\
    & \hspace{2cm} \left[f^{\lambda_1} \times_{\lambda_1} \varphi_{\mathbf{k}_1}^{\mathbf{b}^{\lambda_1}, \mathbf{r}}(R_{\lambda_1}^{-1}\cdot) \right](x) \left[g^{\lambda_2} \times_{\lambda_2} \varphi_{\mathbf{k}_2}^{\mathbf{b}^{\lambda_2}, \mathbf{r}}(R_{\lambda_2}^{-1}\cdot) \right](x) \,  d\lambda_1 \,  d\lambda_2 \\
    &=: \mathcal{B}_{j}^{\alpha, \kappa_2}(f,g)(x,u) .
\end{align*}
Let $\mathcal{K}_j^{\alpha, \kappa_2}$ denote the kernel of the operator $\mathcal{B}_{j}^{\alpha, \kappa_2}$. Similarly as in \eqref{Reduction to the same support ball}, it is enough to prove that, whenever $\alpha>d/2$, there exists a $\delta>0$ such that
\begin{align*}
    \| \chi_{B_j} \mathcal{B}^{\alpha, \kappa_2}_{j} (\mathfrak{F}_j ,\mathfrak{G}_j) \|_{L^{1}(G)} \leq C 2^{-j \delta} \|\mathfrak{F}_j\|_{L^{1}(G)} \|\mathfrak{G}_j\|_{L^{\infty}(G)} \quad \text{whenever} \ \supp{\mathfrak{F}_j}, \supp{\mathfrak{G}_j} \subseteq B_j ,
\end{align*}

Furthermore, from (\ref{Equation: Decomposition in terms of Fourier series}), we also decompose $\mathcal{B}_{j}^{\alpha, \kappa_2}$ as follows:
\begin{align*}
    \chi_{B_j}(x,u) \mathcal{B}_{j}^{\alpha, \kappa_2}(\mathfrak{F}_j,\mathfrak{G}_j)(x,u) &= C \chi_{B_j}(x,u) \sum_{l \in \mathbb{Z}} \left\{ \phi_{j,l}^{\alpha}(\mathcal{L}) \mathfrak{F}_j(x,u) \right\} \left\{ \psi_l^{\kappa_2}(\mathcal{L}, T) \mathfrak{G}_j(x,u) \right\} ,
\end{align*}
where $\psi_l^{\kappa_2} : \mathbb{R} \times \mathbb{R} \to \mathbb{C}$ defined by $\psi_l^{\kappa_2}(\eta_2, \tau_2)= \psi_l(\eta_2) \Omega(\tau_2)$. Let $\Theta$ be the function as defined in \eqref{Definition: Cutoff function theta}. Then similar to \eqref{Use of Remark in Niedorf}, for $M_1, M_2 \in \mathbb{Z}$, we have the following decomposition 
\begin{align}
\label{Cutoff of M1 for the first linear multiplier}
    \phi_{j,l}^{\alpha}(\mathcal{L}) \mathfrak{F}_j &= \sum_{M_1=-\ell_0 }^{\infty} \phi_{j,l,M_1}^{\alpha}(\mathcal{L}, T) \mathfrak{F}_j , \quad \text{and} \quad \psi_{l}(\mathcal{L}) \mathfrak{G}_j = \sum_{M_2=-\ell_0 }^{\infty} \psi_{l,M_2}(\mathcal{L}, T)\mathfrak{G}_j ,
\end{align}
where
\begin{align*}
    \phi_{j,l,M_1}^{\alpha} (\eta_1, \tau_1)= \phi_{j,l}^{\alpha}(\eta_1) \,\Theta(2^{M_1}\tau_1) \quad \text{and} \quad \psi_{l,M_2}(\eta_2, \tau_2) = \psi_{l}(\eta_2) \,\Theta(2^{M_2}\tau_2) .
\end{align*}
Consequently, right now only introducing $M_1$-cut-off we can write
\begin{align}
\label{decomposition in M1 for 1,2}
    & \chi_{B_j}(x,u) \mathcal{B}_{j}^{\alpha, \kappa_2}(\mathfrak{F}_j,\mathfrak{G}_j)(x,u) \\
    &\nonumber= C\chi_{B_j}(x,u) \Big(\sum_{M_1=-\ell_0 }^{j} + \sum_{M_1=j+1 }^{\infty} \Big) \sum_{l \in \mathbb{Z}} \phi_{j,l,M_1}^{\alpha}(\mathcal{L}, T) \mathfrak{F}_j(x,u)\, \psi_{l}^{\kappa_2}(\mathcal{L}, T) \mathfrak{G}_j(x,u) \\
    &\nonumber=: S_1 + S_2 .
\end{align}

\subsection{Estimate of \texorpdfstring{$S_{2}$}{}{}}
Estimate of $I_4$ is easy and can be easily handled by Proposition \ref{Proposition: First layer weighted Plancherel}. An application of H\"older's inequality and \eqref{Restriction estimate inside prop} of Proposition \ref{Proposition: First layer weighted Plancherel} yields 
\begin{align*}
    \|S_2\|_{L^{1}} &\leq C \sum_{l \in \mathbb{Z}} \sum_{M_1=j+ 1 }^{\infty} \| \phi_{j,l, M_1}^{\alpha}(\mathcal{L}, T)\mathfrak{F}_j \|_{L^2} \|\psi_l^{\kappa_2}(\mathcal{L}, T)\mathfrak{G}_j \|_{L^2} \\
    &\leq C \sum_{l \in \mathbb{Z}} \sum_{M_1=j+ 1 }^{\infty} 2^{-M_1 d_2/2} \|\phi_{j,l}^{\alpha}\|_{L^{\infty}} \| \mathfrak{F}_j\|_{L^1} \| \mathfrak{G}_j \|_{L^2} .
\end{align*}
Furthermore, summing over $M_1 \geq j+1$, using the estimate \eqref{Convergence of sum of l using phi} and since $\supp{\mathfrak{G}_j} \subseteq B(0, 3 \cdot 2^{j(1+\varepsilon)})$, we obtain
\begin{align*}
    \|S_2\|_{L^{1}} &\leq C 2^{j \varepsilon} 2^{-j \alpha} 2^{-j d_2/2} \left\| \mathfrak{F}_j \right\|_{L^1} \|\mathfrak{G}_j\|_{L^{\infty}} 2^{jQ(1+\varepsilon)/2} \\ 
    &\leq C 2^{-j \delta} \| \mathfrak{F}_j \|_{L^1} \|\mathfrak{G}_j\|_{L^{\infty}} ,
\end{align*}
where $\delta=\alpha-d/2-\varepsilon(1+Q/2)>0$, as for $\alpha>d/2$, we can choose $\varepsilon$ sufficiently small such that $\alpha-d/2-\varepsilon(1+Q/2)>0$.

\subsection{Estimate of \texorpdfstring{$S_{1}$}{}{}}
To estimate $S_1$, as in \eqref{Cutoff of M1 for the first linear multiplier}, let us introduce an additional cut-off in $|\lambda_2|$ variable.
\begin{align*}
    S_1 &= C \chi_{B_j}(x,u) \sum_{l \in \mathbb{Z}} \sum_{M_1=-\ell_0 }^{j} \sum_{M_2=-\ell_0 }^{\infty} \phi_{j,l,M_1}^{\alpha}(\mathcal{L}, T) \mathfrak{F}_j(x,u)\, \psi_{l, M_2}^{\kappa_2}(\mathcal{L}, T) \mathfrak{G}_j(x,u) ,
\end{align*}
where $\psi_{l,M_2}^{\kappa_2}(\eta_2, \tau_2) = \psi_l(\eta_2) \Omega(\tau_2) \,\Theta(2^{M_2}\tau_2)$.

Note that, due to the support of $\Omega$ and $\Theta$, there exists $L_0>0$ depending on $\delta_0$ such that $M_2 \leq L_0$. In order to estimate $S_1$, we have to further decompose both the support of $\mathfrak{F}_j$ and $\mathfrak{G}_j$. Recall that $\supp{\mathfrak{F}_j}, \supp{\mathfrak{G}_j}\subseteq B\left(0, 3 \cdot 2^{j(1+\varepsilon)}\right)$. Hence applying Lemma \ref{Lemma: Decomposition of ball}, there exists a $C>0$ such that
\begin{align*}
    B\left(0, 3 \cdot 2^{j(1+\varepsilon)}\right) \subseteq B^{|\cdot|}\left(0, 3C \cdot 2^{j(1+\varepsilon)}\right) \times B^{|\cdot|}\left(0, 9C \cdot 2^{2j(1+\varepsilon)}\right) .
\end{align*}
Note that in $S_1$, we always have $-\ell_0 \leq M_1 \leq j$. Accordingly, for each $M_1 \in \{-\ell_0, \ldots, j\}$, we decompose $B^{|\cdot|}\left(0, 3C \cdot 2^{j(1+\varepsilon)}\right) \times B^{|\cdot|}\left(0, 9C \cdot 2^{2j(1+\varepsilon)}\right)$ with respect to the first layer into disjoint sets $S_{m_1}^{M_1}$ such that
\begin{align}
\label{Expression: Decomposition of ball into smaller balls}
    \supp{\mathfrak{F}_j} &= \bigcup_{m_1=1}^{N_{M_1}} S_{m_1}^{M_1} ,
\end{align}
with the property 
\begin{align}
\label{Property of the smaller balls}
    S_{m_1}^{M_1} \subseteq B^{|\cdot|}\left(a_{m_i}^{M_1}, 3C \cdot 2^{M_1 (1+\varepsilon)}\right) \times B^{|\cdot|}\left(0, 9C \cdot 2^{2j(1+\varepsilon)}\right)
\end{align}
and whenever $m_i \neq m_i'$,  $|a_{m_1}^{M_1}-a_{m_1'}^{M_1}|> 3 C 2^{M_1 (1+\varepsilon)}/2 $ holds. Furthermore, the number of subsets $N_{M_1}$ in this decomposition is bounded by constant times $2^{(j-M_1)(1+\varepsilon)d_1}$. For each $1\leq m_1 \leq N_{M_1}$ and $\gamma>0$, we also define
\begin{align}
\label{Defination of dilated balls}
    \widetilde{S}_{m_1}^{M_1} &:= B^{|\cdot|}\left(a_{m_1}^{M_1}, 3 C \cdot 2^{M_1 (1+\varepsilon)} 2^{\gamma j+1} \right) \times B^{|\cdot|}\left(0, C 2^{2j (1+\varepsilon)}\right) .
\end{align}
Similarly we also decompose $\supp{\mathfrak{G}_j}$. With the aid of the above decomposition, we express $\mathfrak{F}_j$ and $\mathfrak{G}_j$ as: 
\begin{align}
\label{Decomposition of f and g}
    \mathfrak{F}_j = \sum_{m_1=1}^{N_{M_1}} \mathfrak{F}_{m_1}^j \quad \quad \text{and} \quad \quad \mathfrak{G}_j = \sum_{m_2=1}^{N_{M_1}} \mathfrak{G}_{m_2}^j ,
\end{align}
where $\mathfrak{F}_{m_1}^j = \mathfrak{F}_j \chi_{S_{m_1}^{M_1}} $ and $\mathfrak{G}_{m_2}^j = \mathfrak{G}_j \chi_{S_{m_2}^{M_1}}$.

Consequently, with the help of (\ref{Defination of dilated balls}) and (\ref{Decomposition of f and g}), we break the summand $S_1$ into three parts as follows 
\begin{align*}
    &S_{1}= \sum_{M_1=-\ell_0 }^{j} \sum_{M_2=-\ell_0 }^{L_0} \sum_{m_1=1}^{N_{M_1}} \chi_{B_j}(x,u) (1-\chi_{\widetilde{S}_{m_1}^{M_1}})(x,u) \mathcal{B}^{\alpha, \kappa_2}_{j,M_1, M_2}(\mathfrak{F}_{m_1}^j, \mathfrak{G}_j)(x,u) \\
    &+\sum_{M_1=-\ell_0 }^{j} \sum_{M_2=-\ell_0 }^{L_0} \sum_{m_1=1}^{N_{M_1}} \sum_{m_2=1}^{N_{M_1}} \chi_{B_j}(x,u) \chi_{\widetilde{S}_{m_1}^{M_1}}(x,u) \chi_{\widetilde{S}_{m_2}^{M_1}}(x) \mathcal{B}^{\alpha, \kappa_2}_{j,M_1, M_2}(\mathfrak{F}_{m_1}^j, \mathfrak{G}_{m_2}^j)(x,u) \\
    &+ \sum_{M_1=-\ell_0 }^{j} \sum_{M_2=-\ell_0 }^{L_0} \sum_{m_1=1}^{N_{M_1}} \sum_{m_2=1}^{N_{M_1}} \chi_{B_j}(x,u) \chi_{\widetilde{S}_{m_1}^{M_1}}(x,u) (1-\chi_{\widetilde{S}_{m_2}^{M_1}})(x,u)  \mathcal{B}^{\alpha, \kappa_2}_{j,M_1, M_2}(\mathfrak{F}_{m_1}^j, \mathfrak{G}_{m_2}^j)(x,u) \\
    &=: S_{11} + S_{12} + S_{13} ,
\end{align*}
where
\begin{align*}
    \mathcal{B}^{\alpha, \kappa_2}_{j,M_1, M_2}(f, g)(x,u) &= \sum_{l \in \mathbb{Z}} \phi_{j,l,M_1}^{\alpha}(\mathcal{L}, T) f(x,u) \psi_{l, M_2}^{\kappa_2}(\mathcal{L}, T) g(x,u) .
\end{align*}

\subsubsection{\textbf{Estimate of} \texorpdfstring{$S_{11}$}{}{}}
\label{subsubsection: estimate of S11}
We show that $S_{11}$ has arbitrarily large decay. An application of H\"older's inequality implies
\begin{align}
\label{Estimate: Estimate of S_{12} in 1,1}
     \|S_{11}\|_{L^{1}} & \leq \sum_{M_1=-\ell_0 }^{j} \sum_{M_2=-\ell_0 }^{L_0} \sum_{m_1=1}^{N_{M_1}} \|\chi_{B_j} (1-\chi_{\widetilde{S}_{m_1}^{M_1}}) \mathcal{B}^{\alpha, \kappa_2}_{j,M_1, M_2}(\mathfrak{F}_{m_1}^j, \mathfrak{G}_j) \|_{L^{1}} .
\end{align}
Again using H\"older's inequality, we further see that 
\begin{align}
\label{Estimate: Application of Holder inequality for error part in 1,1}
    & \|\chi_{B_j} (1-\chi_{\widetilde{S}_{m_1}^{M_1}}) \mathcal{B}^{\alpha, \kappa_2}_{j,M_1, M_2}(\mathfrak{F}_{m_1}^j, \mathfrak{G}_j)\|_{L^{1}} \\
    & \nonumber \leq C \sum_{l \in \mathbb{Z}} \|\chi_{B_j} (1-\chi_{\widetilde{S}_{m_1}^{M_1}}) \phi_{j,l,M_1}^{\alpha}(\mathcal{L}, T) \mathfrak{F}_{m_1}^j \|_{L^2} \| \psi_{l,M_2}^{\kappa_2}(\mathcal{L}, T) \mathfrak{G}_j \|_{L^2} .
\end{align}
Let us focus on the factor $\|\chi_{B_j} (1-\chi_{\widetilde{S}_{m_1}^{M_1}}) \phi_{j,l,M_1}^{\alpha}(\mathcal{L}, T) \mathfrak{F}_{m_1}^j \|_{L^2}$. We denote the convolution kernel of $\phi_{j,l,M_1}^{\alpha}(\mathcal{L}, T)$ by $\mathcal{K}_{\phi_{j,l,M_1}^{\alpha}(\mathcal{L}, T)}$. An application of Minkowski's integral inequality gives
\begin{align}
\label{Estimate: Application of Minkowski inequality for error part in 1,1}
    & \|\chi_{B_j} (1-\chi_{\widetilde{S}_{m_1}^{M_1}}) \phi_{j,l,M_1}^{\alpha}(\mathcal{L}, T) \mathfrak{F}_{m_1}^j \|_{L^2} \\
    &\nonumber \leq \int_{G} |\mathfrak{F}_{m_1}^j(y,t)| \Big( \int_{G} |\chi_{B_j}(x,u) (1-\chi_{\widetilde{S}_{m_1}^{M_1}})(x,u)  \\
    &\nonumber \hspace{6cm} \times  \mathcal{K}_{\phi_{j,l,M_1}^{\alpha}(\mathcal{L}, T)}((y,t)^{-1}(x,u))|^2 \ d(x,u) \Big)^{1/2} d(y,t) .
\end{align}
Note that if $(x,u) \in \supp{\chi_{B_j} (1-\chi_{\widetilde{S}_{m_1}^{M_1}})}$ and $(y,t) \in \supp{\mathfrak{F}_{m_1}^j}$, then
\begin{align*}
    |x-a_{m_1}^{M_1}| \geq C 2^{\gamma j+1} 2^{M_1 (1+\varepsilon)} \quad \text{and} \quad |y-a_{m_1}^{M_1}| \leq C 2^{M_1 (1+\varepsilon)} ,
\end{align*}
and this in particular implies $|x-y| \geq C 2^{\gamma j} 2^{M_1 (1+\varepsilon)}$. Therefore, using this observation along with the translation invariance of the Haar measure, we see that for any $N >0$,
\begin{align}
\label{Estimate: Weighted plancherel on outside ball in (1,1,1/2)}
    & \Big( \int_{G} |\chi_{B_j}(x,u) (1-\chi_{\widetilde{S}_{m_1}^{M_1}})(x,u) \mathcal{K}_{\phi_{j,l,M_1}^{\alpha}(\mathcal{L}, T)}((y,t)^{-1}(x,u))|^2 \ d(x,u) \Big)^{1/2} \\
    &\nonumber \leq C (2^{\gamma j} 2^{M_1 (1+\varepsilon)})^{-N} \Big( \int_{G} ||x-y|^{N} \mathcal{K}_{\phi_{j,l,M_1}^{\alpha}(\mathcal{L}, T)}(x-y, u-t-\tfrac{1}{2}[y,x])|^2 \ d(x,u) \Big)^{1/2} \\
    &\nonumber \leq C (2^{\gamma j} 2^{M_1 (1+\varepsilon)})^{-N} \Big( \int_{G} ||x|^{N} \mathcal{K}_{\phi_{j,l,M_1}^{\alpha}(\mathcal{L}, T)}(x, u)|^2 \ d(x,u) \Big)^{1/2} .
\end{align}
Then substituting the above estimate (\ref{Estimate: Weighted plancherel on outside ball in (1,1,1/2)}) into \eqref{Estimate: Application of Minkowski inequality for error part in 1,1} and applying \eqref{First layer kernel estimate inside Prop} Proposition \ref{Proposition: First layer weighted Plancherel} for $\phi_{j,l}^{\alpha}$ in place of $F$ leads us to 
\begin{align}
\label{Estimate: Outside of M1 ball for f in (1,1,1/2)}
    \|\chi_{B_j} (1-\chi_{\widetilde{S}_{m_1}^{M_1}}) \phi_{j,l,M_1}^{\alpha}(\mathcal{L}, T) \mathfrak{F}_{m_1}^j \|_{L^2} &\leq C 2^{-\gamma j N} 2^{-M_1 \varepsilon N} 2^{-M_1 d_2/2} \|\phi_{j,l}^{\alpha}\|_{L^{\infty}} \|\mathfrak{F}_{m_1}^j\|_{L^1},
\end{align}
for any $N>0$.

Therefore from \eqref{Estimate: Application of Holder inequality for error part in 1,1} and using the estimate \eqref{Estimate: Outside of M1 ball for f in (1,1,1/2)}, along with the fact established in \eqref{Convergence of sum of l using phi}, we obtain
\begin{align*}
    & \|\chi_{B_j} (1-\chi_{\widetilde{S}_{m_1}^{M_1}}) \mathcal{B}^{\alpha, \kappa_2}_{j,M_1, M_2}(\mathfrak{F}_{m_1}^j, \mathfrak{G}_j)\|_{L^{1}} \\
    &\leq C \sum_{l \in \mathbb{Z}} C 2^{-\gamma j N} 2^{-M_1 \varepsilon N} 2^{-M_1 d_2/2} \|\phi_{j,l}^{\alpha}\|_{L^{\infty}} \|\mathfrak{F}_{m_1}^j\|_{L^1} \|\mathfrak{G}_j\|_{L^2} \\
    &\leq C 2^{j \varepsilon} 2^{-j \alpha} 2^{-\gamma j N} 2^{-M_1 \varepsilon N} 2^{-M_1 d_2/2} \|\mathfrak{F}_{m_1}^j\|_{L^1} 2^{j Q (1+\varepsilon)/2} \|\mathfrak{G}_j\|_{L^{\infty}} .
\end{align*}
Consequently, with the help of the above estimate and by choosing $N>0$ large enough and $\varepsilon>0$ very small, there exists a $\delta>0$ such that
\begin{align}
\label{Estimate of S12 for 1,q0}
    \|S_{11}\|_{L^{1}} & \leq \sum_{M_1=-\ell_0 }^{j} \sum_{M_2=-\ell_0 }^{L_0} \sum_{m_1=1}^{N_{M_1}} \|\chi_{B_j} (1-\chi_{\widetilde{S}_{m_1}^{M_1}}) \mathcal{B}^{\alpha, \kappa_2}_{j,M_1, M_2}(\mathfrak{F}_{m_1}^j, \mathfrak{G}_j)\|_{L^{1}} \\
    &\nonumber \leq C 2^{j \varepsilon} 2^{-j \alpha} 2^{-\gamma j N} 2^{jQ(1+\varepsilon)/2} \|\mathfrak{F}_j\|_{L^1} \|\mathfrak{G}_j\|_{L^{\infty}} \\
    &\nonumber \leq C 2^{-j \delta} \|\mathfrak{F}_j\|_{L^1} \|\mathfrak{G}_j\|_{L^{\infty}} .
\end{align}

\subsubsection{\textbf{Estimate of} \texorpdfstring{$S_{12}$}{}{}}
First note that by Lemma \ref{Lemma: Decomposition of ball}, there exists a constant $C>0$ such that $B_j \subseteq B^{|\cdot|}(0, C 2^{j(1+\varepsilon)}) \times B^{|\cdot|}(0, C 2^{2j(1+\varepsilon)})$. Consequently, following the approach in \eqref{Expression: Decomposition of ball into smaller balls}, \eqref{Property of the smaller balls}, we decompose $B_j$ into disjoint sets $S_{m}^{M_1}$ with respect to the first layer and write
\begin{align*}
    & S_{12} = \sum_{M_1=-\ell_0 }^{j} \sum_{M_2=-\ell_0 }^{L_0} \sum_{m=1}^{N_{M_1}} \sum_{m_1=1}^{N_{M_1}} \sum_{m_2=1}^{N_{M_1}} \chi_{S_{m}^{M_1}}(x,u) \chi_{\widetilde{S}_{m_1}^{M_1}}(x,u) \chi_{\widetilde{S}_{m_2}^{M_1}}(x,u)  \mathcal{B}^{\alpha, \kappa_2}_{j,M_1, M_2}(\mathfrak{F}_{m_1}^j, \mathfrak{G}_{m_2}^j)(x,u) \\
    &= \sum_{M_1=-\ell_0 }^{j} \sum_{M_2=-\ell_0 }^{L_0} \sum_{m=1}^{N_{M_1}} \sum_{m_1: S_{m}^{M_1} \cap \widetilde{S}_{m_1}^{M_1} \neq \emptyset} \ \  \sum_{m_2: S_{m}^{M_1} \cap \widetilde{S}_{m_2}^{M_1} \neq \emptyset} \\
    & \hspace{4cm} \chi_{S_{m}^{M_1}}(x,u) \chi_{\widetilde{S}_{m_1}^{M_1}}(x,u) \chi_{\widetilde{S}_{m_2}^{M_1}}(x,u) \mathcal{B}^{\alpha, \kappa_2}_{j,M_1, M_2}(\mathfrak{F}_{m_1}^j, \mathfrak{G}_{m_2}^j)(x,u) .
\end{align*}

Let us first estimate $\|\mathcal{B}^{\alpha, \kappa_2}_{j,M_1, M_2}(\mathfrak{F}_{m_1}^j, \mathfrak{G}_{m_2}^j) \|_{L^{1}}$. An application of H\"older's inequality, \eqref{Restriction estimate inside prop} of Proposition \ref{Proposition: First layer weighted Plancherel}, and the fact \eqref{Convergence of sum of l using phi} yields
\begin{align}
\label{Estimate: L1 estimate for the point 1,q0}
    & \| \mathcal{B}^{\alpha, \kappa_2}_{j,M_1, M_2}(\mathfrak{F}_{m_1}^j, \mathfrak{G}_{m_2}^j) \|_{L^{1}} \leq C \sum_{l \in \mathbb{Z}} \|\phi_{j,l,M_1}^{\alpha}(\mathcal{L}, T) \mathfrak{F}_{m_1}^j \|_{L^2} \|\psi_{l,M_2}^{\kappa_2}(\mathcal{L}, T) \mathfrak{G}_{m_2}^j \|_{L^2} \\
    &\nonumber \leq C \sum_{l \in \mathbb{Z}} 2^{-M_1 d_2/2} \|\phi_{j,l}^{\alpha}\|_{L^{\infty}} \|\mathfrak{F}_{m_1}^j \|_{L^1} \|\mathfrak{G}_{m_2}^j\|_{L^2} \\
    &\nonumber \leq C 2^{j \varepsilon} 2^{-j \alpha} 2^{- M_1 d_2/2} \|\mathfrak{F}_{m_1}^j \|_{L^1} 2^{M_1 (1+\varepsilon) d_1/2} 2^{j (1+\varepsilon) d_2} \|\mathfrak{G}_{m_2}^j\|_{L^{\infty}} .
\end{align}
With the aid of the above estimate, we obtain the following.
\begin{align*}
    \|S_{12}\|_{L^{1}} & \leq C 2^{j \varepsilon(1+ Q/2)} 2^{-j \alpha} 2^{jd/2} \sum_{M_1=-\ell_0 }^{j} 2^{j (M_1-j)(d_1-d_2)/2} \sum_{M_2=-\ell_0 }^{L_0}  \\
    &\nonumber \hspace{1cm} \sum_{m=1}^{N_{M_1}} \Big\{\sum_{m_1: S_{m}^{M_1} \cap \widetilde{S}_{m_1}^{M_1} \neq \emptyset} \|\mathfrak{F}_{m_1}^j \|_{L^1} \Big\}  \Big\{\sum_{m_2: S_{m}^{M_1} \cap \widetilde{S}_{m_2}^{M_1} \neq \emptyset} \|\mathfrak{G}_{m_2}^j\|_{L^{\infty}} \Big\} .
\end{align*}

To continue, we need to estimate the number of overlaps between the sets $S_{m}^{M_1}$ and $\widetilde{S}_{m_i}^{M_1}$ for $i=1,2$. Since, we have chosen the disjoint sets $S_{m_i}^{M_1}$ in such a way that $|a_{m_i}^{M_1}-a_{m_i'}^{M_1}|> C 2^{M_1 (1+\varepsilon)}/2 $ for $m_i \neq m_i'$, and correspondingly defined $\widetilde{S}_{m_i}^{M_1}$ (see \eqref{Defination of dilated balls}), we have the following bounded overlapping property:
\begin{align}
\label{Bounded overlapping for m1 for 1,1}
    \sup_{m} \# \left\{ m_i : S_{m}^{M_1} \cap \widetilde{S}_{m_i}^{M_1} \neq \emptyset \right\} &\leq \sup_{m} \# \left\{ m_i : |a_{m}^{M_1} - a_{m_i}^{M_1}| \leq C 2^{{M_1}(1+\varepsilon)} 2^{\gamma j+1} \right\} \\
    &\nonumber \leq C 2^{C \gamma j} .
\end{align}
Similarly, we can also see
\begin{align}
\label{Bounded overlapping for m2 for 1,1}
    \sup_{m_i} \# \left\{ m : S_{m}^{M_1} \cap \widetilde{S}_{m_i}^{M_1} \neq \emptyset \right\} \leq C 2^{C \gamma j} .
\end{align}

Also, recall that since we are working on M\'etivier groups $G$, we always have $d_1>d_2$. Therefore, by applying the bounded overlapping property \eqref{Bounded overlapping for m1 for 1,1} and \eqref{Bounded overlapping for m2 for 1,1}, we obtain
\begin{align*}
    \|S_{12}\|_{L^{1}} & \leq C 2^{j \varepsilon(1+ Q/2)} 2^{-j \alpha} 2^{jd/2} \sum_{M_1=-\ell_0 }^{j} 2^{j (M_1-j)(d_1-d_2)/2} \sum_{M_2=-\ell_0 }^{L_0} \\
    & \hspace{1cm} \Big\{\sum_{m=1}^{N_{M_1}} \sum_{m_1: S_{m}^{M_1} \cap \widetilde{S}_{m_1}^{M_1} \neq \emptyset} \|\mathfrak{F}_{m_1}^j\|_{L^1} \Big\} \Big\{ \sup_m \sum_{m_2: S_{m}^{M_1} \cap \widetilde{S}_{m_2}^{M_1} \neq \emptyset} \|\mathfrak{G}_{m_2}^j\|_{L^{\infty}} \Big\} \\
    & \leq C 2^{j \varepsilon(1+ Q/2)} 2^{-j \alpha} 2^{2C\gamma j} 2^{jd/2} \|\mathfrak{F}_j\|_{L^1} \|\mathfrak{G}_j\|_{L^{\infty}} \\
    & \leq C 2^{-j \delta} \|\mathfrak{F}_j\|_{L^1} \|\mathfrak{G}_j\|_{L^{\infty}} ,
\end{align*}
where $\delta=\alpha-d/2-2C\gamma-\varepsilon(1+Q/2)$. Since $\alpha>d/2$, we can choose $\varepsilon$ and $\gamma$ to be sufficiently small such that $\delta>0$.

\subsubsection{\textbf{Estimate of} \texorpdfstring{$S_{13}$}{}{}}
The estimate of $S_{13}$ is similar to that of $S_{11}$ (see \ref{subsubsection: estimate of S11}), where we also obtain arbitrary large decay. This is the part where we need the assumption that the Fourier transform of $g$ in the second variable is supported outside the origin. 

Using H\"older's inequality, we observe that 
\begin{align}
\label{Estimate L1 estimate to Holder ineq}
    & \|\chi_{B_j} (1-\chi_{\widetilde{S}_{m_2}^{M_1}}) \mathcal{B}^{\alpha, \kappa_2}_{j,M_1, M_2}(\mathfrak{F}_{m_1}^j, \mathfrak{G}_{m_2}^j) \|_{L^{1}} \\
    &\nonumber \leq C \sum_{l \in \mathbb{Z}} \|\phi_{j,l,M_1}^{\alpha}(\mathcal{L}, T) \mathfrak{F}_{m_1}^j \|_{L^2} \|\chi_{B_j} (1-\chi_{\widetilde{S}_{n_2,m_2}^{M_1,j}}) \psi_{l,M_2}^{\kappa_2}(\mathcal{L}, T) \mathfrak{G}_{m_2}^j \|_{L^2} .
\end{align}

We then have, by applying Minkowski's integral inequality, 
\begin{align}
\label{Estimate: Application of Minkowski inequality for error part in 1,q0}
    & \|\chi_{B_j} (1-\chi_{\widetilde{S}_{n_2,m_2}^{M_1,j}}) \psi_{l,M_2}^{\kappa_2}(\mathcal{L}, T) \mathfrak{G}_{m_2}^j \|_{L^2} \leq \int_{G} |\mathfrak{G}_{m_2}^j(z,s)| \\
    &\nonumber \hspace{1cm} \times \left( \int_{G} |\chi_{B_j}(x,u) (1-\chi_{\widetilde{S}_{m_2}^{M_1}})(x,u) \mathcal{K}_{\psi_{l,M_2}^{\kappa_2}(\mathcal{L}, T)}((z,s)^{-1}(x,u))|^2 \ d(x,u) \right)^{1/2} d(z,s) ,
\end{align}
where $\mathcal{K}_{\psi_{l,M_2}^{\kappa_2}(\mathcal{L}, T)}$ denote the convolution kernel of $\psi_{l,M_2}^{\kappa_2}(\mathcal{L}, T)$.

If $(x,u) \in \supp{\chi_{B_j} (1-\chi_{\widetilde{S}_{m_2}^{M_1}})}$ and $(z,s) \in \supp{\mathfrak{G}_{m_2}^j}$ then one can easily see that $|x-z| \geq C\, 2^{\gamma j} 2^{M_1 (1+\varepsilon)}$. Similarly to \eqref{Estimate: Weighted plancherel on outside ball in (1,1,1/2)}, applying Proposition \ref{Proposition: First layer weighted Plancherel} for any $N>0$ yields
\begin{align}
\label{Estimate: Weighted plancherel on outside ball in (1,q0)}
    & \Big( \int_{G} |\chi_{B_j}(x,u) (1-\chi_{\widetilde{S}_{m_2}^{M_1}})(x,u) \mathcal{K}_{\psi_{l,M_2}^{\kappa_2}(\mathcal{L}, T)}((z,s)^{-1}(x,u))|^2 \ d(x,u) \Big)^{1/2} \\
    &\nonumber \leq C (2^{\gamma j} 2^{M_1 (1+\varepsilon)})^{-N} \Big( \int_{G} ||x-z|^{N} \mathcal{K}_{\psi_{l,M_2}^{\kappa_2}(\mathcal{L}, T)}(x-z, u-s-\tfrac{1}{2}[z,x])|^2 \ d(x,u) \Big)^{1/2} \\
    &\nonumber \leq C 2^{-\gamma j N} 2^{-M_1(1+\varepsilon) N} 2^{M_2(N-d_2/2)} .
\end{align}

Thus, by combining all the above estimates along with \eqref{Restriction estimate inside prop} of Proposition (\ref{Proposition: First layer weighted Plancherel}) and H\"older's inequality, from the estimate \eqref{Estimate L1 estimate to Holder ineq}, we have
\begin{align*}
    & \|\chi_{B_j} (1-\chi_{\widetilde{S}_{m_2}^{M_1}}) \mathcal{B}^{\alpha, \kappa_2}_{j,M_1, M_2}(\mathfrak{F}_{m_1}^j, \mathfrak{G}_{m_2}^j) \|_{L^{1}} \\
    &\leq C \sum_{l \in \mathbb{Z}} 2^{-M_1 d_2/2} \|\phi_{j,l}^{\alpha}\|_{L^{\infty}} \|\mathfrak{F}_{m_1}^j\|_{L^1} 2^{-\gamma j N} 2^{-M_1 (1+\varepsilon) N} 2^{M_2(N-d_2/2)} \|\mathfrak{G}_{m_2}^j\|_{L^1} \\
    &\leq C 2^{j \varepsilon} 2^{-j \alpha} 2^{-M_1 d_2/2} 2^{-\gamma j N} 2^{-M_1 (1+\varepsilon) N} 2^{M_2(N-d_2/2)} \|\mathfrak{F}_{m_1}^j\|_{L^1} 2^{j Q(1+\varepsilon)} \|\mathfrak{G}_{m_2}^j\|_{L^{\infty}} .
\end{align*}

Finally, the above estimate and the bound $N_{M_1} \lesssim 2^{(j-M_1)d_1/2}$ (see just below \eqref{Property of the smaller balls}), immediately implies that
\begin{align*}
    \|S_{13}\|_{L^1} &\leq C 2^{j \varepsilon} 2^{-j \alpha} 2^{-\gamma j N} 2^{j Q(1+\varepsilon)} \sum_{M_1=-\ell_0 }^{j} 2^{-M_1 d_2/2} 2^{-M_1 (1+\varepsilon) N} \sum_{M_2=-\ell_0 }^{L_0} 2^{M_2(N-d_2/2)} \\
    & \hspace{5cm} \times \Big\{\sum_{m_1=1}^{N_{M_1}} \|\mathfrak{F}_{m_2}^j\|_{L^1} \Big\} \Big\{\sum_{m_2=1}^{N_{M_1}} \|\mathfrak{G}_{m_2}^j\|_{L^{\infty}} \Big\} \\
    &\leq C_{N, L_0} 2^{j \varepsilon} 2^{-j \alpha} 2^{-\gamma j N} 2^{j Q(1+\varepsilon)} 2^{j d_1/2} \|\mathfrak{F}_j\|_{L^1} \|\mathfrak{G}_j\|_{L^{\infty}} .
\end{align*}
Now choosing $N>0$ large enough and $\varepsilon>0$ very small, we can find a $\delta>0$ such that
\begin{align*}
    \|S_{13}\|_{L^1} &\leq C 2^{-j \delta} \|\mathfrak{F}_j\|_{L^1} \|\mathfrak{G}_j\|_{L^{\infty}} .
\end{align*}
This completes the proof of Theorem (\ref{Bilinear Bochner-Riesz theorem with restricted f and g}) for $(p_1,p_2,p)=(1,\infty,1)$. \qed

\section{Mixed norm estimate}
\label{Section: Mixed norm estimates}
In this section, we prove Theorem \ref{Theorem: Mixed norm estimate for first layer} and Theorem \ref{Theorem: Mixed norm estimate for second layer}. Since the ideas of all these proofs are similar, we only provide the proof of Theorem \ref{Theorem: Mixed norm estimate for first layer}, and the others follow in an analogous manner. 

\subsection*{Proof of Theorem \ref{Theorem: Mixed norm estimate for first layer}}
Recall that, in view of the decomposition displayed in \eqref{Equation: Dyadic decomposition of Bochner-Riesz}, it is enough to show that, for each $j \geq 0$, whenever $\alpha>\alpha(p_1, p_2)$, there exists a $\delta>0$ such that
\begin{align*}
    \|\mathcal{B}^{\alpha}_{j}(f,g)\|_{L_x^{p'} L_u^{p''}(G)} &\leq C 2^{-j \delta} \|f\|_{L_x^{p_1'}L_u^{p_1''}(G)} \|g\|_{L_x^{p_2'}L_u^{p_2''}(G)} ,
\end{align*}
where $1\leq p_1',p_1'', p_2', p_2'' \leq \infty$, and $$1/p' = 1/p_1'+1/p_2',\quad 1/p'' = 1/p_1''+1/p_2''$$ for $(p_1',p_2',p')=(1,2,2/3)$ and $(p_1'',p_2'',p'')=(1,\infty,1)$.

Similar to \eqref{Decomposition kernel into four parts}, first we decompose $\mathcal{K}_j^{\alpha}$ as follows
\begin{align*}
     \mathcal{K}_j^{\alpha} &= \sum_{\theta_1, \theta_2 \in \{1,2,3,4\}} \mathcal{K}_{j,A_{\theta_1}, A_{\theta_2}}^{\alpha} , 
\end{align*}
where for $\varepsilon>0$,and $\theta_1, \theta_2 \in \{1,2,3,4\}$, we define
\begin{align*}
    A_{1} := B^{|\cdot|} (0, 2^{j(1+\varepsilon)}) \times B^{|\cdot|} (0, 2^{2j(1+\varepsilon)}) ; \quad \quad A_{2} := B^{|\cdot|} (0, 2^{j(1+\varepsilon)}) \times B^{|\cdot|} (0, 2^{2j(1+\varepsilon)})^c ; \\
    A_{3} := B^{|\cdot|} (0, 2^{j(1+\varepsilon)})^{c} \times B^{|\cdot|} (0, 2^{2j(1+\varepsilon)}) ; \quad \quad A_{4} := B^{|\cdot|} (0, 2^{j(1+\varepsilon)})^{c} \times B^{|\cdot|} (0, 2^{2j(1+\varepsilon)})^{c} , 
\end{align*}
and $\mathcal{K}_{j,A_{\theta_1}, A_{\theta_2}}^{\alpha}$ is given by
\begin{align}
\label{decomposition of kernel in mixed case}
      \mathcal{K}_{j,A_{\theta_1}, A_{\theta_2}}^{\alpha}((y,t), (z, s)) &= \mathcal{K}_{j}^{\alpha}((y,t), (z,s))\,  \chi_{A_{\theta_1}} (y,t) \,  \chi_{A_{\theta_2}} (z,s) .
\end{align}
Let $\mathcal{B}^{\alpha}_{j,A_{\theta_1}, A_{\theta_2}}$ denote the bilinear operator corresponding to the kernel $\mathcal{K}_{j,A_{\theta_1}, A_{\theta_2}}^{\alpha}$.

Estimate of $\mathcal{B}^{\alpha}_{j,A_{\theta_1}, A_{\theta_2}}$, except for the case where $\mathcal{B}^{\alpha}_{j,A_{1}, A_{1}}$, can be established using similar techniques to those used for estimating  $\mathcal{B}^{\alpha}_{j,l}$ for $l=2,3,4$ from Section \ref{Section: Dyadic decomposition of Bochner-Riesz}, with the help of H\"older's inequality and Young's inequality for mixed norms. As a representative case, let us prove the estimate for $\mathcal{B}^{\alpha}_{j,A_3, A_2}$; the estimates for the remaining terms can be obtained analogously.

As in subsection \eqref{Subsection: Estimate of Bj4}, by applying Lemma \ref{Lemma: Pointwise kernel estimate for Bj} for any $N>0$ and $\epsilon_1>0$, it follows that
\begin{align*}
    |\mathcal{B}^{\alpha}_{j,A_3,A_2}(f,g)(x,u)| &\leq C 2^{j@N} (|f|*k_1)(x,u) (|g|*k_2)(x,u) ,
\end{align*}
where
\begin{align*}
    k_1(y,t) = \frac{\chi_{A_3} (y,t)}{(1+\|(y,t)\|)^N} \quad  \text{and} \quad  k_2(z,s) = \frac{\chi_{A_2} (z,s)}{(1+\|(z,s)\|)^N} .
\end{align*}

Since $1/p' = 1/p_1'+1/p_2'$, $1/p'' = 1/p_1''+1/p_2''$, an application of H\"older's inequality and Young's inequality for mixed norm yields
\begin{align*}
    \|\mathcal{B}^{\alpha}_{j,A_3, A_2}(f,g)\|_{L_x^{p'} L_u^{p''}(G)} &\leq C 2^{j2N} \Big(\int_{\mathbb{R}^{d_2}} \Big( \int_{\mathbb{R}^{d_1}} ||f|*k_1(x,u)|^{p_1'} \ dx \Big)^{p_1''/p_1'} \ du \Big)^{1/p_1''} \\
    & \hspace{2cm} \times \Big(\int_{\mathbb{R}^{d_2}} \Big( \int_{\mathbb{R}^{d_1}} ||g|*k_2(x,u)|^{p_2'} \ dx \Big)^{p_2''/p_2'} \ du \Big)^{1/p_2''} \\
    &\leq C 2^{j2N} \|k_1\|_{L^1(G)} \|k_2\|_{L^1(G)} \|f\|_{L_x^{p_1'}L_u^{p_1''}(G)} \|g\|_{L_x^{p_2'}L_u^{p_2''}(G)} .
\end{align*}
The $L^1$-norm of $k_1$ and $k_2$ can be estimated as follows.
\begin{align*}
    \|k_1\|_{L^1(G)} \lesssim \int_{|y| > 2^{j(1+\varepsilon)}} \int_{|t| \leq 2^{2j(1+\varepsilon)}} \frac{dy\, dt}{(1+|y|)^N} \leq C 2^{-j (N-d_1-2d_2)(1+\varepsilon)} ,
\end{align*}
and
\begin{align*}
    \|k_2\|_{L^1(G)} \lesssim \int_{|z| \leq 2^{j(1+\varepsilon)}} \int_{|s| > 2^{j(1+\varepsilon)}} \frac{dz\, ds}{(1+|s|)^N} \leq C 2^{-j (N-d_2-d_1)(1+\varepsilon)} .
\end{align*}

Therefore, choosing $N>0$ sufficiently large and $\varepsilon>0$ sufficiently small, there exists $\delta=2N\varepsilon-(2d_1+3d_2)(1+\varepsilon)>0$ such that 
\begin{align*}
    \|\mathcal{B}^{\alpha}_{j,A_3,A_2}(f,g)\|_{L_x^{p'} L_u^{p''}(G)} &\leq C 2^{-j \delta} \|f\|_{L_x^{p_1'}L_u^{p_1''}(G)} \|g\|_{L_x^{p_2'}L_u^{p_2''}(G)} .
\end{align*}

It remains to estimate $\mathcal{B}^{\alpha}_{j,A_{1}, A_{1}}$. We again choose sequences $\{a_{n'}\}_{n' \in \mathbb{N}}$ and $\{b_{n''}\}_{n'' \in \mathbb{N}}$ (see subsection \eqref{subsection:Estimate of Bj1}) such that
\begin{align*}
    |a_{n'}-a_{m'}|>2^{j(1+\varepsilon)}/10, \quad \text{for} \quad n'\neq m' , \quad \sup_{a \in \mathbb{R}^{d_1}} \inf_{n'}|a-a_{n'}| \leq 2^{j(1+\varepsilon)}/10 ; \quad \text{and} \\
    |b_{n''}-b_{m''}|>2^{2j(1+\varepsilon)}/10 \quad \text{for} \ n'' \neq m'' , \quad \sup_{b \in \mathbb{R}^{d_2}} \inf_{n''}|b-b_{n''}| \leq 2^{2j(1+\varepsilon)}/10 .
\end{align*}

Recall that from \eqref{decomposition of kernel in mixed case}, for $(x,u) \in G$ we write
\begin{align*}
    \mathcal{K}^{\alpha}_{j,A_1, A_1}((y,t)^{-1}(x,u), (z,s)^{-1}(x,u)) =: \widetilde{\mathcal{K}^{\alpha}_{j,A_1, A_1}}((x,u),(y,t),(z,s)) ,
\end{align*}
and hence we have
\begin{align*}
    \supp{\widetilde{\mathcal{K}^{\alpha}_{j,A_1, A_1}}} \subseteq \mathcal{D}_{j}^{|\cdot|} &:= \{\big((x,u),(y,t),(z,s)\big) :\  |x-y|\leq 2^{j(1+\varepsilon)}, |u-t| \leq 2^{2j(1+\varepsilon)}, \\
    & \hspace{5cm} |x-z|\leq 2^{j(1+\varepsilon)}, |u-s| \leq 2^{2j(1+\varepsilon)} \} .
\end{align*}
Let us define $\displaystyle{S_{n'}^{|\cdot|_1,j} := \Bar{B}^{|\cdot|}(a_{n'}, \tfrac{2^{j(1+\varepsilon)}}{10}) \setminus \cup_{m < n} \Bar{B}^{|\cdot|}(a_{m'}, \tfrac{2^{j(1+\varepsilon)}}{10})}$ and similarly we also define $\displaystyle{S_{n''}^{|\cdot|_2,j} := \Bar{B}^{|\cdot|}(a_{n''}, \tfrac{2^{2j(1+\varepsilon)}}{10}) \setminus \cup_{m < n} \Bar{B}^{|\cdot|}(a_{m''}, \tfrac{2^{2j(1+\varepsilon)}}{10})}$, then we can easily see
\begin{align*}
    & \mathcal{D}_{j}^{|\cdot|} \subseteq \bigcup_{\substack{n',n'',n_1',n_1'',,n_2',n_2'': \\ |a_{n'}-a_{n_1'}|\leq 2 \cdot 2^{j(1+\varepsilon)}, |b_{n''}-b_{n_1''}|\leq 2 \cdot 2^{2j(1+\varepsilon)} \\
    |a_{n'}-a_{n_2'}|\leq 2 \cdot 2^{j(1+\varepsilon)}, |b_{n''}-b_{n_2''}|\leq 2 \cdot 2^{2j(1+\varepsilon)}}} (S_{n'}^{|\cdot|_1,j} \times S_{n''}^{|\cdot|_2,j}) \times \left((S_{n_1'}^{|\cdot|_1,j} \times S_{n_1''}^{|\cdot|_2,j}) \times (S_{n_2'}^{|\cdot|_1,j} \times S_{n_2''}^{|\cdot|_2,j}) \right) ,
\end{align*}

With the aid of this decomposition, we can write $\mathcal{B}^{\alpha}_{j,A_{1}, A_{1}}$ as 
\begin{align*}
    & \mathcal{B}^{\alpha}_{j,A_{1}, A_{1}}(f,g)(x,u) \\
    &= \sum_{n',n''=0}^{\infty} \sum_{\substack{n_1': |a_{n'}-a_{n_1'}|\leq 2 \cdot 2^{j(1+\varepsilon)}, n_1'':|b_{n''}-b_{n_1''}|\leq 2 \cdot 2^{2j(1+\varepsilon)} \\
    n_2':|a_{n'}-a_{n_2'}|\leq 2 \cdot 2^{j(1+\varepsilon)}, n_2'':|b_{n''}-b_{n_2''}|\leq 2 \cdot 2^{2j(1+\varepsilon)}}} \chi_{S_{n'}^{|\cdot|_1,j} \times S_{n''}^{|\cdot|_2,j}}(x,u) \mathcal{B}^{\alpha}_{j,A_{1}, A_{1}}(f_{n_1', n_1''}^j,g_{n_2', n_2''}^j)(x,u) ,
\end{align*}
where $f_{n_1', n_1''}^j = f \chi_{S_{n_1'}^{|\cdot|_1,j} \times S_{n_1''}^{|\cdot|_2,j}}$ and $g_{n_2', n_2''}^j = g \chi_{S_{n_2'}^{|\cdot|_1,j} \times S_{n_2''}^{|\cdot|_2,j}}$.

Before proceed further, let us make the following claim. There exist some $\epsilon_1>0$ such that
\begin{align}
\label{Expression: Assumption for mixed norms}
    \|\chi_{S_{n'}^{|\cdot|_1,j} \times S_{n''}^{|\cdot|_2,j}} \mathcal{B}^{\alpha}_{j,A_{1}, A_{1}}(f_{n_1', n_1''}^j,g_{n_2', n_2''}^j)\|_{L_x^{2/3} L_u^{1}} & \leq C 2^{-j \alpha} 2^{j \epsilon_1} 2^{j/2} 2^{j(d_1-d_2)/2} \|f_{n_1', n_1''}^j\|_{L^{1}} \|g_{n_2', n_2''}^j\|_{L^{2}} .
\end{align}

First, we assume that the claim holds for the moment and proceed to complete the estimate for $\mathcal{B}^{\alpha}_{j,A_{1}, A_{1}}$. Note that, similar to \eqref{Bounded overlapping property for j ball}, we also have the following bounded overlapping property in this context:
\begin{align}
\label{bounded overlapping for mixed norms}
    \sup_{n'} \#\{m' : |a_{n'}-a_{m'}| \leq  2 \cdot 2^{j(1+\varepsilon)}\} \leq C\  \text{and} \  \sup_{n''} \#\{m'' : |b_{n''}-b_{m''}| \leq  2 \cdot 2^{j(1+\varepsilon)}\} \leq C .
\end{align}

In the following, we adopt the following short hand notation: $\displaystyle{\sum_{n_i':} := \sum_{n_i':|a_{n'}-a_{n_i'}| \leq 2 \cdot 2^{j(1+\varepsilon)}}}$ for $i=1,2$ and also for $n_i''$. Using the fact that the sets $S_{n'}^{|\cdot|_1,j}$ and $S_{n''}^{|\cdot|_2,j}$ are disjoint, it follows that
\begin{align*}
    & \|\mathcal{B}^{\alpha}_{j,A_{1}, A_{1}} (f,g)\|_{L_x^{2/3} L_u^{1}} = \Big\|\sum_{n',n''=0}^{\infty} \sum_{n_1':,n_1'':,n_2':,n_2'':} \chi_{S_{n'}^{|\cdot|_1,j} \times S_{n''}^{|\cdot|_2,j}} \mathcal{B}^{\alpha}_{j,A_{1}, A_{1}}(f_{n_1',n_1''}^j,g_{n_2',n_2''}^j) \Big\|_{L_x^{2/3} L_u^{1}} \\
    &= \sum_{n''=0}^{\infty} \Big\{\Big[ \int_{\mathbb{R}^{d_2}} \Big(\sum_{n'=0}^{\infty} \int_{\mathbb{R}^{d_1}} \Big| \sum_{n_1':,n_1'':,n_2':,n_2'':} \chi_{S_{n'}^{|\cdot|_1,j} \times S_{n''}^{|\cdot|_2,j}} \mathcal{B}^{\alpha}_{j,A_{1}, A_{1}}(f_{n_1',n_1''}^j,g_{n_2',n_2''}^j) \Big|^{2/3} \ dx \Big)^{3/2} \ du \Big]^{2/3} \Big\}^{3/2} .
\end{align*}
Applying triangle inequality and the fact \eqref{Triangle inequality for p less than 1 case}, the last quantity can be dominated by
\begin{align*}
    & \sum_{n''=0}^{\infty} \Big\{\sum_{n'=0}^{\infty} \sum_{n_1':,n_1'':,n_2':,n_2'':} \Big[ \int_{\mathbb{R}^{d_2}} \Big( \int_{\mathbb{R}^{d_1}} \big| \chi_{S_{n'}^{|\cdot|_1,j} \times S_{n''}^{|\cdot|_2,j}} \mathcal{B}^{\alpha}_{j,A_{1}, A_{1}}(f_{n_1',n_1''}^j,g_{n_2',n_2''}^j) \big|^{2/3} \ dx \Big)^{3/2} \ du \Big]^{2/3} \Big\}^{3/2} .
\end{align*}

Consequently, by applying the claim \eqref{Expression: Assumption for mixed norms}, the quantity on the right hand side of the above term can be bounded by
\begin{align*}
    & C 2^{-j \alpha} 2^{j \epsilon_1} 2^{j/2} 2^{j(d_1-d_2)/2} \sum_{n''=0}^{\infty} \Big[\sum_{n'=0}^{\infty} \Big\{ \sum_{n_1':,n_1'':} \Big( \int_{\mathbb{R}^{d_2}} \int_{\mathbb{R}^{d_1}} | f_{n_1',n_1''}^j(y,t)| \ dy\ dt \Big)^{2/3} \\
    & \hspace{7cm} \times \sum_{n_2':,n_2'':} \Big( \int_{\mathbb{R}^{d_2}} \int_{\mathbb{R}^{d_1}} | g_{n_2',n_2''}^j(z,s)|^{2}  \ dz\ ds \Big)^{1/3} \Big\} \Big]^{3/2} .
\end{align*}

In addition, using H\"older's inequality and bounded overlapping property \eqref{bounded overlapping for mixed norms}, the above expression can be further dominated by
\begin{align*}
    & C 2^{-j \alpha} 2^{j \epsilon_1} 2^{j/2} 2^{j(d_1-d_2)/2} \sum_{n''=0}^{\infty} \Big[\sum_{n'=0}^{\infty} \Big\{\sum_{n_1'':} \int_{\mathbb{R}^{d_2}} \sum_{n_1':} \int_{\mathbb{R}^{d_1}} | f_{n_1',n_1''}^j(y,t)| \ dy\ dt \Big\}^{2/3} \\
    & \hspace{6cm} \times \Big\{\sum_{n_2'':} \int_{\mathbb{R}^{d_2}} \sum_{n_2':} \int_{\mathbb{R}^{d_1}} | g_{n_2',n_2''}^j(z,s)|^{2} \ dz \ ds \Big\}^{1/3} \Big]^{3/2} \\
    &\leq C 2^{-j \alpha} 2^{j \epsilon_1} 2^{j/2} 2^{j(d_1-d_2)/2} \sum_{n''=0}^{\infty} \Big[ \Big\{\sum_{n_1'':} \int_{\mathbb{R}^{d_2}} \sum_{n'=0}^{\infty} \sum_{n_1':} \int_{\mathbb{R}^{d_1}} | f_{n_1',n_1''}^j(y,t)| \ dy\ dt \Big\} \\
    & \hspace{6cm} \times \Big\{\sum_{n_2'':} \int_{\mathbb{R}^{d_2}} \sum_{n'=0}^{\infty} \sum_{n_2':} \int_{\mathbb{R}^{d_1}} | g_{n_2',n_2''}^j(z,s)|^{2} \ dz \ ds \Big\}^{1/2} \Big] .
\end{align*}

Again, applying bounded overlapping property \eqref{bounded overlapping for mixed norms}, we observe that the above quantity can be bounded by
\begin{align*}
    & C 2^{-j \alpha} 2^{j \epsilon_1} 2^{j/2} 2^{j(d_1-d_2)/2} 2^{j d_2(1+\varepsilon)} \Big\{\sum_{n''=0}^{\infty} \sum_{n_1'':} \int_{\mathbb{R}^{d_2}} \sum_{n'=0}^{\infty} \sum_{n_1':} \int_{\mathbb{R}^{d_1}} | f_{n_1',n_1''}^j(y,t)| \ dy\ dt \Big\}\\
    & \hspace{1cm} \sup_{n''} \Big\{\sum_{n_2'': |b_{n''}-b_{n_2''}| \leq 2 \cdot 2^{2j(1+\varepsilon)}} \sup_{s \in B^{|\cdot|}(b_{n_2''}, 2^{2j(1+\varepsilon)}/5)} \Big( \sum_{n'=0}^{\infty} \sum_{n_2':} \int_{\mathbb{R}^{d_1}} | g_{n_2',n_2''}^j(z,s)|^{2} \ dz \Big)^{\frac{1}{2} \cdot 2} \Big\}^{1/2} \\
    &\leq C 2^{-j \delta} \Big\{\int_{\mathbb{R}^{d_2}} \int_{\mathbb{R}^{d_1}} | f(x,u)| \ dx\ du \Big\} \Big\{\sup_{u} \Big( \int_{\mathbb{R}^{d_1}} | g(x,u)|^{2} \ dx \Big)^{\frac{1}{2}} \Big\} \\
    &\leq C 2^{-j \delta} \|f\|_{L_x^{1}L_u^{1}} \|g\|_{L_x^{2}L_u^{\infty}} ,
\end{align*}
where we have used as $\alpha>(d+1)/2$, we can choose $\varepsilon, \epsilon_1>0$ sufficiently small such that $\delta=\alpha-(d+1)/2+d_2 \varepsilon+\epsilon_1>0$. This completes the estimate of $\mathcal{B}^{\alpha}_{j,A_{1}, A_{1}}$, upon assuming the claim.

Note that similarly as in \eqref{Reduction to the same support ball} in order to prove \eqref{Expression: Assumption for mixed norms} enough to prove the following:
\begin{align}
\label{Expression: Assumption for mixed norms with only j case}
    \|\chi_{S_{n'}^{|\cdot|_1,j} \times S_{n''}^{|\cdot|_2,j}} \mathcal{B}^{\alpha}_{j}(f_{n_1', n_1''}^j,g_{n_2', n_2''}^j)\|_{L_x^{2/3} L_u^{1}} & \leq C 2^{-j \alpha} 2^{j \epsilon_1} 2^{j/2} 2^{j(d_1-d_2)/2} \|f_{n_1', n_1''}^j\|_{L^{1}} \|g_{n_2', n_2''}^j\|_{L^{2}} .
\end{align}
Thus, it only remains to prove claim \eqref{Expression: Assumption for mixed norms with only j case}. This can be estimated in a similar manner to claim \eqref{Claim for boundedness of Bj1} for $(p_1,p_2,p)=(1,2,2/3)$ (see subsection \ref{Proof at (1,2,2/3) for claim A}). Analogous to \eqref{decomposition in M1 for 1,2}, we first perform the following decomposition:
\begin{align*}
    & \chi_{S_{n'}^{|\cdot|_1,j} \times S_{n''}^{|\cdot|_2,j}}(x,u) \mathcal{B}^{\alpha}_{j}(f_{n_1', n_1''}^j, g_{n_2', n_2''}^j)(x,u) \\
    &= \chi_{S_{n'}^{|\cdot|_1,j} \times S_{n''}^{|\cdot|_2,j}}(x,u) \Big(\sum_{M_1=-\ell_0 }^{j} + \sum_{M_1=j+1 }^{\infty} \Big) \mathcal{B}^{\alpha}_{j, M_1}(f_{n_1', n_1''}^j, g_{n_2', n_2''}^j)(x,u) =: E_1 + E_2 ,
\end{align*}
where
\begin{align}
\label{Fourier series decomposition for mixed norm}
    \mathcal{B}^{\alpha}_{j, M_1}(f_{n_1', n_1''}^j, g_{n_2', n_2''}^j)(x,u) &= \sum_{l \in \mathbb{Z}} \phi_{j,l,M_1}^{\alpha}(\mathcal{L}, T) f_{n_1', n_1''}^j(x,u)\, \psi_{l}(\mathcal{L}) g_{n_2', n_2''}^j(x,u) .
\end{align}

\subsection{Estimate of \texorpdfstring{$E_{2}$}{}{}}
The estimate for $E_2$ can be deduced as follows. Application of H\"older's inequality, \eqref{Restriction estimate inside prop} of Proposition \ref{Proposition: First layer weighted Plancherel}, and \eqref{Convergence of sum of l using phi} yields
\begin{align*}
    \|E_2\|_{L_x^{2/3} L_u^{1}} & \leq C 2^{j d_1(1+\varepsilon)/2} \Big\|\sum_{l \in \mathbb{Z}} \sum_{M_1=j+ 1 }^{\infty} \phi_{j,l,M_1} ^{\alpha}(\mathcal{L}, T) f_{n_1', n_1''}^j \, \psi_{l}(\mathcal{L}) g_{n_2', n_2''}^j\Big\|_{L^1} \\
    &\leq C \sum_{l \in \mathbb{Z}} \sum_{M_1=j+ 1 }^{\infty} 2^{j d_1(1+\varepsilon)/2} \| \phi_{j,l, M_1}^{\alpha}(\mathcal{L}, T)f_{n_1', n_1''}^j \|_{L^2} \| \psi_l(\mathcal{L})g_{n_2', n_2''}^j \|_{L^2} \\
    &\leq C \sum_{l \in \mathbb{Z}} \sum_{M_1=j+ 1 }^{\infty} 2^{j d_1(1+\varepsilon)/2} 2^{-M_1 d_2/2} \|\phi_{j,l}^{\alpha}\|_{L^{\infty}} \| f_{n_1', n_1''}^j\|_{L^1} \| g_{n_2', n_2''}^j \|_{L^2} \\
    &\leq C 2^{j\epsilon_1} 2^{-j \alpha} 2^{j (d_1-d_2)/2} \| f_{n_1', n_1''}^j \|_{L^1} \| g_{n_2', n_2''}^j\|_{L^2} ,
\end{align*}
for some $\epsilon_1>0$.

\subsection{Estimate of \texorpdfstring{$E_{1}$}{}{}}
To estimate $E_{1}$, we first decompose only the support of $f_{n_1', n_1''}^j$ and the function $f_{n_1', n_1''}^j$ itself. For $M_1 \in \{-\ell_0, \ldots, j\}$, and in the same manner as \eqref{Expression: Decomposition of ball into smaller balls}, we decompose $S_{n_1'}^{|\cdot|_1,j}$ into disjoint sets $S_{n_1',m_1'}^{|\cdot|_1, M_1,j}$ such that $S_{n_1',m_1'}^{|\cdot|_1, M_1,j} \subseteq B^{|\cdot|}(a_{n_1', m_1'}^{M_1}, 2^{M_1(1+\varepsilon)}/5)$ and $|a_{n_1', m_1'}^{M_1}-a_{n_1', \Tilde{m}_1'}^{M_1}|>2^{M_1(1+\varepsilon)}/10$, whenever $m_1' \neq \Tilde{m}_1'$. For $\gamma>0$, we also define $\Tilde{S}_{n_1',m_1'}^{|\cdot|_1, M_1,j} := B^{|\cdot|}(a_{n_1', m_1'}^{M_1}, 2^{M_1(1+\varepsilon)} 2^{\gamma j+1}/5)$, and decompose $f_{n_1', n_1''}^j = \sum_{m_1'=1}^{N_{M_1}} f_{n_1',m_1', n_1''}^{M_1,j}$, where $f_{n_1',m_1', n_1''}^{M_1,j} = f_{n_1', n_1''}^j \chi_{S_{n_1',m_1'}^{|\cdot|, M_1,j}}$. Therefore, we break $E_1$ into two sums as
\begin{align*}
    E_{1} &= \sum_{M_1=-\ell_0 }^{j} \sum_{m_1'=1}^{N_{M_1}} \chi_{S_{n'}^{|\cdot|_1,j} \times S_{n''}^{|\cdot|_2,j}}(x,u) \chi_{\Tilde{S}_{n_1',m_1'}^{|\cdot|_1, M_1,j} \times S_{n''}^{|\cdot|_2,j}}(x,u) \mathcal{B}^{\alpha}_{j,M_1}(f_{n_1',m_1', n_1''}^{M_1,j}, g_{n_2', n_2''}^j)(x,u) \\
    &+ \sum_{M_1=-\ell_0 }^{j} \sum_{m_1'=1}^{N_{M_1}} \chi_{S_{n'}^{|\cdot|_1,j} \times S_{n''}^{|\cdot|_2,j}}(x,u) (1-\chi_{\Tilde{S}_{n_1',m_1'}^{|\cdot|_1, M_1,j} \times S_{n''}^{|\cdot|_2,j}})(x,u) \mathcal{B}^{\alpha}_{j,M_1}(f_{n_1',m_1', n_1''}^{M_1,j}, g_{n_2', n_2''}^j)(x,u) \\
    &=: E_{11} + E_{12} .
\end{align*}

\subsubsection{\textbf{Estimate of} \texorpdfstring{$E_{12}$}{}{}}
This estimate is similar to that of $S_{11}$ (see subsection \eqref{subsubsection: estimate of S11}), so we omit the details.

\subsubsection{\textbf{Estimate of} \texorpdfstring{$E_{11}$}{}{}}
Using \eqref{Fourier series decomposition for mixed norm}, let us first express $E_{11}$ as follows.
\begin{align*}
    E_{11} &= C \sum_{M_1=-\ell_0 }^{j} \sum_{l \in \mathbb{Z}} \Big\{ \sum_{m_1'=1}^{N_{M_1}} \chi_{\Tilde{S}_{n_1',m_1'}^{|\cdot|_1, M_1,j} \times S_{n''}^{|\cdot|_2,j}}(x,u) \phi_{j,l,M_1}^{\alpha}(\mathcal{L}, T) f_{n_1',m_1', n_1''}^{M_1,j}(x,u) \Big\}\{ \psi_{l}(\mathcal{L}) g_{n_2', n_2''}^j(x,u) \}.
\end{align*}
Applying the fact \eqref{Triangle inequality for p less than 1 case} and the triangle inequality, we see that
\begin{align*}
    \|E_{11}\|_{L_x^{2/3} L_u^1}^{2/3} &\leq C \sum_{M_1=-\ell_0 }^{j} \sum_{l \in \mathbb{Z}} \Big[\int_{\mathbb{R}^{d_2}} \Big( \int_{\mathbb{R}^{d_1}} | \sum_{m_1'=1}^{N_{M_1}} \chi_{\Tilde{S}_{n_1',m_1'}^{|\cdot|_1, M_1,j} \times S_{n''}^{|\cdot|_2,j}}(x,u) \phi_{j,l,M_1}^{\alpha}(\mathcal{L}, T) f_{n_1',m_1', n_1''}^{M_1,j}(x,u) \\
    &\hspace{6cm} \psi_{l}(\mathcal{L}) g_{n_2', n_2''}^j(x,u) |^{2/3} dx \Big)^{3/2} du \Big]^{2/3} .
\end{align*}

Consequently, using H\"older's inequality with respect to $x$ as well as $u$-variable, we find the quantity on the right hand side of the previous inequality is controlled by
\begin{align*}
     & C \sum_{M_1=-\ell_0 }^{j} \sum_{l \in \mathbb{Z}} \Big[\int_{\mathbb{R}^{d_2}} \Big(\int_{\mathbb{R}^{d_1}} \big|\sum_{m_1'=1}^{N_{M_1}} \chi_{\Tilde{S}_{n_1',m_1'}^{|\cdot|_1, M_1,j} \times S_{n''}^{|\cdot|_2,j}}(x,u) \phi_{j,l,M_1}^{\alpha}(\mathcal{L}, T) f_{n_1',m_1', n_1''}^{M_1,j}(x,u) \big| dx \Big)^{2} du \Big]^{\frac{1}{2} \cdot \frac{2}{3}} \\
     & \hspace{6cm} \times \Big[\int_{\mathbb{R}^{d_2}} \int_{\mathbb{R}^{d_1}} | \psi_{l}(\mathcal{L}) g_{n_2', n_2''}^j(x,u) |^2 dx \  du \Big]^{\frac{1}{2} \cdot \frac{2}{3}} .
\end{align*}
Notice that $|\Tilde{S}_{n_1',m_1'}^{|\cdot|_1, M_1,j}| \leq C 2^{M_1 d_1 (1+\varepsilon)} 2^{C\gamma j}$. Applying H\"older's inequality and \eqref{Restriction estimate inside prop} of Proposition \ref{Proposition: First layer weighted Plancherel}, we see that the above expression is dominated by
\begin{align}
\label{InE11 last expression befor phi estimate}
     & C \sum_{M_1=-\ell_0 }^{j} \sum_{l \in \mathbb{Z}} \Biggl[\sum_{m_1'=1}^{N_{M_1}} 2^{M_1 d_1(1+\varepsilon)/2} 2^{C \gamma j} \| \phi_{j,l,M_1}^{\alpha}(\mathcal{L}, T) f_{n_1',m_1', n_1''}^{M_1,j} \|_{L^2} \Biggr]^{2/3} \| g_{n_2', n_2''}^j \|_{L^2}^{2/3} \\
     &\nonumber\leq C \sum_{M_1=-\ell_0 }^{j} \sum_{l \in \mathbb{Z}} \Big[2^{M_1 d_1/2} 2^{C \gamma j} 2^{-M_1 d_2/2} \|\phi_{j,l}^{\alpha}\|_{L^{\infty}} \sum_{m_1'=1}^{N_{M_1}} \| f_{n_1',m_1', n_1''}^{M_1,j} \|_{L^1} \Big]^{2/3} \| g_{n_2', n_2''}^j \|_{L^2}^{2/3} .
\end{align}
Using the fact stated in (\ref{Convergence of sum of l using phi}), we can see that
\begin{align}
\label{sum of l for 1,2}
    \sum_{l \in \mathbb{Z}} \|\phi_{j,l}^{\alpha}\|_{L^{\infty}}^{2/3} &\leq \sum_{l \in \mathbb{Z}} \frac{1}{(1+|l|)^{1+\varepsilon}} \{(1+|l|)^{3/2+3\varepsilon/2} \|\phi_{j,l}^{\alpha}\|_{L^{\infty}} \}^{2/3} \leq C 2^{-j 2\alpha/3} 2^{j(1/3+\varepsilon)} .
\end{align}
With the help of the above fact, the expression in \eqref{InE11 last expression befor phi estimate} is further dominated by
\begin{align*}
    C 2^{\epsilon_1 j} 2^{-j 2 \alpha /3} 2^{j/3} 2^{j(d_1-d_2)/3} \| f_{n_1', n_1''}^{j} \|_{L^1}^{2/3} \| g_{n_2', n_2''}^j \|_{L^2}^{2/3} ,
\end{align*}
where we have used $\sum_{M_1=-\ell_0 }^{j} 2^{(M_1-j) (d_1-d_2)/3} \leq C$, since $d_1>d_2$ in case of M\'etivier groups.

Hence, we obtain
\begin{align*}
    \|E_{11}\|_{L_x^{2/3} L_u^1} &\leq C 2^{\epsilon_1 j} 2^{-j \alpha} 2^{j/2} 2^{j (d_1-d_2)/2} \| f_{n_1', n_1''}^{j} \|_{L^1} \| g_{n_2', n_2''}^j \|_{L^2} .
\end{align*}

This completes the proof of claim \eqref{Expression: Assumption for mixed norms with only j case}, and with it, the proof of Theorem \ref{Theorem: Mixed norm estimate for first layer} is also concluded. \qed

\section*{Acknowledgments}
The first and third author would like to acknowledge the support provided by the FIST grant (SR/FST/MS-II/2019/51) at IISER Kolkata provided by the Government of India. The second author gratefully acknowledges the financial support provided by the NBHM Post-doctoral fellowship, DAE, Government of India. The third author would like to acknowledge the support of the Prime Minister's Research Fellows (PMRF) supported by Ministry of Education, Government of India. The authors are indebted to the anonymous referee for insightful and valuable suggestions which helped to remarkably improve the paper.


\providecommand{\bysame}{\leavevmode\hbox to3em{\hrulefill}\thinspace}
\providecommand{\MR}{\relax\ifhmode\unskip\space\fi MR }
\providecommand{\MRhref}[2]{%
  \href{http://www.ams.org/mathscinet-getitem?mr=#1}{#2}
}
\providecommand{\href}[2]{#2}

\end{document}